\documentclass{article}[11pt]
\usepackage{titling}
\usepackage{ amssymb }
\usepackage{amsfonts}
\usepackage{amsthm}
\usepackage{amsmath}
\usepackage{bbm}
\usepackage{setspace}
\usepackage[normalem]{ulem}

\usepackage[utf8]{inputenc}
\def\badmu{0.5r^{-2}\varepsilon^{3/2} p^{3/2}(1-0.25\varepsilon)((1-0.25\varepsilon)\varepsilon^3 p^3-\lambda)}
\def\badprob{5n^2 e^{-\mu D}}
\def\badprobR{(2n+1)e^{-0.01\varepsilon p^{3}r^{-2}D}}

\def\acknowledgements{1}
\def\showauthornotes{1}

\def\showkeys{0}
\def\showdraftbox{0}
\def\showcolorlinks{1}
\def\usemicrotype{1}
\def\showfixme{0}

\usepackage{csquotes}

\usepackage[
backend=biber,
maxcitenames=50,
maxbibnames=99,
style=alphabetic,
sorting=ynt
]{biblatex}
\title{New High Dimensional Expanders from Covers}
\author{Yotam Dikstein\thanks{Weizmann Institute of Science, ISRAEL. email: yotam.dikstein@weizmann.ac.il. Supported by Irit Dinur's ERC grant 772839.}} 
\date{\today}

\providecommand{\RR}{\mathbb{R}}
\providecommand{\NN}{\mathbb{N}}

\usepackage[l2tabu, orthodox]{nag}

\usepackage{xspace,enumerate}

\usepackage[dvipsnames]{xcolor}

\usepackage[T1]{fontenc}
\usepackage[full]{textcomp}

\usepackage[american]{babel}

\usepackage{mathtools}

\usepackage{amsthm}
\usepackage{amsmath}

\newtheorem{theorem}{Theorem}[section]
\newtheorem*{theorem*}{Theorem}

\newtheorem{proposition}[theorem]{Proposition}
\newtheorem*{proposition*}{Proposition}
\newtheorem{lemma}[theorem]{Lemma}
\newtheorem*{lemma*}{Lemma}
\newtheorem{corollary}[theorem]{Corollary}
\newtheorem*{corollary*}{Corollary}
\newtheorem*{conjecture*}{Conjecture}

\newtheorem*{fact*}{Fact}

\newtheorem*{hypothesis*}{Hypothesis}

\theoremstyle{definition}
\newtheorem{definition}[theorem]{Definition}
\newtheorem*{definition*}{Definition}

\newtheorem{example}[theorem]{Example}

\theoremstyle{remark}
\newtheorem{claim}[theorem]{Claim}
\newtheorem*{claim*}{Claim}
\newtheorem{remark}[theorem]{Remark}
\newtheorem*{remark*}{Remark}
\newtheorem{observation}[theorem]{Observation}
\newtheorem*{observation*}{Observation}

 \usepackage[letterpaper,
 top=1in,
 bottom=1in,
 left=1in,
 right=1in]{geometry}

\usepackage[varg]{pxfonts} 

\usepackage[sc,osf]{mathpazo}
\usepackage{lmodern}

\ifnum\showkeys=1
\usepackage[color]{showkeys}
\fi

\ifnum\showcolorlinks=1
\usepackage[
colorlinks=true,
urlcolor=blue,
linkcolor=blue,
citecolor=OliveGreen,
]{hyperref}
\fi

\ifnum\showcolorlinks=0
\usepackage[
colorlinks=false,
pdfborder={0 0 0}
]{hyperref}
\fi

\usepackage{prettyref}

\newcommand{\savehyperref}[2]{\texorpdfstring{\hyperref[#1]{#2}}{#2}}

\newrefformat{eq}{\savehyperref{#1}{\textup{(\ref*{#1})}}}
\newrefformat{lem}{\savehyperref{#1}{Lemma~\ref*{#1}}}
\newrefformat{def}{\savehyperref{#1}{Definition~\ref*{#1}}}
\newrefformat{thm}{\savehyperref{#1}{Theorem~\ref*{#1}}}
\newrefformat{cor}{\savehyperref{#1}{Corollary~\ref*{#1}}}
\newrefformat{cha}{\savehyperref{#1}{Chapter~\ref*{#1}}}
\newrefformat{sec}{\savehyperref{#1}{Section~\ref*{#1}}}
\newrefformat{app}{\savehyperref{#1}{Appendix~\ref*{#1}}}
\newrefformat{tab}{\savehyperref{#1}{Table~\ref*{#1}}}
\newrefformat{fig}{\savehyperref{#1}{Figure~\ref*{#1}}}
\newrefformat{hyp}{\savehyperref{#1}{Hypothesis~\ref*{#1}}}
\newrefformat{alg}{\savehyperref{#1}{Algorithm~\ref*{#1}}}
\newrefformat{rem}{\savehyperref{#1}{Remark~\ref*{#1}}}
\newrefformat{item}{\savehyperref{#1}{Item~\ref*{#1}}}
\newrefformat{step}{\savehyperref{#1}{step~\ref*{#1}}}
\newrefformat{conj}{\savehyperref{#1}{Conjecture~\ref*{#1}}}
\newrefformat{fact}{\savehyperref{#1}{Fact~\ref*{#1}}}
\newrefformat{prop}{\savehyperref{#1}{Proposition~\ref*{#1}}}
\newrefformat{prob}{\savehyperref{#1}{Problem~\ref*{#1}}}
\newrefformat{claim}{\savehyperref{#1}{Claim~\ref*{#1}}}
\newrefformat{relax}{\savehyperref{#1}{Relaxation~\ref*{#1}}}
\newrefformat{red}{\savehyperref{#1}{Reduction~\ref*{#1}}}
\newrefformat{part}{\savehyperref{#1}{Part~\ref*{#1}}}
\newrefformat{ex}{\savehyperref{#1}{Example~\ref*{#1}}}
\newrefformat{obs}{\savehyperref{#1}{Observation~\ref*{#1}}}

\newcommand{\Sref}[1]{\hyperref[#1]{\S\ref*{#1}}}

\usepackage{nicefrac}

\ifnum\usemicrotype=1
\usepackage{microtype}
\fi

\ifnum\showauthornotes=1
\newcommand{\Authornote}[2]{{\sffamily\small\color{red}{[#1: #2]}}}
\newcommand{\Authornotecolored}[3]{{\sffamily\small\color{#1}{[#2: #3]}}}
\newcommand{\Authorcomment}[2]{{\sffamily\small\color{gray}{[#1: #2]}}}
\newcommand{\Authorstartcomment}[1]{\sffamily\small\color{gray}[#1: }

\newcommand{\Authorfnote}[2]{\footnote{\color{red}{#1: #2}}}
\newcommand{\Authorfixme}[1]{\Authornote{#1}{\textbf{??}}}
\newcommand{\Authormarginmark}[1]{\marginpar{\textcolor{red}{\fbox{\Large #1:!}}}}
\else
\newcommand{\Authornote}[2]{}
\newcommand{\Authornotecolored}[3]{}
\newcommand{\Authorcomment}[2]{}
\newcommand{\Authorstartcomment}[1]{}

\newcommand{\Authorfnote}[2]{}
\newcommand{\Authorfixme}[1]{}
\newcommand{\Authormarginmark}[1]{}
\fi

\ifnum\showfixme=0

\fi

\usepackage{boxedminipage}

\newcommand{\brac}[1]{[#1]}
\newcommand{\Brac}[1]{\left[#1\right]}

\newcommand{\abs}[1]{\lvert#1\rvert}

\newcommand\sett[2]{\left\{ #1 \left| \; \vphantom{#1 #2} \right. #2  \right\}}
\newcommand{\set}[1]{\{#1\}}

\newcommand{\norm}[1]{\lVert#1\rVert}

\newcommand{\iprod}[1]{\langle#1\rangle}

\newcommand{\Esymb}{\mathbb{E}}
\newcommand{\Psymb}{\mathbb{P}}

\DeclareMathOperator*{\E}{\Esymb}

\DeclareMathOperator*{\ProbOp}{\Psymb}

\renewcommand{\Pr}{\ProbOp}
\def\one{{\mathbf{1}}}

\newcommand{\prob}[1]{\Pr \left[ {#1} \right] }
\newcommand{\Prob}[2][]{\Pr_{{#1}}\left[#2\right]} 
\newcommand{\cProb}[3]{\Pr_{{#1}}\left[ #2 \left| \; \vphantom{#2 #3} \right. #3  \right]} 

\newcommand{\ex}[1]{\E\brac{#1}}
\newcommand{\Ex}[2][]{\E_{{#1}}\Brac{#2}}

\newcommand{\ve}{\;\hbox{and}\;}

 \usepackage{dsfont}
\usepackage{mathrsfs}

\newcommand{\textparen}[1]{\text{(#1)}}

\ifx\because\undefined
\newcommand{\because}[1]{\textparen{because #1}}
\else
\renewcommand{\because}[1]{\textparen{because #1}}
\fi

\newcommand\bdot\bullet

\DeclareMathOperator{\poly}{poly}

\renewcommand{\leq}{\leqslant}

\renewcommand{\geq}{\geqslant}

\ifnum\showdraftbox=1

\else

\fi

\let\epsilon=\varepsilon

\numberwithin{equation}{section}

\newcommand{\MYstore}[2]{%
  \global\expandafter \def \csname MYMEMORY #1 \endcsname{#2}%
}

\newcommand{\MYload}[1]{%
  \csname MYMEMORY #1 \endcsname%
}

\newcommand{\MYnewlabel}[1]{%
  \newcommand\MYcurrentlabel{#1}%
  \MYoldlabel{#1}%
}

\newcommand{\MYdummylabel}[1]{}

\newcommand{\torestate}[1]{%
  \let\MYoldlabel\label%
  \let\label\MYnewlabel%
  #1%
  \MYstore{\MYcurrentlabel}{#1}%
  \let\label\MYoldlabel%
}

\newcommand{\restatetheorem}[1]{%
  \let\MYoldlabel\label
  \let\label\MYdummylabel
  \begin{theorem*}[Restatement of \prettyref{#1}]
    \MYload{#1}
  \end{theorem*}
  \let\label\MYoldlabel
}

\newcommand{\restatelemma}[1]{%
  \let\MYoldlabel\label
  \let\label\MYdummylabel
  \begin{lemma*}[Restatement of \prettyref{#1}]
    \MYload{#1}
  \end{lemma*}
  \let\label\MYoldlabel
}

\newcommand{\restateprop}[1]{%
  \let\MYoldlabel\label
  \let\label\MYdummylabel
  \begin{proposition*}[Restatement of \prettyref{#1}]
    \MYload{#1}
  \end{proposition*}
  \let\label\MYoldlabel
}

\newcommand{\restateclaim}[1]{%
  \let\MYoldlabel\label
  \let\label\MYdummylabel
  \begin{claim*}[Restatement of \prettyref{#1}]
    \MYload{#1}
  \end{claim*}
  \let\label\MYoldlabel
}

\newcommand{\restatecorollary}[1]{%
  \let\MYoldlabel\label
  \let\label\MYdummylabel
  \begin{corollary*}[Restatement of \prettyref{#1}]
    \MYload{#1}
  \end{corollary*}
  \let\label\MYoldlabel
}

\newcommand{\restatefact}[1]{%
  \let\MYoldlabel\label
  \let\label\MYdummylabel
  \begin{fact*}[Restatement of \prettyref{#1}]
    \MYload{#1}
  \end{fact*}
  \let\label\MYoldlabel
}

\newcommand{\restatedefinition}[1]{
\let\MYoldlabel\label 
\let\label\MYdummylabel 
\begin{definition*}[Restatement of \prettyref{#1}] 
    \MYload{#1} 
\end{definition*} 
\let\label\MYoldlabel 
} 

\newcommand{\restate}[1]{%
  \let\MYoldlabel\label
  \let\label\MYdummylabel
  \MYload{#1}
  \let\label\MYoldlabel
}

\let\origparagraph\paragraph
\renewcommand{\paragraph}[1]{\origparagraph{#1.}}

\allowdisplaybreaks

\newcommand{\dunion}{\mathbin{\mathaccent\cdot\cup}}

\let\pref=\prettyref

\newcommand{\dir}[1]{\overset{\to}{#1}}

\begin{document}
\maketitle

\begin{abstract}
We present a new construction of high dimensional expanders based on covering spaces of simplicial complexes. High dimensional expanders (HDXs) are hypergraph analogues of expander graphs. They have many uses in theoretical computer science, but unfortunately only few constructions are known which have arbitrarily small local spectral expansion.

We give a randomized algorithm that takes as input a high dimensional expander \(X\) (satisfying some mild assumptions). It outputs a sub-complex \(Y \subseteq X\) that is a high dimensional expander and has infinitely many simplicial covers. These covers form new families of bounded-degree high dimensional expanders. The sub-complex \(Y\) inherits \(X\)'s underlying graph and its links are sparsifications of the links of \(X\). When the size of the links of \(X\) is \(O(\log |X|)\), this algorithm can be made deterministic.

Our algorithm is based on the groups and generating sets discovered by Lubotzky, Samuels and Vishne (2005), that were used to construct the first discovered high dimensional expanders. We show these groups give rise to many more ``randomized'' high dimensional expanders. 

In addition, our techniques also give a random sparsification algorithm for high dimensional expanders, that maintains its local spectral properties. This may be of independent interest.
\end{abstract}

\clearpage
\section{Introduction}
Expanders are graphs that are highly connected. It is a known fact that there exist families of expander graphs that are bounded-degree. These graphs play a key role in theoretical computer science, combinatorics and many other areas in mathematics \cite{HooryLW2006}.
High dimensional expanders (HDXs) are a hypergraph analogue of expander graphs. Loosely speaking, high dimensional expanders have the property that neighborhoods of faces are themselves expanding graphs. Lubotzky surveys the different definitions of high dimensional expanders \cite{Lubotzky2017}.

High dimensional expanders are promising objects in theoretical computer science. They have been used for efficiently solving CSPs \cite{AlevFT2019}, for constructing agreement testers \cite{DinurK2017,DiksteinD2019} and for sampling and counting matroids and other combinatorial objects (e.g. \cite{AnariLGV2019}, \cite{ChenLV2021}).

There are only two constructions of high dimensional expanders that achieve arbitrarily good expansion,  \cite{LubotzkySV2005a} and \cite{KaufmanO2018}. Both constructions rely on group theory. This is contrary to expander graphs, where in addition to group-theoretic constructions, elementary combinatorial constructions and random constructions are known \cite{HooryLW2006}.

The goal of this work is to construct new high dimensional expanders by using random simplicial covers.
Graph covers (or lifts) are a key ingredient in many random expander constructions. A graph \(G'\) covers a graph \(G\) if there is a surjective graph homomorphism \(\psi:G' \to G\) so that for every vertex \(v \in G'\), \(deg(v)=deg(\psi(v))\). There is a simple algorithm (described below) that given a graph \(G\), outputs a random cover \(G'\)  of \(G\). In many cases the covers will also be expander graphs, and indeed this is exploited to construct infinite families of expanders (see e.g. \cite{BiluL2006} and \cite{MarkusSS2015}). In this work, we manage to extend this technique to high dimensional expanders.

\subsection{High dimensional expanders}
A pure \(d\)-dimensional simplicial complex \(X\) is a set system (or hypergraph) consisting of an arbitrary collection of sets of size \(d + 1\) together with all their subsets. The sets of size \(i+1\) in \(X\) are denoted by \(X(i)\), and in particular, the vertices of \(X\) are denoted by \(X(0)\). The \(1\)-skeleton of \(X\) is the graph whose vertices are \(X(0)\) and edges are \(X(1)\). The sets in a simplicial complex are called faces. We say that \(X\) is \emph{connected} if its \(1\)-skeleton is connected.

Let \(s \in X(i)\). The link of \(s\) is a simplicial complex denoted by \(X_s\) whose sets are all \(t \in X\) so that \(s \cap t = \emptyset\) and \(s \cup t \in X\). For \(\lambda \geq 0\), a \(\lambda\)-two-sided high dimensional expander is a simplicial complex so that its \(1\)-skeleton, and all the \(1\)-skeletons of its links are \(\lambda\)-two-sided spectral expanders.

This definition is due to \cite{DinurK2017}; it concentrates on spectral properties. High dimensional expanders by this definition are useful in many applications as discussed above. However, there are other non-equivalent definitions for high dimensional expansion. We recommend \cite{Lubotzky2017} for a comparison between the main definitions that interest the community.

\subsection{Graph covers and simplicial covers}Let \(X,Y\) be pure \(d\)-dimensional simplicial complexes. A simplicial homomorphism is a function \(\psi:X(0) \to Y(0)\) so that every \(s \in X(d)\) is mapped to \(\psi(s)\in Y(d)\). A homomorphism is a \emph{cover} if it is surjective, and locally it behaves like an isomorphism. That is, for every face \(t \in Y(i)\) and preimage \(s \in X(i)\) so that \(\psi(s)=t\), the restriction of \(\psi\) to \(X_s\) is an isomorphism between \(X_s\) and \(Y_t\). For graphs, this definition coincides with the previous definition discussed above.

Given a graph on \(n\) vertices \(G = (V = \set{v_1,v_2,...,v_n},E)\), we can construct a cover \(G'=(V',E')\) of \(G\) as follows. Let \(\Gamma\) be any finite group and let \(f:E \to \Gamma\) be a labeling. The vertices of \(G'\) are set to be \(V' = V \times \Gamma\). For every \(i<j\), the vertices \((v_i,g), (v_j,h)\) are adjacent in \(G'\) if \(v_i v_j \in E\) and \(h=g \cdot f(ij)\). The covering map is the map that projects \((i,g)\) to its left coordinate \(i\). In fact, every graph cover corresponds to some group and some edge labeling \(f:E \to \Gamma\) \cite{Surowski1984}.

This method has proven to be useful in constructing new expander graphs from initial expanders. Amit and Linial initiated the study of random graph covers \cite{AmitL2002}. Bilu and Linial showed that if \(G\) is a bounded-degree expander, \(\Gamma = \mathbb{F}_2\) and \(f\) is chosen at random, then \(G'\) is an expander with positive probability \cite{BiluL2006}. A breakthrough paper by Marcus, Spielman and Srivastava \cite{MarkusSS2015} used graph covers to create optimal (Ramanujan) bipartite expander graphs. It is tempting to try and generalize this simple construction to high dimensional expanders.

Unfortunately, in simplicial complexes this method does not work as is. Given a \(d\)-dimensional simplicial complex \(X\) and a labeling of the \(1\)-skeleton \(f:X(1) \to \Gamma\), the resulting covering graph \(G'\) is not necessarily a (\(1\)-skeleton of a) \(d\)-dimensional simplicial complex at all. It turns out that connected covers of a simplicial complex correspond to subgroups of its fundamental group (as a topological space) \cite{Surowski1984}. This group may be trivial so complexes may have no non-trivial connected covers.
%

\subsection{Our contribution}
Not all is lost. In this paper we show that even if a high dimensional expander \(X\) has no covers, with the help of certain groups \(\Gamma\) we can find a non-trivial sub-complex \(Y \subseteq X\) that is a high dimensional expander \emph{and has infinitely many covers}. In more detail, we start with:
\begin{enumerate}
    \item\label{item:first} An initial high dimensional expander \(X\). The links of \(X\) need to be (almost) regular and relatively dense, but \(X\) itself may have bounded degree.
    \item\label{item:second} An infinite group \(\Gamma\) that has infinitely many normal subgroups of finite index, together with a generating set \(S \subseteq \Gamma\). Links of the clique complex of the Cayley graph \(Cay(\Gamma,S)\) should be themselves expanding graphs. \footnote{A clique complex of a graph is a simplicial complex whose vertices and edges are those of the graph \(Cay(\Gamma,S)\). The high-dimensional faces of the complex are all cliques in the graph.}
\end{enumerate}

\begin{theorem}[Informal, for the formal theorem see \pref{thm:main-formal}] \label{thm:main}
There is a randomized algorithm that takes as input a high-dimensional expander \(X\) satisfying \pref{item:first}, and a group and generating set \((\Gamma,S)\) satisfying \pref{item:second}. The algorithm outputs a subcomplex \(Y \subseteq X\) and a labeling \(f:X(1) \to S\) so that:
\begin{enumerate}
    \item \(Y\) has infinitely many covers \(\set{Y_n}_{n=1}^\infty\) which can be constructed in polynomial time using \(f\).
    \item \(Y\) is a high dimensional expander.
    \item \(Y\) contains a constant fraction of the faces of \(X\), proportional to \(\frac{1}{poly(|S|)}\). 
\end{enumerate}
\end{theorem}
We note that the cover \(\set{Y_n}_{n=1}^\infty\) are the new family of high dimensional expanders we seek. Their vertices are \(Y(0) \times (\Gamma / \Gamma_n)\) where \(\set{\Gamma_n}_{n=1}^\infty\) are normal subgroups of \(\Gamma\) of finite index. Their links are isomorphic to links of faces in the original subcomplex \(Y\). Hence links of these covers are sparsifications of links in the original complex \(X\).

Groups and generating sets \((\Gamma,S)\) as required in the input appeared in \cite{LubotzkySV2005b,LubotzkySV2005a}, where they were used to construct the first known high dimensional expanders. The resulting simplicial complexes we construct are however, are different and depend on the input \(X\) we give the algorithm.

Interestingly, we show that this new simplicial complex \(Y\) inherits properties both from \(X\) and from \(Cay(\Gamma,S)\). Globally, the \(1\)-skeleton of \(Y\) is identical to the \(1\)-skeleton of \(X\) (i.e it has the same vertex set and edge set). For every \(\sigma \in Y\), the link \(Y_\sigma \subseteq X_\sigma\) is a sparsification that has a constant fraction of the faces in \(X_\sigma\). The algorithm sparsifies the higher-dimensional faces so that locally its links are homomorphic to links of faces in the clique complex of \(Cay(\Gamma,S)\).

There are many starting points \(X\) that satisfies the requirements in \pref{thm:main}. The simplest is the complete complex. We could also use a random model to sample our initial simplicial complex, such as the model suggested by \cite{LinialM2006}: its vertices are \(\set{1,2,...,m}\), every \(d\)-face is sampled independently with some probability \(p \in (0,1)\). Although we don't usually think about the complete complex or the model in \cite{LinialM2006} as ``bounded-degree'', recall that \(X\) is just a single complex (and thus always has bounded degree). Albeit, we can initialize our construction by taking \(X\) out of a family of bounded-degree of high dimensional expanders as long as they are regular. A work by Friedgut and Iluz shows how to construct strongly regular high dimensional expanders from existing ones \cite{FriedgutI2020}, which we can use.

The algorithm in \pref{thm:main} can also be derandomized when \(X\) has logarithmic sized links.
\medskip

Though our main result deals with covers, our random sparsification technique has another application of independent interest. We show that we can sparsify a high dimensional expander \(X\), so that it will be homomorphic to some fixed smaller high dimensional expander \(C_0\), while maintaining its high dimensional expansion.
\begin{theorem}[Informal, for the formal theorem see \pref{thm:second-construction}] \label{thm:second-construction-informal}
There is a randomized algorithm that takes as input a simplicial complex \(C_0\), and simplicial complex \(X\) that satisfies \pref{item:first}, and outputs a sub-complex \(Y \subseteq X\) so that:
\begin{enumerate}
    \item \(Y\) is homomorphic to \(C_0\).
    \item \(Y\) has a constant fraction of the faces of \(X\) (the constant depends on \(C_0\)).
    \item If \(X\) and \(C_0\) are high dimensional expanders then \(Y\) is also a high dimensional expander.
\end{enumerate}
\end{theorem}
Note that \(Y \subseteq X\) could be viewed also as a sparsifying algorithm for \(X\), that produces a sparser simplicial complex that still has high dimensional expansion properties. We also show that when \(X\) has small enough links, then we can derandomize the algorithm in \pref{thm:second-construction-informal}.

\subsection{Constructing covers in higher dimension}
Let us understand how to generalize the graph cover construction to higher dimensional simplicial complexes.

Let \(X\) be a \(d\)-dimensional high dimensional expander where \(d \geq 2\) with vertex set \(X(0)=\set{v_1,v_2,...,v_n}\) (we order of the vertices arbitrarily). Let \(f:X(1) \to \Gamma\) be a labeling to some group \(\Gamma\). Let \(G'\) be the \emph{graph cover} of the \(1\)-skeleton of \(X\), generated by \(f\). In order to extend \(G'\) to a \emph{simplicial cover} of \(f\) to have the property that for every triangle \(v_i v_j v_k \in X\) where \(i<j<k\)
\begin{equation}\label{eq:triangle-property}
    f(v_i v_j)f(v_j v_k)=f(v_i v_k).
\end{equation}
Labelings \(f\) with this property are called \(\Gamma\)-cosystols. If \(f\) is a \(\Gamma\)-cosystol, then whenever \((v_i,g) \sim (v_j,g f(v_i v_j))\) and \((v_i,g) \sim (v_k, g f(v_i v_k))\), it also holds that 
\[(v_j, g f(v_i v_j)) \sim (v_k,g f(v_i v_j) f(v_j v_k) = g f(v_i v_k) ),\]
I.e. all edges of the triangle \(\set{(v_i,g),(v_j,g f(v_i v_j)),(v_k,g f(v_i v_k))}\) are in \(G'\). We can then extend \(G'\) to a simplicial complex \(X'\) with
\[X'(2) = \sett{ \set{(v_i,g),(v_j,g f(v_i v_j)),(v_k,g f(v_i v_k))}}{v_i v_j v_k 
\in X(2), i<j<k<, g \in \Gamma}\]
and similarly for \(\ell > 2\), 
\[X'(\ell) = \sett{\set{(v_{i_0},g), (v_{i_1}, g f(v_{i_0}v_{i_1})),...,(v_{i_\ell}, g f(v_{i_0} v_{i_\ell}))}}{i_0 i_1 ... i_\ell \in X(\ell), i_0 < i_1 <... < i_\ell, g \in \Gamma }).\]
%
%
\subsection{Overview of our construction}
The cover constructed from a labeling \(f:X(1) \to \Gamma\) may be trivial, that is, that \(X'\) may consist of \(|\Gamma|\) disjoint copies of \(X\). Recall that connected covers are determined by the fundamental group of \(X\), and it could be the case that all covers \(X'\) of \(X\) are trivial.

To overcome this in our construction, we ``go the other way around''. We sample a labeling \(f:X(1)\to \Gamma\) by sampling the label of each edge independently. We then wish to remove from \(X\) all faces that contain triangles \(v_i v_j v_k \in X(2)\) where \(i<j<k\) and \(f(v_i v_j)f(v_j v_k)\ne f(v_i v_k)\). If we do so, we get a sub-complex \(Y \subseteq X\) so that \(f\) is one of its \(\Gamma\)-cosystols.

The main problem with this naive triangle removal, is that with high probability some links of faces in \(Y\) will not be expanders anymore. It could be, for example, that all triangles adjacent to some vertex are removed (and its link becomes disconnected). Denote by \(B_\sigma\) the event that the link of \(\sigma\) is not an expander graph.

Fortunately, Lovasz's Local Lemma \cite{ErdosL1974} guarantees us that if the probability of each \(B_\sigma\) is small, and that every event \(B_\sigma\) is independent of all but a few of the \(B_{\sigma'}\), then the event where no \(B_\sigma\) occurs has positive probability (i.e. \(Y\) is a high dimensional expander with positive probability). The heart of our analysis is to properly define these ``local bad events'', and to show that we can indeed apply Lovasz's Local Lemma to them. We use similar techniques to also promise that the \(f\)-cover of \(Y\) is indeed connected.

Moser and Tardos proved that there is a procedure to turn the existential proof based on Lovasz's Local Lemma to a randomized algorithm that finds \(Y\) in polynomial time \cite{MoserT2010}. When the links of \(X\) are not too small, we can derandomize this algorithm using the work of \cite{ChandrasekaranGH2013}. 

As \(Y\) is obtained by removing faces from \(X\), it is not surprising that \(1\)-skeletons of links of faces in \(Y\) are obtained by removing vertices and edges from links of faces in \(X\). To show that these links still expand with high probability, we analyze random sparsifications of expander graphs with logarithmic degree, and show that with high enough probability a graph remains an expander after such a sparsification. To do so we use the Inverse Expander Mixing Lemma proven by Bilu and Linial \cite{BiluL2006}. This analysis is quite general and may be of independent interest. 

\subsection{Related works}
Properties of high dimensional expanders were first used in the work of Garland \cite{Garland1973}, even before they were explicitly defined. Local spectral properties of links were used there to show that the real cohomology of a simplicial complex vanishes. Later on, a strong notion of coboundary expansion was discovered by Gromov \cite{Gromov2010}, and also by Linial, Meshulam and Wallach \cite{LinialM2006}, \cite{MeshulamW09}. The definition we use of high dimensional expanders is due to Dinur and Kaufman \cite{DinurK2017}, who also defined the one-sided version of high-dimensional expansion (before their work, a weaker notion of expansion in the links was used in e.g. \cite{EvraK2016,KaufmanKL2014}). The first construction of bounded-degree high dimensional expanders was done before the definition we use today, by Lubotzky, Samuels and Vishne \cite{LubotzkySV2005b}. A second algebraic construction was discovered by Kaufman and Oppenheim \cite{KaufmanO2018,KaufmanO2021}. Recently, \cite{ODonnellP2022} construct new high dimensional expanders using the method of \cite{KaufmanO2018} with Chevalley groups. Friedgut and Iluz modified existing constructions for creating hyper-regular high dimensional expanders \cite{FriedgutI2020}. There have been combinatorial constructions of high dimensional expanders by Chapman, Linial and Peled \cite{ChapmanLP2020}, Liu, Mohanty and Yang \cite{LiuMY2020} and Golowich \cite{Golowich2021}. However, the combinatorial constructions currently produce \(\lambda\)-HDXs for \(\lambda \geq \frac{1}{2}\).

As mentioned above, there is a long line of works on random graph covers and their properties e.g. \cite{AmitL2002, AmitLMR2001, Friedman2003, BiluL2006, Puder2015, MarkusSS2015}. This work lead to the breakthrough construction of bipartite Ramanujan graphs of all degrees by Marcus, Spielman and Srivastava. The works that study expansion usually focused on analysis of the eigenvalues of the covering graph. In higher dimension however, every connected cover is also a high dimensional expander. This fact is a corollary of a work of Oppenheim \cite{Oppenheim2018}, which we discuss more precisely in \pref{sec:preliminaries}. However, to analyze the links in our construction, we use the techniques developed by Bilu and Linial \cite{BiluL2006} towards analyzing random graph covers.

Dinur and Meshulam also studied on simplicial covers from a TCS perspective \cite{DinurM2019}. They study a topological property test on a class of simplicial complexes called \emph{cosystolic expanders}. They show that if \(X\) is such a complex, then every surjective mapping \(f:Y \to X\) that satisfies almost all the local conditions of being a cover is close to a genuine cover \(g:Y \to X\).

Random sparsification of graphs and their adjacency matrices is also a well studied problem, and there are known randomized and deterministic algorithms to sparsify graphs while keeping their spectral properties, e.g. \cite{SpielmanS08}, \cite{SpielmanT2011}, \cite{LeeS2017}, \cite{JambulapatiS18} and \cite{LaiXX2020}. Furedi and Kolmos initiated the study of random sparsified matrices \cite{FurediK1981} (see also the work of Vu \cite{Vu2007}). A celebrated result by Spielman and Teng, studied sparsification of spectral expanders by randomly removing edges \cite{SpielmanT2011}. However, the probability of removing an edge in \cite{SpielmanT2011} varies according to the degrees of its vertices, and we could not use its analysis as is. Hence, we prove a graph sparsification result that is tailored for our needs.

\subsection{Open questions}
The construction we give in this paper is for two sided high dimensional expansion. Is there also an extension of our algorithm or our analysis that works on \emph{one sided} high dimensional expanders \(X\)?

All constructions of bounded-degree high dimensional expanders with \(\lambda < \frac{1}{2}\) rely on group theory (either deterministic or randomized). Finding a combinatorical construction that does not rely on a group for arbitrarily small \(\lambda>0\) or even \(\lambda =0.49\) is an open problem. It is interesting to note that the theorem of Oppenheim mentioned above \cite{Oppenheim2018} only applies for \(\lambda <\frac{1}{2}\). This seems to show that \(\lambda\)-high dimensional expanders for \(\lambda <\frac{1}{2}\) must have some non-trivial structure.

Our construction has properties inherited by both the complex \(X\), and the clique complex of \(Cay(\Gamma,S)\). There are many other interesting properties that could be of interest in our construction. For example, one can repeat this construction with groups and generating sets so that the links of the clique complex of \(Cay(\Gamma,S)\) are \emph{small set expanders}, but not expanders themselves. Will the resulting complexes be small set expanders (either locally or globally)? Another potentially interesting property is coboundary and cosystolic expansion. Will the resulting complexes be cosystolic expanders?

Our construction relies on infinite groups \(\Gamma\) that are residually finite and have expanding links. Besides the groups in \cite{LubotzkySV2005a}, we are not aware of groups that have these properties \footnote{Not including trivial constructions such as \((\Gamma \times \Gamma, S \times S)\).}. Finding new examples of such groups is interesting. The groups that were used in \cite{KaufmanO2018,KaufmanO2021} are possible candidates, as they already showed potential for construction of high dimensional expanders. 
Note that \cite{Zuk2003} showed that any such \(\Gamma\) has Kazhdan's property-\(T\). This already constrains the types of groups we can consider. \cite{Gromov2003} defined a random model of groups that seems to have expanding links, but it is not clear whether they are residually finite or not. It will be interesting to understand if this model, or some variation to it, could produce new examples of groups useful for our construction.

Another open question that may have algorithmic applications, is embedding simplicial complexes in high dimensional expanders. That is, given a simplicial complex \(X\), is it possible to embed it in a high dimensional-expander \(X' \supseteq X\), so that \(|X'|/|X|\) is constant. For graphs this is easy, since given \(G = (V,E)\), we can find an expander graph with the same vertex set \(H = (V,E')\) and embed \(G\) in \(G' = (V,E \cup E')\). This simple embedding in graphs is useful, for example, in the PCP-theorem proof by \cite{Dinur2005}. Understanding when we also have such an embedding in simplicial complexes could potentially lead to new algorithms for solving constraint satisfaction problems.

Finally, \cite{DiksteinDFH2018} generalized the notion of high dimensional-expanders to \emph{expanding posets}. Another open direction of research is whether one could use topological covers to construct new expanding posets.

\if\acknowledgements1
\subsection{Acknowledgments}
I would like to express my deepest gratitude to Irit Dinur for her invaluable comments and feedback that greatly improved this paper. I would also like to thank Gil Melnik and Yael Hitron for their feedback on preliminary versions of this paper.
\fi

\subsection{Organization of this paper}
 \pref{sec:preliminaries} contains the preliminaries necessary for this paper. We formally state and prove our main theorem about simplicial covers in \pref{sec:main}, while deferring some technical claims on sparsification of graphs to \pref{sec:expander-sparsification}. In \pref{sec:ramifications} we also show that for every fixed simplicial complex \(C\), there is an algorithm that receives a simplicial complex \(X\) and outputs a sub-complex \(Y \subseteq X\) that is homomorphic to \(C\). In \pref{sec:derandomization} we show that when the input complex \(X\) has constant degree, we obtain deterministic algorithms for constructing the covering families and for sparsification.

\section{Preliminaries} \label{sec:preliminaries}
\subsection{Classical probability}

\begin{theorem}[Chernoff's inequality] \label{thm:chernoff}
Let \(X_1,X_2,...,X_n\) be independent random variables. Denote \(\ex{\sum_{j=1}^n X_j} = \mu\). Then for any \(\varepsilon \in (0,1)\) it holds that
    \[ \prob{\abs{\sum_{j=1}^n X_n - \mu}\geq \varepsilon \mu} < 2e^{-0.3\varepsilon^2 \mu}.\]
\end{theorem}

\begin{theorem}[Hoeffding's inequality] \label{thm:hoeffding}
Let \(X_1,X_2,...,X_n\) be independent random variables. Let \([a,b] \subseteq \RR\) so that for all \(j=1,2,...,n\) it holds that \(X_j \in [a,b]\). Denote \(\ex{\sum_{j=1}^n X_j} = \mu\). Then for any \(\varepsilon > 0\) it holds that
\[\prob{\abs{\sum_{j=1}^n X_n - \mu}\geq \varepsilon \mu} < 2e^{\frac{-2\varepsilon^2 \mu^2}{\sum_{j=1}^n (b_i-a_i)^2}}.\]
\end{theorem}

\begin{theorem}[Lovazs Local Lemma, \cite{ErdosL1974}, algorithmic version by \cite{MoserT2010}] \label{thm:l3}
Let \(X_1,...,X_d\) be independent random variables. Let \(E_1,E_2,...,E_n\) be events determined by these random variables. For every \(E_i\) we denote by \(A_i = \sett{j}{E_i,E_j \text{ are not independent}}\). Assume that there are \(\alpha_i \in [0,1)\) so that for every event
\[ \prob{E_i} \leq \alpha_i \prod_{j \in A_i}(1-\alpha_j).\] 
Then \(\prob{\neg E_1 \wedge \neg E_2 \wedge ...\wedge \neg E_n} > 0\). Moreover, there exists a randomized algorithm that finds assignments to \(X_1,X_2,...,X_d\) so that all the events \(E_1,E_2,...,E_n\) don't occur. The algorithm runs in an expected time of 
\(\sum_{i=1}^n \frac{|A_i|\alpha_i}{1-\alpha_i}\).
\end{theorem}
Note that when all \(|A_i| < R\), if one takes \(\alpha_i = \frac{1}{R}\) and shows that \(\prob{E_i} \leq \frac{1}{e(R+1)}\), then this theorem promises us there exists an assignment so that \(\prob{\neg E_1 \wedge \neg E_2 \wedge ...\wedge \neg E_n} > 0\). Moreover, it gives us a polynomial time algorithm to find such an assignment to \(X_1,X_2,...,X_d\).

\subsection{Expander graphs}
A graph \(G=(V,E)\) is a finite set \(V\), the vertices, and \(E \subseteq \binom{V}{2}\), the edges. Graphs in this paper are assumed to be with no isolated vertices. The graphs we consider have some probability distribution \(\nu :E \to [0,1]\), and we assume henceforth that \(\nu(e) \ne 0\) for every edge \(e \in E\). As a convention \(\nu(e') = 0\) for every \(e' \notin E\). Sometimes it is convenient to consider oriented edges \(\dir{E} = \sett{(u,v),(v,u)}{\set{u,v} \in E}\). For an oriented edge we define \(\nu(u,v) = \frac{1}{2} \nu(\set{u,v})\).

The probabilities on edges give rise to a probability measure on $V$, i.e. $\nu(v) = \frac{1}{2}\sum_{w \sim v} \prob{\set{v,w}} = \sum_{w \sim v} \nu(w,v)$. This probability measure defines an inner product on \(\ell_2(V)\) the space of real valued functions on the vertices \(f,g:V \to \RR\) by
    \[ \iprod{f,g} := \Ex[v \in V]{f(v)g(v)} ,\]
where the expectation is with respect to \(\nu\). Every graph induces a random walk on its vertices; the transition probability from $v$ to $u$ is
    \[\nu(u|v) := \frac{\nu(uv)}{\sum_{u' \sim v}\nu(u' v)}.\]
Denote by $A=A(G)$ the Markov operator associated with this random walk. We associate \(A\) with the matrix indexed by \(V \times V\) where \(A(v,u) = \nu(u | v)\). We call this operator the \emph{normalized adjacency operator} (in this paper we will refer to it as the adjacency operator). $A$ is an operator on real valued functions on the vertices, where
    \[ \forall v\in V \; Af(v) = \Ex[u \sim v]{f(u)} \left (= \sum_{u \sim v} \nu(u|v)f(u). \right )\]

It is well known that \(A\) is self adjoint with respect to the inner product defined above. Its eigenvalues are in the interval $[-1,1]$. We denote its eigenvalues by $\lambda_1 \geq \lambda_2 \geq ... \geq \lambda_n$ (with multiplicities). The largest eigenvalue is always $\lambda_1 = 1$, and it is obtained by the constant function. The second eigenvalue is strictly less than $1$ if and only if the graph is connected.

\begin{definition}[spectral expanders]\torestate{\label{def:spectral-expanders}
Let $G$ be a graph and let \(0 \leq \lambda < 1\). $G$ is a \emph{$\lambda$-one sided spectral expander}, if
        \[ \lambda_2 \leq \lambda.\]
$G$ is a \emph{$\lambda$-two sided spectral expander}, if
    \[ \max (\abs{\lambda_2},\abs{\lambda_n}) \leq \lambda.\]}
\end{definition}
In this paper, when saying a graph is a \(\lambda\)-spectral expander, we mean it is a \(\lambda\)-\emph{two sided} spectral expander.
\subsubsection{Bipartite Graphs and Bipartite Expanders}
    A bipartite graph is a graph where the vertex set can be partitioned to two independent sets $V = L \dunion R$, called sides. We sometimes denote such a bipartite graph by \(G=(L,R,E)\).

\paragraph{The Bipartite Adjacency Operator}
    In a bipartite graph, we view each side as a separate probability space, where for any $v \in L$ (resp. $R$), $\nu(v) = \sum_{w \sim v} \nu(vw)$. We can define the \emph{bipartite adjacency operator} as the operator  $B: \ell_2(L) \to \ell_2(R)$ by
    \[\forall f \in \ell_2(L), v \in R, \; Bf(v) = \Ex[w \sim v]{f(u)}\]
    where the expectation is taken with respect to the probability space $L$, conditioned on being adjacent to $v$.
There is a similar operator $B^*: \ell_2(R) \to \ell_2(L)$ as the bipartite operator for the opposite side. As the notation suggests, $B^*$ is adjoint to $B$ with respect to the inner products of $\ell_2(L), \ell_2(R)$.

    We denote by $\lambda(B)$ the spectral norm of $B$ when restricted to $\ell_2^0(L) = \set{\one}^\bot$, the orthogonal complement of the constant functions (according to the inner product the measure induces). Namely
    \[ \lambda(B) = \sup \sett{\iprod{Bf,g}}{\norm{g},\norm{f}=1, f \bot \one_L}.\]
    \begin{definition}[Bipartite Expander]\torestate{\label{def:bipartite-expander}
    Let $G$ be a bipartite graph, let $\lambda < 1$. We say $G$ is a \emph{$\lambda$-bipartite expander}, if
    $\lambda(B) \leq \lambda$.}
    \end{definition}
It is easy to show that a bipartite graph is a \(\lambda\)-bipartite expander if and only if it is a \(\lambda\)-one sided spectral expander. So we use these terms interchangeably on bipartite graphs.

\subsubsection{Expander Mixing Lemma}
A classical result in expander graphs is the \emph{expander mixing lemma}, that intuitively says that the weight of the edges between any two vertex sets $S,T \subset V$ is proportionate to the probabilities of $S,T$. We denote by \(E(S,T)\) the edge set between \(S\) and \(T\).
    \begin{lemma}[Expander Mixing Lemma]
    \label{lem:EML}
    Let $G = (V,E)$ be a $\lambda$-two sided spectral expanders. Then for any $S,T \subset V$
\begin{equation}\label{eq:EML-prelim}
\abs{\nu(E(S,T)) - \nu(S)\nu(T) } \leq \lambda \sqrt{\nu(S) \nu(T) (1-\nu(S))(1-\nu(T))} .    
\end{equation}   
$\qed$
\end{lemma}

Bipartite graphs have their own type of expander mixing lemma:
    \begin{lemma}[Bipartite Expander Mixing Lemma]
        \label{lem:bipartite-EML}
        Let $G = (L,R,E)$ be a $\lambda$-bipartite expander. Then for any $S \subset L, T \subset R$
\begin{equation}\label{eq:bipartite-EML-prelim}
\abs{\nu(E(S,T)) - \nu(S)\nu(T) } \leq \lambda \sqrt{\nu(S) \nu(T) (1-\nu(S))(1-\nu(T))} .
\end{equation}
        $\qed$
\end{lemma}

Bilu and Linial gave a converse to the expander mixing lemma:
\begin{theorem}[\cite{BiluL2006}] 
\torestate{ \label{thm:converse-bipartite-eml}
Let \(A\) be a normalized adjacency-operator of a bipartite graph \(H'=(A,B,E')\). Assume that for all \(S \subseteq A, T \subseteq B\) it holds that
\[ \abs{\nu(E(S,T)) - \nu(S)\nu(T) } \leq \alpha \sqrt{\nu(S) \nu(T)} .\]
Then \(\lambda(A) \leq 260\alpha(1+\log_2(3/\alpha)).\)
}
\end{theorem}
The Theorem in \cite{BiluL2006} assumes \(H'\) is regular. We show their proof extends to arbitrary weighted graphs in \pref{sec:bilu-linial}.

\subsection{Graph homomorphisms}
Let \(G=(V,E)\) and \(H=(V',E')\) be graphs. We say that \(f:V \to V'\) is a \emph{graph homomorphism} if for any edge \(\set{u,v} \in E\) it holds that \(\set{f(v),f(u)}\in E'\). If \(f\) is surjective we say that \(f\) an \(H\)-coloring of \(G\), and that \(G\) is \(H\)-colorable. In addition, if for every edge \({a,b} \in E'\) there is \(\set{u,v}\in E\) so that \(f(u)=a, f(v)=b\) then we say that \(G\) is \emph{non-degeneratly \(H\)-colorable} and that \(f\) is a \emph{non-degenerate \(H\)-coloring}. We depict this in \pref{fig:colorable-graph}.

Let \(G=(V,E),H=(V',E')\) be graphs. We say that \(G\) \emph{covers} \(H\) if there is a graph homomorphism \(f:V \to V'\) so that for every \(v \in V\), \(f\) restricted to a neighbourhood of \(v\) a bijection to the neighbourhood of \(f(v)\).

\begin{figure} 
    \centering
    \includegraphics[scale=0.8]{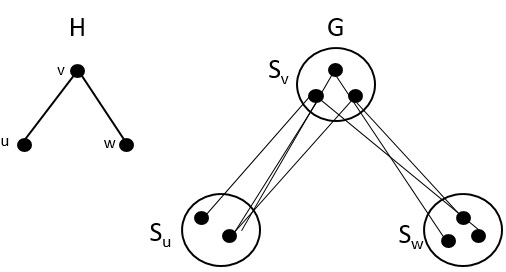}
    \caption{\(G\) is an \(H\)-colorable graph.}
    \label{fig:colorable-graph}
\end{figure}

When \((H,\nu_H),(G,\nu_G)\) have probability measures, then a non-degenerate \(H\)-coloring defines a \emph{new} measure \(\nu_f\) on \(G\) as follows. For \(uv \in \dir{E}\) so that \(f(u)=a,f(b)=b\),
\[\nu_f (uv) := \nu_H(ab) \frac{\nu_G(uv)}{\sum_{u'v': f(u'v')=ab} \nu_G(u'v')}. \]
In other words, we sample (oriented or non oriented) edges in \(E\) by the following process:
\begin{enumerate}
   \item Sample an edge \((a,b) \in \dir{E'}\).
   \item Sample an edge \((u,v)\in \dir{E}\) conditioned on its image being \(\set{a,b}\).
\end{enumerate}
We stress that \(\nu_f\) could be different than \(\nu_G\) the original measure on \(G\). For example it always holds that \(\nu_f(f^{-1}(a))=\nu_H(a)\), regardless of \(\nu_G(f^{-1}(a))\).

When \(G\) is an \(H\)-colorable graph, we can partition the vertices of \(G\) by \(V = \bigcup_{a \in V'} S_a\) where \(S_a = \sett{v \in V}{f(v)=a}\). In \pref{sec:expander-sparsification} we show that when \(H\) is an expander and the bipartite graphs between the parts in the partition of \(G\) are bipartite expanders, then \((G,\Pr_f)\) is also an expander:
\begin{claim} \torestate{\label{claim:graph-composition}
Let \(H = (V',E')\) be a \(\lambda\)-two sided (one-sided) spectral expander. Let \(G = (V,E)\) and let \(f:V \to V'\) be a non-degenerate \(H\)-coloring of \(G\). Assume that for every \(\set{a,b} \in E'\) the bipartite graph between \(S_a,S_b\) is \(\eta\)-bipartite spectral expander (with respect to the conditional weights \(\frac{\nu_G(uv)}{\sum_{u'v': f(u'v')=ab} \nu_G(u'v')}\) for \(uv \in S_a \times S_b\)). Then the weighted graph \((G,\nu_f)\) is a \(\max \set{\lambda,\eta}\)-two sided (one sided) spectral expander.}
\end{claim}

\subsection{Simplicial complexes and high dimensional expanders}
\label{sec:HDX}
We include here the basic definitions needed for our results. For a more comprehensive introduction to this topic we refer the reader to \cite{DinurK2017} and the references therein. 

A simplicial complex is a hypergraph that is downward closed with respect to containment. It is called $d$-dimensional if the largest hyperedge has size $d+1$. We refer to $X(\ell)$ as the hyperedges (also called faces) of size $\ell+1$. $X(0)$ are the vertices. For a face \(s\) and \(\ell \leq d\) we denote by \(deg_\ell(s) = \abs{\sett{t \in X(\ell)}{s \subseteq t}}\).

It will sometimes be useful to consider oriented faces. We denote by 
\[\dir{X}(\ell) = \sett{(v_0,v_1,...,v_\ell)}{\set{v_0,v_1,...,v_\ell} \in X(\ell)}.\]

We define a weighted simplicial complex. Suppose we have a $d$-dimensional simplicial complex $X$ and a probability distribution $\nu : X(d) \to [0,1]$. We consider the following sampling for choosing lower dimensional faces:
    \begin{enumerate}
    \item Choose some $d$-face $s_d =\set{v_0,v_1,...,v_d} \in X(d)$ with probability $\nu(s_d)$.
    \item Choose an ordering $\dir{s}_d = (v_0,v_1,...,v_d)$ uniformly at random.
\end{enumerate}
%
%
For any $(v_0,v_1,...,v_k) \in \dir{X}(k)$ we denote by
\[\prob{(v_0,v_1,...,v_k)} = \Prob[(u_0,u_1,...,u_d) \in \dir{X}(d)]{u_0=v_0,u_1=v_1,...,u_k=v_k},\]
the probability of sampling an ordered face so that \((v_0,v_1,...,v_k)\) is its prefix. For an unordered face \(s \in X(k)\), we denote by
\[ \prob{s} = \frac{1}{\binom{d+1}{k+1}}\sum_{s' \in X(d), s' \supseteq s} \nu(s') .\]
This is the probability of sampling \((v_0,v_1,...,v_d)\) so that \(s=\set{v_0,v_1,...,v_k}\) is the underlying set of the first \(k\)-vertices.

From here throughout the rest of the paper, when we refer to a simplicial complex $X$, we always assume that there is a probability measure on it constructed as above.

A link of a face in a simplicial complex, is a generalization of a neighbourhood of a vertex in a graph:

\begin{definition}[link of a face]\torestate{\label{def:link}
Let $s \in X(k)$ be some $k$-face. The \emph{link} of $s$ is a $d-(k+1)$-dimensional simplicial complex defined by:
\[X_s = \set{ t \backslash s : s \subseteq t \in X }.\]
The associated probability measure $\nu_{X_s}$, for the link of $s$ is defined to be the probability of sampling \(t\) in the process above given that we sampled \(s\). It is proportional equal to \(\nu_{X_s}(t) = \frac{\nu_X(t \cup s)}{\binom{|t|+|s|}{|s|}\nu(s)}\).}
\end{definition}
For an oriented face \((v_0,v_1,...,v_\ell) \in \dir{X}(\ell)\) we denote by \(X_{(v_0,v_1,...,v_\ell)} = X_{\set{v_0,v_1,...,v_\ell}}\) the link of its underlying set. Let \((u_0,u_1,...,u_\ell) \in \dir{X}(\ell)\), whose underlying set is \(s = \set{u_0,u_1,...,u_\ell}\). Let \((v_{\ell+1},v_{\ell+2},...,v_{\ell+k+2}) \in \dir{X}_s(k)\). Then it is easy to verify that
\[\Prob[\nu_{X_s}]{(v_{\ell+1},v_{\ell+2},...,v_{\ell+k+2})} = \frac{\Prob[X]{(u_0,u_1,...,u_\ell,v_{\ell+1},v_{\ell+2},...,v_{\ell+k+2})}}{\Prob[X]{(u_0,u_1,...,u_\ell)}}.\]

\begin{definition}[1-skeleton]\torestate{\label{def:underlying-graph}
    The \emph{\(1\)-skeleton} of a simplicial complex $X$ with some probability measure as define above, is the graph whose vertices are $X(0)$ and edges are $X(1)$, with (the restriction of) the probability measures of $X$ to the vertices and edges.}
\end{definition}
We are ready to define our notion of high dimensional expanders.
\begin{definition}[one-sided and two-sided high dimensional expander] \torestate{\label{def:prelim-link-expander}
    Let $0 \leq \lambda < 1$. A simplicial complex $X$ is a \emph{$\lambda$-two sided link expander} (or $\lambda$-two sided HDX) if for every $-1 \leq k \leq d-2$ and every $s \in X(k)$, the underlying graph of the link $X_s$ is a $\lambda$-two sided spectral expander.

Similarly, $X$ is a \emph{$\lambda$-one sided link expander} (or $\lambda$-one sided HDX) if for every $-1 \leq k \leq d-2$ and every $s \in X(k)$, the underlying graph of the link $X_s$ is a $\lambda$-one sided spectral expander.}
\end{definition}
When $X$ is a graph, this definition coincides with the definition of a spectral expander.

A theorem by \cite{Oppenheim2018} shows that expansion ``trickles down'', that is, expansion of links of \(d-2\)-faces imply expansion of links of \(d-1\)-faces (if they are connected).
\begin{theorem}[Theorem 5.2 in \cite{Oppenheim2018}]
\label{thm:Oppenheim-trickling-down-lemma}

If for all $s \in X(k)$, $\lambda(X_s) \leq \lambda$, for some $\lambda \in (0,\frac{1}{2}]$, then for any $r \in X(k-1)$, s.t. $X_r$'s \(1\)-skeleton is connected, $\lambda_2(X_r) \leq \frac{\lambda}{1-\lambda}$.
\end{theorem}


\subsection{Simplicial homomorphisms} \label{sec:simplicial-homomorphism}
Let \(X,Y\) be two pure \(d\)-dimensional simplicial complexes. A \emph{simplicial homomorphism} is a function \(f:X(0) \to Y(0)\) so that every \(k\)-face \(s \in X\) is mapped by \(f\) to a face \(k\)-face \(f(s)\in Y(k)\). We say that \(f\) is a \emph{non-degenerate \(Y\)-coloring} if for every \(d\)-face \(t \in Y(d)\) there exists some \(d\)-face \(s \in X(d)\) so that \(f(s)=t\). 

A surjective \emph{function} (not necessarily a coloring) \(f:X(0) \to Y(0)\) induces a measure \(\mu_f\) on \(X\) defined by the following sampling process:
\begin{enumerate}
    \item Sample a \(d\)-face \(t \in Y(d)\).
    \item Sample a \(d\)-face \(s \in X(d)\) given that \(f(s)=t\).
\end{enumerate}
We note that by taking a subcomplex \(X' \subseteq X\) that contains all faces in \(X\) that are in the support of \(\mu_f\), then \(f:X'(0)\to Y(0)\) is a non-degenerate coloring (we can also induce a measure on \(X'\) so that for every \(d\)-face \(s \in X'(d)\), \(\nu_X'(s) = \nu_X(s) / \nu_X(X'(d))\)).
%

We say that two simplicial complexes \(X,Y\) are \emph{isomorphic} if there exists a bijection \(f:X(0) \to Y(0)\) so that \(s \in X\) if and only if \(f(s) \in Y\). We then denote \(X \cong Y\).

\paragraph{Simplicial homomorphisms and oriented faces} Let \(\dir{s} = (v_0,v_1,...,v_d) \in \dir{X}(d)\), whose underlying set \(s = \set{v_0,v_1,...,v_d}\). Let \(f(\dir{s})=(f(v_0),f(v_1),...,f(v_d))\) and let \(f(s) = \set{f(v_0),f(v_1),...,f(v_d)}\).

its probability is
\[\Prob[\mu_f]{\dir{s}}=\frac{1}{(d+1)!} \Prob[\mu_f]{s} = \frac{1}{(d+1)!} \Prob[Y]{f(s)} \cProb{s'\in X(d)}{s'=s}{f(s')=f(s)}\]
\[ = \Prob[Y]{f(\dir{s})}\cProb{s'\in X(d)}{s'=s}{f(s')=f(s)}.\]
For every \(s'\) so that \(f(s')=f(s)\) there is a single orientation \(\dir{s'}\) of \(s'\) so that \(f(\dir{s'})=f(\dir{s})\) so this is equal to
\[ = \Prob[Y]{f(\dir{s})}\cProb{\dir{s}'\in \dir{X}(d)}{\dir{s}'=\dir{s}}{f(\dir{s}')=f(\dir{s})}.\]

Hence we can also describe \(\mu_f\) via the sampling process of oriented faces:
\begin{enumerate}
    \item Sample an oriented \(d\)-face \(\dir{t} \in \dir{Y}(d)\).
    \item Sample an oriented \(d\)-face \(\dir{s} \in \dir{X}(d)\) given that \(f(\dir{s})=\dir{t}\).
\end{enumerate}

\subsection{Simplicial covers}
\begin{definition}[Simplicial cover]
    Let \(X,Y\) be two pure \(d\)-dimensional simplicial complexes. We say \(Y\) \emph{covers} \(X\) if there exists a surjective simplicial homomorphism \(f:Y \to X\) so that for every non-empty face \(s \in Y\), \(f|_{Y_s(0)}\) is a simplicial isomorphism between \(Y_s\) and \(X_{f(s)}\).
\end{definition}
Simplicial covers are the combinatorial equivalent of \emph{topological covers}. These are usually classified and analyzed by the \emph{fundamental group} of a simplicial complex as a topological space, and the \emph{universal cover}. For more on this, see \cite{Surowski1984}. In this paper however, we use an alternative description of covers via the non-abelian cocycles (defined below). We define what we need below, however we will not give full proofs (but we refer the reader again to \cite{Surowski1984}).

Let \(X\) be a connected finite \(d\)-dimensional simplicial complex (for \(d \geq 2\)), where \(X(0)=\set{v_1,v_2,...,v_n}\) (the vertices are ordered in some arbitrary order).  Let \(\Gamma\) be a group (not necessarily finite). Let \(f:X(1) \to \Gamma\) be some labeling. We say that \(f\) is a \(\Gamma\)-cocycle if for every triangle \(v_i v_j v_k \in X(2)\) where \(i<j<k\) it holds that
\[f(v_i v_j)f(v_j v_k)f(v_i v_k)^{-1}=e\]
where \(e\) is the identity in \(\Gamma\). We denote the set of cocycles by \(Z_1(X,\Gamma)\).

For a cocycle \(f \in Z_1(X,\Gamma)\) the \(f\)-cover of \(X\) is the following simplicial complex.

\begin{definition}[Group cover construction]
Let \(X\) be a pure \(d\)-dimensional simplicial complex, and \(\Gamma\) be a finite group. Let \(f \in Z_1(X,\Gamma)\). The \(f\)-cover of \(X\), denoted by \(\tilde{X} = X^f\) is a pure \(d\) dimensional simplicial complex defined as follows. The vertex set \(\tilde{X}(0) = X(0)\times \Gamma\). The \(d\)-faces are all sets \(s=\set{(v_{i_0},g_0),...,(v_{i_d},g_d)}\) so that (1) \(\set{v_{i_0},v_{i_1},...,v_{i_d}}\in X(d)\) and (2) for every two distinct \(i_{j_1},i_{j_2}\) so that \(i_{j_1}< i_{j_2}\), \(f(v_{i_{j_1}} v_{i_{j_2}})=g_{i_{j_1}}^{-1}g_{i_{j_2}}\).

If \(X\) is measured, then \(\mu_{\tilde{X}}(s) = \mu_X(s|_{X}) \frac{1}{|\Gamma|}\), where \(s|_{X}\) of a face \(s = \set{(v_{i_0},g_0),...,(v_{i_k},g_k)}\) is \(s|_{X} = \set{v_{i_0},...,v_{i_k}} \in X(k)\).
\end{definition}
One can show that \(\tilde{X}\) indeed covers \(X\) with the map \(\phi:\tilde{X} \to X\), \(\phi(v_i,g)=v_i\). We note for the curious reader that there is also an inverse to this statement, see \cite{Surowski1984} Proposition 2.3 for a precise statement.
%

We are interested in simplicial complexes that are connected. The following claim characterizes the connected components of \(X^f\).

Let \(\dir{X}(1) = \sett{(v_i,v_j),(v_j,v_i)}{v_i v_j \in X(1)}\) be the set of directed edges on \(X\). For a labeling \(f:X(1) \to \Gamma\), we denote by \(\dir{f}:\dir{X}(1)\to \Gamma\) the function 
\begin{align} \label{eq:dir-f-def}
    \dir{f}(v_i,v_j) = \begin{cases} 
f(v_i v_j) & i < j \\
f(v_i v_j)^{-1} & i>j
\end{cases}.
\end{align}
\begin{claim} \label{claim:connectedness}
Let \(X\) be a \(d\)-dimensional simplicial complex whose \(1\)-skeleton is connected. Fix some arbitrary \(v \in X(0)\) and denote by \(H_v \subseteq G\) to be the set of all elements \(g\in G\) for which there is a cycle \(v=v_{i_0},v_{i_1},v_{i_2},...,v_{i_m}=v\) such that \(g=\dir{f}(v_{i_0}, v_{i_1}) \dir{f}(v_{i_1}, v_{i_2}) \dir{f}(v_{i_2}, v_{i_3}) ... \dir{f}(v_{i_{m-1}}, v_{i_{m}})\).

Then \(H_v\) is a subgroup of \(\Gamma\). Moreover, for every \([g] \in \Gamma/H_v\) there is a connected component in the cover \(X^f\) of \(X\). 
\end{claim}

We prove the claim below, but the following corollary is all we need for our result:
\begin{corollary} \label{cor:connected-characterization}
Let \(X\) be a \(d\)-dimensional simplicial complex whose \(1\)-skeleton is connected. \(X^f\) is connected if and only if \(H_v = G\) for some \(v\in X(0)\).
\end{corollary}

Finally, we note that if \(f \in Z_1(X,\Gamma)\) and \(N \trianglelefteq \Gamma\) is a normal subgroup, then \(f_N \in Z_1(X,\Gamma / N)\) that is defined by \(f_N(v) = f(v)N\) is also a cocycle. Thus if \(\Gamma\) has an infinite sequence of finite-indexed normal subgroups \(N_1,N_2,...\) whose index goes to infinity, then any \(f \in Z_1(X,\Gamma)\) gives rise to an infinite family of simplicial complexes that cover \(X\), namely, \(\set{X^{f_{N_m}}}_m\).

\begin{proof}[Proof of \pref{claim:connectedness}]
\(e \in H_v\) as the trivial loop gives \(e\). If \(g_1, g_2 \in H_v\) then concatenating their cycles gives the element \(g_1 g_2\). For an element \(g \in H_v\), with a cycle \(v=v_{i_0},v_{i_1},v_{i_2},...,v_{i_m}=v\). The reverse cycle \(v=v_{i_m},v_{i_{m-1}},v_{i_{m-2}},...,v_{i_0}=v\) shows that \(g^{-1} \in H_v\). Thus \(H_v\) is a subgroup of \(\Gamma\).

Next we show that every vertex \((v,g)\) is connected \((v,g')\) if and only if \(g' \in gH\):
Recall that every edge in \(X^f\) is \((u,h)\tilde (w,h\dir{f}(uw))\). Thus any path \(((u_0,g_0),(u_1,g_1),...,(u_m, g_m))\) implies that \(g_m=g_0\dir{f}(u_0 u_1) \dir{f}(u_1 u_2) ... \dir{f}(u_{m-1} u_m).\)
In particular, any path from \((v_0=v,g_0),(v_1,g_1),...,(v_m=v,g_m)\) implies that \(g_m=g_0\dir{f}(v v_1) \dir{f}(v_1 v_2) ... \dir{f}(v_{m-1} v).\) As \((v_0=v,v_1,...,v_m=v)\) is a cycle from \(v\) to itself, then \(\dir{f}(v_1 v_2) ... \dir{f}(v_{m-1} v) \in H_v\). Thus \((v,g)\) is connected only to \(\sett{(v,g')}{g'\in gH_v}\).

On the other hand, for every \(g'\in gH\), there exists a cycle \((v_0=v,v_1,...,v_m=v)\), so that \(\dir{f}(v_1 v_2) ... \dir{f}(v_{m-1} v) = g^{-1}g'\). Hence the path \((v,g),(v_1,g\dir{f}(v v_1)),...,(v_m=v,g\dir{f}(v_1 v_2) ... \dir{f}(v_{m-1} v))\) goes from \((v,g)\) to \((v,g(g^{-1}g'))\).

Finally, we note that every vertex \((u,h)\) is connected to at least one \((v,g)\) (since there is a path from \(u\) to \(v\) that lifts to a path from \((u,h)\) to \emph{some} \((v,g)\)). We've shown that that every two vertices in a set \(\set{v} \times gH_v\) are connected to one another, that every other vertex is connected to at least one of these sets, and that two distinct \(\set{v} \times gH_v\) and \(\set{v} \times g'H_v\) are disconnected. This shows that there is a connected component for every \(gH_v \in \Gamma / H_v\).
\end{proof}

\subsection{The clique complex of a Cayley graph} \label{sec:group-complex}
Let \(\Gamma\) be a finitely presented group (not necessarily finite), and let \(S = \set{s_1,s_2,...,s_m}\) a finite symmetric set of generators. We define \(C(\Gamma,S)\), as the clique complex that arises from the Cayley graph \(Cay(\Gamma,S)\). More explicitly, the vertex set is \(\Gamma\), the labeled edges are all pairs \((s_i,s_k)_{s_j}\) so that \(s_i s_j = s_k\), and the \(d\)-faces are all \(d\)-cliques in this graph. While this complex is infinite, the neighbourhood of every vertex is of finite size \(|S|\).

We limit ourselves to groups and generating sets where there exists some \(d\geq1\) so that every edge is in some \((d+1)\)-clique. In this case, the resulting complex is \emph{pure} and \((d+1)\)-dimensional. We abuse definitions and say that an infinite complex is an \(\eta\)-high dimensional expander if all its links (not including the complex itself) are \(\eta\)-spectral expanders.
\begin{example}
\cite{LubotzkySV2005b} showed that for any \(d>0\) and prime power \(q^e\), the group \(PGL_d(\mathbb{F}_{q^e})\) has a generating set \(S\) of size \(m=m(d)\) so that \(C(\Gamma,S)\) is a pure \(d\)-dimensional simplicial complex. Moreover, all links of \(C(\Gamma,S)\) are \(\eta\)-one-sided spectral expanders for \(\eta\) that goes to \(0\) as \(q^e\) tends to infinity. By using \pref{thm:Oppenheim-trickling-down-lemma} one can show that low dimensional skeletons of these complexes are also two-sided spectral expanders.
\end{example}

\section{Random high dimensional expanders from covers} \label{sec:main}
In this section we present our construction of high dimensional expanders that are based on covers. We first define the simplicial complexes that are suitable for our construction.
\begin{definition} \label{def:suitable-complex}
    Let \(c,r > 1\) and \(\eta > 0\). Let \(X\) be a \(d\)-dimensional pure simplicial complex and denote by \(Q = \max_{v \in X(0)}\deg_d(v)\). We say that \(X\) is \((c,r,\eta)\)-\emph{suitable} if
\begin{enumerate}
    \item \(X\) is an \(\eta\)-two-sided high dimensional expander.
    \item For every \(0 \leq \ell \leq d-2\), \(\sigma \in X(\ell)\) and every \(v \in X_\sigma(0)\), the degree of \(v\) in the \(1\)-skeleton of \(X_\sigma(0)\) is at least \(c (1+\log Q)\).
    \item For every \(0 \leq \ell \leq d-2\) and \(\sigma \in X(\ell)\), the weight of every edge in \(e \in X_\sigma(1)\) is between \(\frac{1}{r|X_\sigma(1)|}\) and \(\frac{r}{|X_\sigma(1)|}\). Moreover, the weight of every vertex \(v\in X_\sigma(0)\) is between \(\frac{1}{r|X_\sigma(0)|}\) and \(\frac{r}{|X_\sigma(0)|}\).
\end{enumerate}
\end{definition}
Note that a \((c,r,\eta)\)-suitable simplicial complex is also \((c',r',\eta')\)-suitable for any \(c'\leq c\), \(r' \geq r\) and \(\eta' \geq \eta\).

\begin{theorem}[Main] \label{thm:main-formal}
For every pair of integers \(d,m\) there exists some \(c,r > 1, \eta>0\) so that the following holds. Let \(\Gamma\) be a group with a generating set \(S\) of size \(m\). Assume that \(C(\Gamma,S)\) is a \(d\)-dimensional \(\frac{\lambda}{2}\)-high dimensional expander for some \(\lambda < 1\). Let \(X\) be a \((c,r,\eta)\)-suitable \(d\)-dimensional simplicial complex. Then there exists some \(Y \subseteq X\) so that:
\begin{enumerate}
    \item \(Y\) is an \(d\)-dimensional \(\lambda\)-high dimensional expander.
    \item \(Y\) has a connected \(\Gamma\)-cover.
    \item \(Y\) contains \(\Omega_{m,d}(1)\) fraction of the faces of \(X\).
\end{enumerate} 
There is a randomized algorithm that outputs \(Y\) in expected polynomial time.
\end{theorem}
    
\begin{remark}
We prove \pref{thm:main-formal} in an existential manner, using the Lovasz Local Lemma. Once existence is establish, we use the algorithmic version \pref{thm:l3} by \cite{MoserT2010} as a black box to obtain an algorithm that outputs \(Y\).
\end{remark}

\subsection{Proof outline} \label{sec:proof-outline-main-theorem}
As discussed in the introduction, we prove \pref{thm:main-formal} via the probablistic method. Fix a group \(\Gamma\), and a high dimensional expander \(X\) that satisfy the assumptions of the theorem. We sample a function \(f:X(1) \to \set{s_1,...,s_m}\) uniformly at random, and remove all faces \(s \in X\) that contain a triangle \(v_i v_j v_k \in X\) so that \(f(v_i v_j)f(v_j v_k)f(v_i v_k)^{-1} \ne e\) (where \(i<j<k\)). Denote by \(Y \subseteq X\) the resulting sub-complex.

We cannot prove that \(Y\) will be a high dimensional expander with high probability. For example, if \(Q \ll |X(d)|\) (where \(Q\) is the maximal \(d\)-degree of a vertex), then with high probability there exists a face \(\tau \in X(d-2)\) with a disconnected link. This is because the event that a face \(\tau\) has a disconnected link depends only on \(poly(Q)\) edges (the edges whose distance from \(\tau\) is at most \(2\)). Thus, the probability is a function of \(Q\) only. When \(Q\) stays constant and \(|X(d)|\) grows to infinity, the number of faces with a disconnected link will grow in expectation.

This is where the Lovasz Local Lemma comes into play. For every face \(\tau \in X\) we will define a ``bad event'' \(E_\tau\), that informally, says that the link of \(\tau\) is not an expander graph \footnote{The formal definition of \(E_\tau\) is more involved. We also want that when all \(E_\tau\) don't occur, then \(Y\) will have a connected \(\Gamma\)-cover and that \(Y\) will contain a constant fraction of the faces \(X\).}.

To use the Lovasz Local Lemma, we show that these events \(E_\tau, E_{\tau'}\) are independent whenever the distance between every pair of vertices \(v\in \tau, u\in \tau'\) is greater than \(2\). Thus every \(E_\tau\) depends on \(poly(Q)\) other \(E_{\tau'}\). This makes sense since the structure of the link of a face \(\tau \in X\), depends only on the value of \(f\) on edges \(uv\) so that \(\tau \cup uv \in X\).

Furthermore, we show that the probability of \(E_\tau\) is at most \( 1/Q^{\ell}\) for \(\ell\) as large as we need (this is where the fact that \(X\) is \((c,r,\eta)\)-suitable for appropriate \(c,r,\eta\) comes into play). Now we can apply the Lovasz Local Lemma, and get that \(\bigwedge_{t \in X} \neg E_t\) occurs with positive probability. This implies that the resulting sub-complex \(Y\) is a high dimensional expander (the same argument will show the other properties required from \(Y\) occur as well).

\subsection{Notation for the proof}
\begin{definition}
Let \(f:X(1) \to \set{s_1,...,s_m}\). 
\begin{enumerate}
    \item We say a face \(\sigma \in X\) \emph{is satisfied by \(f\)} if for every triangle \(v_i v_j v_k \subseteq \sigma\) so that \(i<j<k\) it holds that \(f(v_i v_j)f(v_j v_k)f(v_i v_k)^{-1}=e\).
    \item For a face \(\sigma \in X\), we say that \(\sigma\) is \emph{maximally satisfied} if there exists \(\tau \in X(d)\) so that \(\sigma \subseteq \tau\) and \(\tau\) is satisfied by \(f\).
\end{enumerate}
\end{definition}
Note that vertices and edges are (vacuously) satisfied by any \(f\). We observe the following about satisfied faces.
\begin{observation} \label{obs:satisfaction}
A face \(\set{v_0,v_1,...,v_r}\) is satisfied if and only if \(\set{f(v_0 v_1), f(v_0 v_2),...,f(v_0 v_r)} \cup \set{e}\) is a face in \(C(\Gamma,S)\), if and only if \(\set{f(v_0 v_1), f(v_0 v_2),...,f(v_0 v_r)} \cup \set{e}\) is a clique in \(Cay(\Gamma,S)\).
\end{observation}
This observation follows from the fact that whenever \(f(v_i v_j)f(v_j v_k)f(v_i v_k)^{-1}=e\) this implies that \(e, f(v_i v_j), f(v_i v_k)\) are a triangle in the Cayley graph of \(\Gamma,S\).

\begin{definition}
Let \(f:X(1) \to \set{s_1,...,s_m}\). The \emph{\(f\)-pruning} of \(X\) is the simplicial complex \(Y\) whose vertices are \(Y(0)=X(0)\) and faces are all \(t \in X\) that are maximally satisfied by \(f\).
\end{definition}
We imbue \(\dir{Y}(d)\) with the following measure, where we sample an oriented \((d-1)\)-face in the link of the unit \(e\) in \(C(\Gamma,S)\), \((s_{i_0} s_{i_1},...,s_{i_{d-1}})\in C_{e}(\Gamma,S)\), and then sample an oriented \(d\)-face \((u_0,u_1,...,u_d) \in Y\) so that  \(f(u_0 u_j)=s_{i_j}\) and for every \(j_1,j_2 \ne 0\), \(f(u_{j_1} u_{j_2}) = s_{j_1}^{-1} s_{j_2}\). This induces a measure over \(Y(d)\) by taking the underlying set of \((u_0,u_1,...,u_d)\). As \(\prob{(u_0,u_1,...,u_d)} = \prob{(u_{\pi(0)},u_{\pi(1)},...,u_{\pi(d)})}\) for any permutation \(\pi\), it holds that \(\prob{u_0,u_1,...,u_d} = (d+1)! \prob{(u_0,u_1,...,u_d)}\). This measure is only possible if for every \(C(\Gamma,S)\), \(\set{s_{i_0} s_{i_1},...,s_{i_d}}\in C_{e}(\Gamma,S)\) there exists a face in \(Y\) so that \(f(u_0 u_j)=s_{i_j}\) and for every \(j_1,j_2 \ne 0\), \(f(u_{j_1} u_{j_2}) = s_{j_1}^{-1} s_{j_2}\). When this is not possible we do not consider \(Y\) as measured.

\begin{definition}
Fix \(f:X(1) \to \set{s_1,...,s_m}\), and let \(\sigma \in X(\ell)\) be a satisfied face for \(\ell \leq d-2\). The \emph{satisfaction graph} of \(\sigma\), denoted by \(G_\sigma = (V_\sigma, E_\sigma)\) is the following:
\begin{itemize}
    \item The vertices are all \(v \in X_\sigma(0)\) so that \(\sigma \dunion v\) is satisfied.
    \item The edges are all \(uv \in X_\sigma(1)\) so that \(uw \dunion \sigma\) is satisfied.
\end{itemize}
\end{definition}
Let \(\sigma = \set{v_0,v_1,...,v_r}\) be some \(f\)-satisfied face. Denote by \(f(v_0 v_j) = s_{i_j}\) and let \(a = \set{e} \cup \set{s_{i_j}}_{j=1}^r\). Recall that \(C_a\) is the link of \(a\) in the complexes induced by cliques in \(Cay(\Gamma,S)\). The satisfaction graph \(G_{\sigma}\) has a \(C_a\)-coloring \(\psi_{\sigma}:V_\sigma (0) \to C_a(0)\) by \(\psi_\sigma (u) = f(v_0 u)\). The fact that this is a coloring following from \pref{obs:satisfaction}. We consider the measure on \(G_\sigma\) to be the measure induced by the coloring \(\psi_\sigma\) as described in \pref{sec:simplicial-homomorphism}. More explicitly, we sample an edge \(uv \in E_\sigma\) by sampling \(s_i s_j \in C_a\) and then we sample an edge \(uv \in E_\sigma\) so that

\subsection{The proof}
\begin{proof}[Proof of \pref{thm:main-formal}]
Consider the probabilistic experiment where we sample \(f:X(1) \to \set{s_1,...,s_m}\) uniformly at random and set \(Y \subseteq X\) to be the \(f\)-pruning of \(X\).
%
Define the following events for faces in \(\tau \in X\):
\begin{itemize}
    \item[*] \(AT(\tau)\), the event where for some face \(\tau=\set{v_0,v_1,...,v_\ell} \in X(\ell)\), there exists an \(\ell+1\) tuple \((s_{i_0},s_{i_1},...,s_{i_\ell})\) so that the probability (in \(X_\tau\)) of \(u \in X_\tau(0)\) so that that \(f(v_0 u) = s_{i_0}, f(v_1 u)=s_{i_1},...,f(v_\ell u)=s_{i_\ell}\) is either greater than \(\geq \frac{r^2}{m^{\ell+1}} \) or smaller than \(\leq \frac{1}{r^2 m^{\ell+1}}\). (AT stands for ``atypical tuple''). Note that as this should hold for all tuples, then this doesn't depend on the chosen orientation of \(\tau\).
    \item[*] \(NE(\tau)\), the event where the satisfaction graph \(G_\tau\) is not a \(\frac{\lambda}{2}\)-spectral expander (NE stands for ``not expander'').
    \item[*] \(BC(\tau)\), for \(\tau \in X(0)\), the event where there exists some \(s_i \in S\) so that every for every triangle \(\tau u w \in X(2)\), it holds that \(\dir{f}(\tau u) \dir{f}(u,w) \dir{f}(w \tau)\ne s_i\)\footnote{Recall that by \eqref{eq:dir-f-def} the definition of \(\dir{f}:\dir{X}(1) \to \Gamma\) is \(\dir{f}(v_i,v_j) = f(v_i v_j)\) if \(i<j\) and \(\dir{f}(v_i,v_j) = f(v_i v_j)^{-1}\) otherwise.} (BC stands for ``bad cover''). 
\end{itemize}

These events allow us to define \(\sett{E_\tau}{\tau \in X \setminus X(d)}\) as in the proof outline:
\begin{align}
E_\tau = \begin{cases} 
    AT(\tau) & \tau \in X(d-1) \\
    AT(\tau) \cup NE(\tau) & \tau \in X(\ell), 1 \leq \ell \leq d-2 \\
    AT(\tau) \cup NE(\tau) \cup BC(\tau) & \tau \in X(0)
\end{cases}.
\end{align}
The formula is ambiguous when \(d=2\). In this case we define \(E_\tau  = AT(\tau) \cup NE(\tau) \cup BC(\tau)\) for all vertices \(\tau \in X(0)\).

Finally let \(\mathcal{G} = \bigwedge_{\tau \in X \setminus X(d)} \neg E_\tau\). We claim the following:
\begin{claim}\label{claim:G-implies-Y-is-good}
When \(\mathcal{G}\) occurs, then 
\begin{enumerate}
    \item \(Y\) is an \(d\)-dimensional \(\lambda\)-high dimensional expander.
    \item \(Y\) has a connected \(\Gamma\)-cover.
    \item \(Y\) contains \(\Omega_{m,d}(1)\) fraction of the faces of \(X\).
\end{enumerate} 
\end{claim}

\begin{claim} \label{claim:G-occurs-with-positive-probability}
The event \(\mathcal{G}\) occurs with positive probability.
\end{claim}

The proof is complete given these claims.
\end{proof}

\begin{proof}[Proof of \pref{claim:G-implies-Y-is-good}]
We assert that when \(\neg AT(\tau)\) holds for all \(\tau \in X \setminus X(d)\) the measure of the \(1\)-skeleton in every link \(X_\sigma\) is similar to the measure induced by \(G_\sigma\) for every \(\sigma \in X\).
\begin{claim} \torestate{\label{claim:not-at-implies-measure-behaves-nicely-in-links}
When \(\bigcup_{\tau \in X \setminus X(d)} \neg AT(\tau)\) occurs, then
\begin{enumerate}
    \item Let \(j \leq d-1\) and \(\dir{\sigma} = (v_0,v_1,...,v_j) \in \dir{X}(j)\) be some satisfied face. Let \(\bar{s} = \set{s_{i_{\ell_1},i_{\ell_2}}}_{\ell_1 \ne \ell_2}\) be so that \(s_{i_{\ell_1},i_{\ell_2}} = f(v_{\ell_1} v_{\ell_2})\) is the labeling of the edges of \(\sigma\)\footnote{formally \(\bar{s}\) is a function that receives as input \(\ell_1 \ne \ell_2 \in \set{0,1,...,j}\) and outputs some \(s_{\ell_1,\ell_2} \in S\).}. Let \(\hat{\sigma} = (s_{i_0,i_1}, s_{i_0 i_2},...,s_{i_0, i_j})\). Then 
    \[r^{-9d} \Prob[C_e]{\hat{\sigma}} \cProb{X}{\dir{\sigma}}{\bar{s}} \leq \Prob[Y]{\dir{\sigma}} \leq r^{9d}\Prob[C_e]{\hat{\sigma}} \cProb{X}{\dir{\sigma}}{\bar{s}},\] 
    where \(\cProb{X}{\dir{\sigma}}{\bar{s}}\) is the probability of sampling \((u_0,u_1,...,u_j) \in X(j)\) given that we sampled some \((u_0',u_1',...,u_j') \in X(j)\) so that its edges are labeled \(f(u_{\ell_1}' u_{\ell_2}') = s_{i_{\ell_1},i_{\ell_2}}\). We interpret \( \Prob[C_e]{\hat{\sigma}} = 0\) when \(\sigma \notin C_e\).
    \item For every \(-1\leq j \leq d-2\), every \(\sigma \in X(j)\) and \(\dir{uv} \in \dir{Y}_\sigma(1)\) or \(w \in Y_\sigma(0)\) it holds that
    \[r^{-15d} \leq \frac{\Prob[Y_\sigma]{\dir{uv}}}{\Prob[G_\sigma]{\dir{uv}}} \leq r^{15d}.\]
    \[r^{-15d} \leq \frac{\Prob[Y_\sigma]{w}}{\Prob[G_\sigma]{w}} \leq r^{15d}.\]
\end{enumerate}}
\end{claim}
We prove this claim in \pref{sec:proof-of-claim-not-at-implies-measure-behaves-nicely-in-links}. In particular, the first item implies that every satisfied face is also maximally satisfied since it has positive probability. This implies that the \(1\)-skeleton of \(Y_\tau\) and \(G_\tau\) are the same as sets. The second item implies that for every \(\tau\), the \(\psi_\tau\) coloring is non-degenerate and the the measures of \(G_\tau\) and \(Y_\tau\) are similar.

\begin{enumerate}
    \item \textbf{\(Y\) is an \(d\)-dimensional \(\lambda\)-high dimensional expander.}

If \(\mathcal{G}\) occurs, then for every \(\tau \in X\) of dimension \(\leq d-2\), the \(1\)-skeleton of a link \(Y_\tau\) is the satisfaction graph \(G_\tau\) (since every satisfied face is also maximally satisfied). Moreover, fix some \(\tau\) and let \(A\) is the adjacency matrix of the \(1\)-skeleton of \(Y_\tau\) and \(\tilde{A}\) is the adjacency matrix of \(G_\tau\) defined via the non-degenerate coloring \(\psi_\tau\). The spectral norm of \(A\) is the operator norm of the matrix \(B = A - J_{1}\) where \((A-J_1)_{u,v} = \cProb{w \in Y_\tau(0)}{w=v}{w \sim u} - \Prob[Y_\tau]{v}\) and similarly the spectral norm of \(\tilde{A}\) is the operator norm of \(\tilde{B} = A - J_2\) where \((A-J_2)_{u,v} = \cProb{w \in G_\tau}{w=v}{w \sim u} - \Prob[G_\tau]{v}\). 

By the second item of \pref{claim:not-at-implies-measure-behaves-nicely-in-links} and direct calculations, the difference matrix \(\mathcal{E} = B - \tilde{B}\) is a matrix so that in every entry \((u,v)\)
\[\abs{{\mathcal{E}_\tau}(v,u)} \leq (r^{30d} - 1) \cdot \left ( \cProb{w \in G_\tau}{w=v}{w \sim u} + \Prob[G_\tau]{v} \right ).\] 
Thus the \(\ell_1\) norm of every row is at most \(2 (r^{30d} - 1)\). This implies that \(\norm{\mathcal{E}}_{op} \leq 2(r^{30d} - 1)\), and hence
\[ \lambda(Y_\tau) = \norm{B}_{op} \leq \norm{\tilde{B}}_{op} + 2(r^{30d} - 1) = \lambda(G_\tau) + 2(r^{30d} - 1).\]
By \(\neg NE(\tau)\), we have that \(\lambda(G_\tau) \leq \frac{\lambda}{2}\) and so if we choose \(r\) so that \(2(r^{30d} - 1) \leq \frac{\lambda}{2}\) we get that \(\lambda(Y_\tau) \leq \lambda\).

\item \textbf{\(Y\) has a connected \(\Gamma\)-cover.}

The cover is the one induced by \(f\) itself. By \pref{cor:connected-characterization} it is enough to fix some \(v \in Y(0)\) and show that for every \(g \in \Gamma\) there is a cycle in \(Y(1)\), \((v=v_0,v_1,...,v_t=v) \) so that 
\[ \dir{f}(v_0,v_1) \dir{f}(v_1,v_2) \cdot ... \cdot \dir{f}(v_{t-1},v_t) = g. \]
The cycle need not be simple. If \(\mathcal{G}\) occurs the \(1\)-skeleton of \(X\) and of \(Y\) are the same, so its enough to find such a cycle in \(X(1)\).
Fix an arbitrary vertex \(v \in X(0)\). When \(\neg BC(v)\) holds, for every generator \(s_i\) there is a path \((v,u,w,v)\) (in the \(1\)-skeleton) so that \(\dir{f}(v,u) \dir{f}(u,w) \dir{f}(w,v) = s_i\). In particular, this implies that for every \(g=s_{i_1} s_{i_2}...s_{i_r} \in \Gamma\) we can concatenate the cycles for \(s_{i_1},s_{i_2},...,s_{i_r}\) to get a cycle that generates \(g\) as required.

\item \textbf{\(Y\) contains \(\Omega_{m,d}(1)\) fraction of the faces of \(X\).}

If \(\mathcal{G}\) occurs, if an \(\ell\)-face satisfies \(f\) there are at least \(\frac{1}{2m^d}\)-fraction of the \(\ell+1\)-faces that contain it that also satisfy \(f\). Thus by induction, at least \(\left (\frac{1}{2m^d} \right )^{d-\ell}\)-fraction of the \(\ell\)-faces of \(X\) that appear in \(Y\) (where the basis of the induction is by the fact that all edges of \(X\) appear in \(Y\)).
\end{enumerate}
\end{proof}

\begin{proof}[Proof of \pref{claim:G-occurs-with-positive-probability}]
To show that \(\mathcal{G}\) has positive probability we use Lovasz Local Lemma \pref{thm:l3}.
First, for a given \(\tau \in X\) we bound the number \(\tau' \in X\) so that \(E_\tau\) and \(E_{\tau'}\) are not independent.

Note that \(AT(\tau),NE(\tau)\) and \(BC(\tau)\) depend only on the value \(f(uv)\) for edges \(uv \in X(1)\) that either intersect \(\tau\), or so that at least one of \(u,v \in X_\tau(0)\). \(E_\tau\) is independent of the value of \(f(uv)\) for the rest of the edges. Thus, for every \(\tau,\tau' \in X\), if \(E_\tau, E_{\tau'}\) are not independent, then there is a path of length at-most \(2\) edges from a vertex in \(\tau\) to a vertex in \(\tau'\).

Denote by \(Q\) is the maximal number of \(d\)-faces that a contain some fixed vertex \(v \in X(0)\). 
Let \(R\) be the number of edges adjacent to a vertex \(v \in X(0)\). For a fixed \(\tau\), there are at most \(d 2^d Q(1+R+R^2)\) faces \(\tau' \in X\) so that \(E_\tau, E_{\tau'}\) depend on the same edge. This (crude) bound is obtained as follows: we have at most \(d\) choices for \(v \in \tau\), then we have \(1+R+R^2\) choices for \(u \in X(0)\) so that \(u\) is connected to \(v\) by a path of length \(\leq 2\). There are at most \(Q\) \(d\)-faces that contain \(u\) hence there are at most \(Q2^d\) faces on all levels that contain \(u\). Denote \(K = d 2^d Q(1+R+R^2)\), and we record that \(K = poly(Q)\). 

If we show that for all \(\ell<d\) and \(\tau \in X(\ell)\) it holds that
\begin{equation} \label{eq:lll-prob-bound}
    \prob{E_{\tau}}<\frac{1}{(K+1)e}
\end{equation} 
then by the Lovazs Local Lemma, we will be promised that \(\mathcal{G}\) occurs with some positive probability. The \(e\) in \eqref{eq:lll-prob-bound} is Euler's number.

To show \eqref{eq:lll-prob-bound}, we need the following assertions.
\begin{lemma} \torestate{\label{lem:good-expander}
There exists \(c,r>1\) and \(\eta > 0\) so that the following guarantee holds for any \((c,r,\eta)\)-suitable \(d\)-dimensional simplicial complex \(X\) with maximal \(d\)-degree \(Q\). Let \(\tau \in Y(j)\). Then \[\prob{NE(\tau)} \leq 8m^3|X_\tau(0)|^2 e^{\log|X_\tau(0)| - \mu c \log Q } + (|X_\tau(0)|+1) m^d e^{0.03m^{-d}(r-1)^2 c \log Q}.\] 
Here \(\mu=\mu(r,\lambda,m,d)\) is some positive constant.

That is, the satisfaction graph \(G_\tau\) is a \(\frac{\lambda}{2}\)-spectral expander with probability at least \(1-8m^3|X_\tau(0)|^2 e^{\log|X_\tau(0)| - \mu c \log Q }\ - (|X_\tau(0)|+1) m^d e^{0.03m^{-d} (r-1)^2 c \log Q}\). 
}
\end{lemma}

\begin{claim} \label{claim:face-not-appearing-prob}
Let \(X\) be \((c,r,\eta)\)-suitable with maximal \(d\)-degree \(Q\). Let \(\tau \in X(\ell)\) for \(\ell \leq d-1\). Then \(\prob{AT(\tau)} \leq m^{d} e^{-0.03m^{-d} (r-1)^2 c \log Q}\).
\end{claim}

\begin{claim} \label{claim:face-too-light-prob}
Let \(X\) be \((c,r,\eta)\)-suitable with maximal \(d\)-degree \(Q\), let \(\ell < d\). Fix \(v \in X(0)\). Then \[\prob{BC(v)} < \left (1-\frac{1}{m^2} \right)^{c log Q}.\]

That is, the probability that there exists some \(s_i \in S\) so that for every cycle \((v,u,w,v)\), it holds that  \(\dir{f}(vu)\dir{f}(uw)\dir{f}(wv)\ne s_i\) is at most \(\left (1-\frac{1}{m^2} \right)^{c log Q}\).
\end{claim}

For every \(\tau \in X\), \(|X_\tau(0)|\leq Q\). Thus in the assertions above, the upper bounds on the probabilities of \(AT(\tau),NE(\tau)\) and \(BC(\tau)\) are all of the form \(poly(Q)\cdot e^{-\Omega(c \log Q))}\) where the hidden constants depend on \(m,d,\lambda,r\). Hence by fixing some values of \(m,d,\lambda,r\) one can take large enough \(c\) and small enough \(\eta\) so that the bounds are all \( \leq \frac{1}{5(K+1)e}\) for any \((c,r,\eta)\)-suitable simplicial complexes \(X\).

Thus \eqref{eq:lll-prob-bound} holds.
\end{proof}

We first prove \pref{claim:face-not-appearing-prob} and \pref{claim:face-too-light-prob}. Then we prove \pref{lem:good-expander}.

\begin{proof}[Proof of \pref{claim:face-not-appearing-prob}]
Fix \(\tau=\set{u_0,u_1,...,u_\ell} \in X(\ell)\). Recall that the measure of every vertex in the link is \(\frac{1}{r|X_\tau(0)|}\leq \prob{v}\leq \frac{r}{|X_\tau(0)|}\). Thus for every tuple \(\bar{s} = (s_{i_0},s_{i_1},...,s_{i_\ell})\), the \emph{number} of vertices \(v \in X_\tau(0)\) so that \(f(u_0 v)=s_{i_0}, f(u_1 v) = s_{i_1}, ..., f(u_\ell v) =s_{i_\ell}\) is between \(r^{-1}m^{-d}|X_\tau(0)|\) and \(r m^{-d}|X_\tau(0)|\), then \(\neg AT(\tau)\) doesn't hold.

For a fixed tuple \(\bar{s} = (s_{i_0},s_{i_1},...,s_{i_\ell})\) the probability that \(v \in X_\tau(0)\) has that \(f(u_0 v)=s_{i_0}, f(u_1 v) = s_{i_1}, ..., f(u_\ell v) =s_{i_\ell}\) is \(m^{-d}\). Let \(E(v,\bar{s})\) be this event. As \(\sett{E(v,\bar{s})}{v \in X_\tau(0)}\) are independent. Let \(R_s\) be the random variable counting the number of \(E(v,\bar{s})\). The fraction of \(v \in X_\tau(0)\) so that \(E(v,\bar{s})\) holds is greater than \(r^2 m^{-d}\) or smaller than \(r^2 m^{-d}\) when \(\abs{R_s - m^{-d}|X_\tau(0)|} \geq (r-1)|X_\tau(0)|\). By Chernoff's bound, 
\[\prob{\abs{R_s - m^{-d}|X_\tau(0)|} \geq (r-1)|X_\tau(0)| } \leq e^{-0.03m^{-d}(r-1)^2 |X_\tau(0)|}.\]
Thus by a union bound over all tuples \(\bar{s}\), the probability that \(AT(\tau)\) occurs is at most \(m^d e^{-0.03m^{-d}(r-1)^2 |X_\tau(0)|}\).

\end{proof}

\begin{proof}[Proof of \pref{claim:face-too-light-prob}]
Fix \(v \in X(0)\), and let \(u \in X_v(0)\).
\begin{enumerate}
        \item There are at least \(c \log Q\) triangles that contain \(uv\). 
        \item All triangles share only the edge \(vu\). The rest of the edges \(uw,wv\) appear only in one three-cycle each.
        \item Thus the probability that \(\dir{f}(vu)\dir{f}(uw)\dir{f}(wv)=s_i\) is greater or equal to the probability that \(f(uw)=f(vu)^{-1} \ve f(wv) = s_i\), which is \(\frac{1}{m^2}\). Furthermore, these events are all independent from one another (given \emph{any} value \(f(vu)\)).
        \item Hence the probability that this doesn't occur for any triangle as above is at most \(\left (1- \frac{1}{m^2} \right )^{c \log Q}\). We multiply by \(m=|S|\) as a union bound on all \(s_i \in S\).
    \end{enumerate}
\end{proof}

\subsection{Structure of a satisfaction graph} \label{sec:link-description}
In order to prove \pref{lem:good-expander}, we give a description of the satisfaction graph of \(G_\tau\) for \(\tau \in X(\ell), \ell \leq d-2\). We already saw that as sets \(G_\tau\) is equal to the \(1\)-skeleton of \(Y_\tau\) when \(\bigcap_{\sigma} \neg AT(\sigma)\) holds. Recall the definition of \(C(\Gamma,S)\) to be the (possibly infinite) simplicial complex whose faces are cliques in the Cayley graph of \(\Gamma\) with generators \(S\).

%
Order \(\tau = \set{u_0,u_1,...,u_{\ell}}\) arbitrarily, and denote by \(s_{i_j} = \dir{f}(u_0 u_j)\). Denote by \(a = \set{e,s_{i_1},s_{i_2},...,s_{i_{\ell}}}\), and let \(C_a\) be the link of \(a \in C(\Gamma,S)\). Fixing the values of \(f\) on all the edges in \(\tau\), we use \pref{obs:satisfaction} to describe the satisfaction graph \(G_\tau\) according to the following probabilistic experiment:
\begin{enumerate}
    \item \textbf{Vertex sampling:} sample \(f(u_i v)\) for every \(v \in X_\tau (0)\) and \(u_i \in \tau\). We sample \(v\) into \(G_\tau\) if and only if \(\tau \cup v\) is satisfied by \(f\).
    
    This step partitions \(X_\tau(0)\) to distinct sets according to the values of \(f(u_0 v)\). Denote by \(S_i\) all vertices \(v\) so that \(\dir{f}(u_0 v) = s_i\).
    \item \textbf{Edge sampling:} Sample \(f(v v')\) for all edges \(vv' \in X_\tau (1)\). For every edge \(s_i s_j \in C_a(1)\) there exists some \(s_{i,j}\) so that \(s_i s_{i,j} = s_j\). In the edge sample step an edge \(v v' \in X_\tau(1)\) is sampled into \(Y_\tau (1)\) if and only if \(v \in S_i, v' \in S_j\), and \(f(vv')=s_{i,j}\) for some edge \(s_i s_j\) in the link of \(a\).
\end{enumerate}

For every \(s_i \in C_a(0)\), the probability that \(v \in X_\tau(0)\) is sampled into \(S_i\) is \(m^{-(\ell+1)}\): this is the probability that \(\dir{f}(u_0 v) = s_i\), and the rest of the edges \(\dir{f}(u_i v)\) are chosen so that all triangles in \(\tau \cup v\) are satisfied. Every two \(v,v' \in X_\tau(0)\) are sampled independently in this step.

Now consider the edge sample step. Let \(v \in S_i, v' \in S_j\) be so that \(v v' \in X_\tau(1)\). If \(s_i\) and \(s_j\) are not connected in \(C_a\) then the edge \(vv' \notin Y_\tau(1)\) (as \(a \cup s_i s_j\) is not a clique in \(C(\Gamma,S)\)). 

Finally, let \(s_i s_j \in C_a(1)\), i.e. \(a \cup s_i s_j\) is a clique in \(C(\Gamma,S)\). Denote by \(s_{i,j} \in S\) the element so that \(s_i s_{i,j} = s_j\) (there is such an element since \(a \cup s_i s_j\) is a clique in \(C(\Gamma,S)\)). Then for every \(v v' \in X_\tau(1)\) so that \(v \in S_i, v' \in S_j\), the edge \(vv' \in Y_\tau(1)\) if and only if \(f(vv')=s_{i,j}\). By \pref{obs:satisfaction} this is the only possible assignment for \(vv'\) so that \(\tau \cup vv'\) is satisfied by \(f\), and since \(\tau \cup vv'\) is a \(d\)-dimensional face, it is sampled into \(Y\) if and only if it is satisfied. 

To summarize, \(Y_\tau\) is obtained by vertex removal, partitioning the remaining vertices to random sets \(S_i\), and for every pair of sets \(S_i,S_j\) so that \(s_i, s_j\) are in the link of \(a\), random sub-sampling of the edges between \(S_i,S_j\). From this explanation we can conclude the following:
\begin{claim} \label{claim:link-coloring}
There is a \(C_a\)-coloring of \(Y_\tau\), \(\psi_\tau :Y_\tau(0) \to C_a(0),\; \psi_\tau(v)=f(u_0 v).\) \(\qed\)
\end{claim}


\subsection{Proof of \pref{lem:good-expander}} \label{sec:proof-of-good-expander-lemma}
We shall use the following assertions proven in \pref{sec:expander-sparsification}.
\restateclaim{claim:graph-composition}
\begin{proposition} \torestate{\label{prop:part-of-dense-expander}
Let \(G = (V,E)\) be a \(\lambda\)-two-sided spectral expander over \(n\) vertices for \(\lambda \leq 0.0001\). Let \(r>1\) be so that every edge has weight between \(\frac{1}{r|E|}\) and \(\frac{r}{|E|}\). Assume further that every vertex has weight between \(\frac{1}{r|V|}\) and \(\frac{r}{|V|}\). Let \(p \in (0,\frac{1}{2})\). Consider the sampling of a random bipartite graph \(H=(A,B,E')\) as follows:
\begin{enumerate}
    \item Sample a set \(A \subseteq V\) by adding every \(v \in V\) to \(A\) independently with probability \(p\).
    \item Sample a set \(B \subseteq V \setminus A\) by adding every \(v \in V \setminus A\) to \(B\) independently with probability \(\frac{p}{1-p}\).
    \item take the edges of \(H\) to be \(E'=E(A,B)\).
\end{enumerate}
Then \(H\) is a \(\frac{100}{p^{3}}\lambda\)-bipartite spectral expander with probability \(\geq 3ne^{-0.2 \lambda^2 p^2 r^{-2}D} \), where \(D\) is the minimal degree of a vertex in \(G\).} 
\end{proposition}

\begin{lemma} \torestate{\label{lem:second-step}
Let \(p \in (0,1)\). Let \(H = (A,B,E)\) be an \(n\)-vertex bipartite graph so that every edge has weight between \(\frac{1}{r|E|}\) and \(\frac{r}{|E|}\) for some \(r>1\). Let \(0< \varepsilon <0.05\) and \(0<\lambda <(1-0.25\varepsilon)\varepsilon^3p^3\). Denote by \(\mu = \badmu\). Let \(D \in \NN\) so that \(\mu D  > 2\log n + 5\).

Assume that  \(\min_{v \in A \cup B} deg(v) \geq D\) and that \(H\) is a \(\lambda\)-spectral expander. Let \(H'=(A,B,E')\) be a random subgraph where every edge is sampled in \(H'\) with probability \(p\).

Then with probability at least \(1- \badprob\), \(H'\) is a \(260 \varepsilon(1+\ln(3/\varepsilon))\)-spectral expander.}
\end{lemma}

\begin{proof}[Proof of \pref{lem:good-expander}]
Recall that by \pref{claim:link-coloring} there is a \(C_a\) coloring of \(Y_\tau\), \(\\psi_\tau(v) = f(u_0 v)\). 

The probability that \(AT(\tau)\) or \(AT(\tau \cup \set{v})\) occurs for some \(v \in X_\tau(0)\) is at most \((|X_\tau(0)|+1)m^{d} e^{-0.03m^{-d}(r-1)^2 c \log Q}\) (by \pref{claim:face-not-appearing-prob}). When all these events do not occur, the \(C_a\) coloring is non-degenerate; every possible label \(s_i \in C_a\) is given to a vertex (due to \(\neg AT(\tau)\)), and for a vertex \(v\) labeled \(s_i\), every possible edge also appears (due to \(\neg AT(\tau \cup \set{v}\)). Thus we can use \pref{claim:graph-composition}.

By assumption \(C_a\) is a \(\frac{\lambda}{2}\)-spectral expander, so we need to verify that for every edge \(s_i s_j \in C_a\), the bipartite subgraph graph between elements labeled \(s_i\) and those labeled \(s_j\) is also an \(\frac{\lambda}{2}\) bipartite expander.
%

There are at most \(m^2\) edges in the link of \(Cay(\Gamma,S)\). For each edge \(s_i s_j\), we bound the probability that the induced bipartite subgraph between \(S_i, S_j\) is \emph{not} a \(\frac{\lambda}{2}\)-bipartite expander. If we show that this probability is bounded by \(8m|X_\tau(0)|e^{\log |X_\tau(0)|-\mu c \log Q}\) then by a union bound over the \(m^2\) edge we can conclude. Denote by \(s_{i,j}\) the generator so that \(s_i s_{i,j} (s_j)^{-1}=e\).

By \pref{prop:part-of-dense-expander} the bipartite graph between the \(S_i\) and the preimage of \(S_j\) \emph{in \(X_\tau\)} is not a \(O(\eta/m^{1.5d^2})\)-expander with probability \(\leq 3|X_\tau(0)|e^{-0.2\lambda^2r^{-2}c \log Q}\). By \pref{lem:second-step}, whenever the graph between \(S_i, S_j\) in \emph{in \(X_\tau\)} is a \(O(\eta/m^{1.5d^2})\)-expander, the graph in \(Y_\tau\) is not a \(\frac{\lambda}{2}\)-expander with probability at most \(3|X_\tau(0)|^2 e^{\mu c \log Q - log |X_\tau(0)| }\) for some \(\mu>0\). Summing up the probability that the said events mentioned, and taking union bounds over all \(m^2\) subgraphs induced by \(s_i s_j\) we get the desired bound.
\end{proof}

\subsection{Proof of \pref{claim:not-at-implies-measure-behaves-nicely-in-links}} \label{sec:proof-of-claim-not-at-implies-measure-behaves-nicely-in-links}
\restateclaim{claim:not-at-implies-measure-behaves-nicely-in-links}

For \(\ell \leq d\) and \(\bar{s} = (s_{i_0,i_1},s_{i_0,i_2},...,s_{i_{\ell-1},i_{\ell}})\). We denote by \(s_{\leq j}\) the tuple restricted to indexes less or equal to \(j\) and \(s_{>j}\) to be the rest of the indexes. Let 
\[T(X,\bar{s}) = \sett{(v_0,v_1,...,v_\ell) \in \dir{X}(\ell)}{\forall j_1,j_2 \; \dir{f}(v_{j_1} v_{j_2}) = s_{j_1,j_2}}.\]
For a tuple \(\bar{u}=(v_0,v_1,...,v_j)\) we also denote by
\[T(X,\bar{s},\bar{u}) = \sett{(v_{j+1},v_{j+2},...,v_\ell)}{(v_0,v_1,...,v_j,v_{j+1},v_{j+2},...,v_\ell) \in T(X,\bar{s})}.\]

Let \(\bar{u} \in \dir{X}(j)\) and let \(\bar{s} = (s_{i_0,i_1},s_{i_0,i_2},...,s_{i_{\ell},i_{\ell+1}})\) be so that \(s_{\leq \ell}\) is the edge labeling of \(\bar{u}\) (i.e. \(f(u_{j_1} u_{j_2}) = s_{i_{j_1}, i_{j_2}}\)). The \(\neg AT(\bar{u})\) implies that
\[\Prob[X_u(0)]{T(X,\bar{s},\bar{u})} \approx m^{-(j+1)}.\]
The next claim is a generalization of this property.

\begin{proposition} \label{prop:labeling-is-approximately-uniform}
Let \(-1 \leq \ell,j \) be integers so that \(\ell+j \leq d-1\). Let \(t = (v_0,v_1,...,v_\ell) \in \dir{X}(\ell)\) and \(\bar{s} = (s_{i_0,i_1},...,s_{i_{\ell+j},i_{\ell+j+1}})\) so that \(t \in T(X,\bar{s}_{\leq \ell})\). Then
\begin{equation} \label{eq:face-fraction}
    r^{-3d} m^{\binom{\ell+1}{2} -\binom{\ell+j+2}{2}} \leq \Prob[\dir{X}_{t}(j)]{T(X,\bar{s},t)} \leq r^{3d} m^{\binom{\ell+1}{2} -\binom{\ell+j+2}{2}}.
\end{equation}
\end{proposition}
This proposition says that whenever we condition on a face having some specific labeling on the edges \(\bar{s}\) then this is (approximately) multiplying by a constant independent of \(\bar{s}\).

\begin{proof}[Proof of \pref{claim:not-at-implies-measure-behaves-nicely-in-links}]
For a face \(\tau =(v_0,v_1,...,v_\ell) \in X(\ell) \) we denote by \(\bar{s}_\tau = (s_{i_0,i_1},...,s_{i_{\ell-1},i_\ell})\) be so that \(\tau \in T(X,\bar{s}_\tau)\). We denote by \(\hat{\sigma} = (s_{\tau,i_0, i_1},...,s_{\tau,i_0,i_j})\). When \(\tau\) is satisfied we denote by \(s_{\hat{\tau}} = s_\tau\); this is well defined since when \(\tau\) is satisfied, \(s_\tau\) is determined by \(\hat{\tau}\) since 
\[s_{i_{j_1},i_{j_2}} = f(v_{j_1} v_{j_2}) \overset{satisfaction}{=} f(v_0 v_{j_1})^{-1} f(v_0 v_{j_2}) = s_{i_0 i_{j_1}}^{-1} s_{i_0, i_{j_2}}.\]

For a tuple \(\rho = (u_0,u_1,...,u_j)\) we denote \(\tau \circ \rho = (v_0,v_1,...,v_\ell,u_0,u_1,...,u_j)\) the concatenation operation. 

For the first item, fix some \(\sigma = (v_0,v_1,...,v_j) \in \dir{X}(j)\) that is a satisfied face. Recall that by \pref{obs:satisfaction} \(\hat{\sigma} \in C_e(j-1)\). The probability of \(\sigma\) is
\begin{equation}
    \begin{aligned}
        \Prob[Y]{\sigma} & = \sum_{\tau \in \dir{Y}(d), \tau = \sigma \circ \rho} \Prob[\dir{Y}(d)]{\tau} =  \sum_{\tau \in \dir{Y}(d), \tau =\sigma \circ \rho} \Prob[C]{\hat{\tau}} \frac{\Prob[\dir{X}(d)]{\tau}}{\Prob[\dir{X}(d)]{T(X,s_\tau)}} \\
        &= \sum_{\hat{\tau} = \hat{\sigma} \circ \hat{\rho}}\frac{\Prob[C]{\hat{\tau}}}{\Prob[\dir{X}(d)]{T(X,s_\tau)}} \sum_{\tau' = \sigma \circ \rho, \; s_{\tau'} = s_{\hat{\tau}}} \Prob[X]{\tau'}\\
        & = \sum_{\hat{\tau} = \hat{\sigma} \circ \hat{\rho}} \frac{\Prob[C]{\hat{\tau}} \cdot \Prob[\dir{X}(j)]{\sigma} \Prob[\dir{X}_\sigma(d-\ell-1)]{T(X,s_\tau,\sigma)}}{\Prob[\dir{X}(d)]{T(X,s_\tau)}}.
    \end{aligned}
\end{equation}

By utilizing \pref{prop:labeling-is-approximately-uniform} on \(\Prob[\dir{X}_\sigma(d-\ell-1)]{T(X,s_\tau,\sigma)}\) and on \(\Prob[\dir{X}(d)]{T(X,s_\tau)}\) we have that
\begin{equation} \label{eq:bound-on-prob-of-face}
    r^{-6d} m^{\binom{d+1}{2} - \binom{\ell+1}{2}} \cdot \Prob[\dir{C}_e(j)]{\hat{\sigma}} \cdot \Prob[\dir{X}(j)]{\sigma} \leq \Prob[Y]{\sigma} \leq r^{6d} m^{\binom{d+1}{2} - \binom{\ell+1}{2}} \cdot \Prob[\dir{C}_e(j)]{\hat{\sigma}} \cdot \Prob[\dir{X}(j)]{\sigma}.
\end{equation}

By  \pref{prop:labeling-is-approximately-uniform} 
\begin{equation} \label{eq:relation-of-T-X-in-links}
r^{-3d} \Prob[X]{T(X,s_\sigma)}^{-1} \leq m^{\binom{d+1}{2} - \binom{\ell+1}{2}} \leq r^{3d} \Prob[X]{T(X,s_\sigma)}^{-1}.
\end{equation}
 We plug \eqref{eq:relation-of-T-X-in-links} in \eqref{eq:bound-on-prob-of-face} and use the fact that 
\(\Prob[ \dir{X}(j) ]{\sigma}\cdot \Prob[X]{T(X,s_\sigma)}^{-1} = \cProb{\dir{X}(j)}{\sigma}{s_\sigma}\) to get the first item.
\bigskip

For the second item, we note that for oriented faces \(\sigma \in \dir{X}(j)\) and \(\tau \in \dir{X}_\sigma(\ell)\), it holds that
\[\Prob[Y_\sigma]{\tau} = \frac{\Prob[Y]{\sigma \circ \tau}}{\Prob[Y]{\sigma}} \overset{\eqref{eq:bound-on-prob-of-face}}{\leq} r^{12d} \frac{\Prob[\dir{C}_e]{\widehat{\sigma \circ \tau} }}{\Prob[\dir{C}_e]{\hat{\sigma}}} \cdot \frac{\Prob[\dir{X}(j+\ell+1)]{\tau \circ \sigma}}{\Prob[\dir{X}(j)]{\sigma}} m^{\binom{\ell+j+2}{2}-\binom{\ell+1}{2}} \leq \]
\[r^{15d} \frac{\Prob[\dir{C}_e]{\widehat{\sigma \circ \tau} }}{\Prob[\dir{C}_e]{\hat{\sigma}}} \cdot \frac{\Prob[\dir{X}(j+\ell+1)]{\tau \circ \sigma}}{\Prob[\dir{X}(j)]{\sigma} \Prob[X_\sigma]{T(X,s_{\sigma \circ \tau})}} \leq\]
\[r^{15d} \Prob[\dir{C}_{e \cup \hat{\sigma}}]{\hat{\tau}} \cProb{X_\sigma(\ell)}{\tau}{s_\tau}. \]
Where the second inequality is because \pref{prop:labeling-is-approximately-uniform} implies that 
\begin{equation} \label{eq:relation-of-T-X-in-links-2}
r^{-3d} \Prob[X_\sigma]{T(X,s_{\sigma \circ \tau})}^{-1} \leq m^{\binom{\ell+j+2}{2} - \binom{\ell+1}{2}} \leq r^{3d} \Prob[X]{T(X,s_{\sigma \circ \tau})}^{-1}.
\end{equation}
Thus \[\frac{\Prob[Y_\sigma]{\tau}}{\Prob[\dir{C}_{e \cup \hat{\sigma}}]{\hat{\tau}} \cProb{X_\sigma(\ell)}{\tau}{s_\tau}} \leq r^{15d}.\]
The symmetric inequalities in the other direction also give us 
\[\frac{\Prob[Y_\sigma]{\tau}}{\Prob[\dir{C}_{e \cup \hat{\sigma}}]{\hat{\tau}} \cProb{X_\sigma(\ell)}{\tau}{s_\tau}} \geq r^{-15d}.\]

In particular, for \(\tau = uv \in \dir{X}_\sigma(1)\) or \(\tau = v \in X_\sigma(0)\), we have by the definition of \(G_\sigma\) that
\[\Prob[\dir{C}_{e \cup \hat{\sigma}}]{\hat{\tau}} \cProb{X_\sigma}{\tau}{s_\tau} = \Prob[G_\sigma]{\tau}.\]
Thus item 2 holds.
\end{proof}

We turn to proving \pref{prop:labeling-is-approximately-uniform}. Although the proof is heavy with notation, the intuitive idea is simple. For every \(\ell-1\)-face \(\sigma = (v_0,v_1,...,v_{\ell-1})\), and every assignment \((s_{i_0, i_\ell}, s_{i_1, i_\ell},...,s_{i_{d-1}, i_\ell})\) the fraction of \(\ell\)-faces \((v_0,v_1,...,v_{\ell-1},u)\) so that \(f(v_j u) =s_{i_j, i_\ell}\) is approximately \(m^{-\ell}\) since \(\neg AT\) holds for \(\sigma\). 

Thus it also holds that for an assignment \((s_{i_0, i_\ell}, s_{i_1, i_\ell},...,s_{i_{d-1}, i_\ell}, s_{i_0, i_{\ell+1}}, s_{i_1, i_{\ell+1}},...,s_{i_\ell, i_{\ell+1}})\), that the fraction of \((\ell+1)\)-faces \((v_0,v_1,...,v_{\ell-1}, u, w)\) so that \(f(v_j u) =s_{i_j, i_\ell}\) \emph{and} \(f(v_j w) =s_{i_j, i_{\ell+1}}\) is a \(m^{-(\ell+1)}\)-fraction of the \(\ell\)-faces \((v_0,v_1,...,v_{\ell-1},u)\) so that \(f(v_j u) =s_{i_j, i_\ell}\) (since for every fixed \((v_0,v_1,...,v_{\ell-1},u)\) \(\neg AT\) also holds). This argument propagates up to the \(\ell+k\)-faces, faces for every \(k\) which is the content of the proposition.

\begin{proof}[Proof of \pref{prop:labeling-is-approximately-uniform}]
Let us first prove the proposition assuming \(\ell + j = d-1\), and show the general case below. Let \(\bar{s}\) be a tuple of \(\binom{d+1}{2}\) indexes as in the proposition. Recall that \(\bar{s}_{\leq \ell}\) be the restriction of \(\bar{s}\) to \(s_{j_1,j_2}\) so that \(j_1,j_2 \leq \ell\).

We prove the right-hand inequality (the proof for left-hand inequality is the same up to flipping the inequalities). We will show by induction on \(k\leq j-2\) that
\begin{equation} \label{eq:face-fraction-inductive}
    \Prob[\dir{X}_\sigma(d)]{T(X,\bar{s},\sigma)} \leq r^{2(k+1)} m^{\binom{\ell+1}{2} - \binom{\ell+k+2}{2}}\max_{\bar{v}=(v_0,v_1,...,v_k) \in T(X,s_{\leq \ell+k+1})} \; \; \Prob[\dir{X}_{\bar{v}}(d-\ell-k-1)]{T(X,\bar{s},\sigma \circ \bar{v})}
\end{equation}
where we recall that \(\sigma \circ \bar{v}\) is concatenation of the two tuples.
Given \eqref{eq:face-fraction-inductive} we note that for every \(\bar{v} = (v_0,v_1,...,v_{d-1}) \in T(X,s_{\leq d-1})\), the probability \(\Prob[\dir{X}_{\bar{v}}(0)]{T(X,\bar{s},\bar{v})}\) is at most \(r^2 m^{-d-1}\) by \(\neg AT(\bar{v})\) and \eqref{eq:face-fraction} follows (for \(\ell + j = d-1\)) .

Indeed, for the base case \(k=0\), we note that 
\begin{equation*}
\begin{aligned}
\Prob[\dir{X}(d)]{T(X,\bar{s})} &= \Ex[v_0 \in X_\sigma(0)]{\Prob[\dir{X}_{\sigma \circ v_0}(d-1)]{T(X,\bar{s},\sigma \circ v_0)}} \\
&\leq r^2 m^{-(\ell+1)} \max_{v_0 \in T(X,\bar{s_{\leq \ell+1}},\sigma)} \Prob[\dir{X}_{v_0}(d-\ell-1)]{T(X,\bar{s},v_0)}.
\end{aligned}
\end{equation*}
The inequality is because when \(v_0 \notin T(X,\bar{s_{\leq \ell+1}},\sigma)\) then 
\(\Prob[\dir{X}_{\sigma \circ v_0}(d-1)]{T(X,\bar{s},\sigma \circ v_0)} = 0\). On the other hand, \(\neg AT(\sigma)\) implies that the fraction of \(v \in X_\sigma(0)\) so that 
the probability of \(v \in T(X,\bar{s_{\leq \ell+1}},\sigma)\) is \(\leq r^2 m^{-\ell-1}\). We note that when \(\ell=-1\) then we can't use the assumption about \(AT(\emptyset)\), but in this case \(m^{-(\ell+1)}=1\) so the bound still holds.

Assume that \eqref{eq:face-fraction-inductive} holds for \(k\) and we show it for \(k+1\). For any fixed \(\bar{v} = (v_0,v_1,...,v_k) \in \dir{X}_\sigma(k)\), it holds by \(\neg AT(\sigma \circ (v_0,v_1,...,v_\ell))\) that there are at most \(r m^{-k+\ell+1}\)-fraction of \(v_{\ell+k+1} \in X_{\sigma \circ (v_0,v_1,...,v_k)}(0)\) so that \(\dir{f}(v_j v_{\ell+k+1}) = s_{i_j,i_{\ell+k+1}}\). Thus
\begin{equation}
    \begin{aligned}
        \Prob[\dir{X}_{\bar{v}}(d-\ell-k-1)]{T(X,\bar{s}, \sigma \circ \bar{v})} &= \Ex[v_{\ell+k+1} \in \dir{X}_{\bar{v}}(0)]{\Prob[\dir{X}_{\sigma \circ \bar{v} \circ v_{\ell+k+1}}(d-\ell-k-2)]{T(X,\bar{s}, \sigma \circ \bar{v} \circ v_{\ell+k+1}) )}}\\
        &=\Prob[v_{\ell+k+1} \in \dir{X}_{\bar{v}}(0)]{T(X,\bar{s}_{\leq \ell+k+2}, \bar{v})} \cdot\\
        &\Ex[v_{\ell+k+1} \in T(X,\bar{s}_{\leq \ell+k+1}, \sigma \circ \bar{v})]{\Prob[\dir{X}_{\sigma \circ \bar{v} \circ v_{\ell+k+1})}(d-\ell-k-2)]{T(X,\bar{s}, \sigma \circ \bar{v} \circ v_{\ell+1} )}}  \\
        &\overset{\neg AT(\bar{v})}{\leq} r^2m^{-(\ell+k+2)} \Ex[v_{\ell+k+1} \in T(X,\bar{s}_{\leq \ell+k+1}, \sigma \circ \bar{v})]{\Prob[\dir{X}_{\sigma \circ \bar{v} \circ v_{\ell+k+1})}(d-\ell-k-2)]{T(X,\bar{s}, \sigma \circ \bar{v} \circ v_{\ell+1} )}} \\
        &\leq  r^2m^{-(\ell+k+2)} \max_{v_{\ell+k+1} \in T(X,\bar{s}_{\leq \ell+k+1}, \sigma \circ \bar{v})} \Prob[\dir{X}_{\sigma \circ \bar{v} \circ v_{\ell+k+1})}(d-\ell-k-2)]{T(X,\bar{s}, \sigma \circ \bar{v} \circ v_{\ell+1} )}.
    \end{aligned}
\end{equation}
Here the second equality is because whenever \(v_{\ell+k+1} \notin T(X,\bar{s}_{\leq \ell+k+2}, \bar{v})\) then \(T(X,\bar{s}, \sigma \circ \bar{v} \circ v_{\ell+1} ) = \emptyset\). Plugging this back in \eqref{eq:face-fraction-inductive} and using Pascal's identity we have that
\begin{equation*} \begin{aligned}
    \Prob[\dir{X}_\sigma (d)]{T(X,\bar{s},\sigma )} &\leq r^{2(k+1)} m^{\binom{\ell+1}{2} - \binom{\ell+k+2}{2}}\max_{\bar{v} \in T(X,s_{\leq \ell+k+1})} \; \; \Prob[\dir{X}_{\bar{v}}(d-\ell-k-1)]{T(X,\bar{s},\sigma \circ \bar{v})}\\
    &\leq r^{2(k+2)} \max_{\bar{v} \in T(X,s_{\leq k+1}), v_{\ell+k+1} \in T(X,\bar{s}_{\leq \ell+k+1}, \sigma \circ \bar{v})} \Prob[\dir{X}_{\sigma \circ \bar{v} \circ v_{\ell+k+1})}(d-\ell-k-2)]{T(X,\bar{s}, \sigma \circ \bar{v} \circ v_{\ell+1} )}
\end{aligned}\end{equation*}
Note that \(\bar{v} \in T(X,s_{\leq \ell+k+1})\) and \(v_{\ell+1} \in T(X,\bar{s}_{\leq \ell+k+2},\sigma \circ \bar{v})\) if and only if \(\bar{v} \circ v_{\ell+1} \in T(X,\bar{s}_{\leq \ell+k+2})\) so the statement follows.

Finally, we can prove the statement for \(\ell + j < d-1\). For \(\bar{s}\) of length \(\binom{\ell+j'+2}{2}\) let \(R\) be the set of all \(\bar{s'}\) of length \(\binom{d+1}{2}\) so that \(s'_{\leq \ell+j+2} = s\). The size of \(R\) is \(m^{\binom{d+1}{2} - \binom{\ell+j'+2}{2}}\). Then
\[\prob{T(X,s,\sigma)} = \sum_{s' \in R} \prob{(X,s',\sigma)} \leq m^{\binom{d+1}{2} - \binom{\ell+j'+2}{2}} \cdot r^{3d} m^{\binom{\ell+1}{2} - \binom{d+1}{2}},\]
which is what we need to show. The other inequality follows similarly.
\end{proof}

\subsection{Instantiating the construction}
\paragraph{The groups} As stated in \pref{sec:group-complex}, for every \(\lambda > 0\) and \(d > 0\), \cite{LubotzkySV2005b} constructed infinite groups \(\Gamma\) with generating sets \(S\) so that the Cayley complex \(Cay(\Gamma,S)\) is has \(\lambda\)-expanding links. These groups are residually finite, and thus they have infinite finite indexed subgroups \(N_1,N_2,N_3,... \unlhd \Gamma\). It holds that \(S \cap N_i = \emptyset\). Thus for every \(f:Y \to S\) obtained by \pref{thm:main-formal}, the function \(f_i:Y(1) \to \Gamma / N_i,\; f_{N_i}(e)= f(e) N_i\) gives rise to a cover \(Y^{f_i}\) which is finite.

\paragraph{The complexes \(X\)} 
\begin{itemize}
    \item The complete \(d\)-dimensional complex with \(n\) vertices, has all the properties needed from \(X\) in \pref{thm:main-formal} for large enough \(n\). We don't usually think of the complete complex as being bounded-degree. However, recall that \(X\) is a single complex and the sequence of covers grows to infinity in size. 
    \item The complexes defined by \cite{LubotzkySV2005b} or \cite{KaufmanO2018} \emph{don't} suffice for this theorem since the weights in their links are not uniform or close to uniform as required by \pref{thm:main-formal}. However, a the work by Friedgut and Iluz \cite{FriedgutI2020} shows how to transform some high dimensional expanders to high dimensional expanders in which the degree is regular in every link, and the weights are uniform. This work applies to the high dimensional expanders in \cite{KaufmanO2018}.

    \item By tensoring, we can also satisfy the requirement that the links be \emph{locally dense}, i.e. that if the maximal number of \(d\)-faces covering a vertex \(v\) is \(Q\), then every \(1\)-skeleton of a link has a minimal degree of at least \(c \log Q\) (for a large enough \(c\)). If the degree in every link is too low, then instead of \(X\) we can take the complex \(X^{t}\), which is a tensor product between \(X\) and the complete complex over \(t\)-vertices. This complex is explicitly defined as follows:
\begin{enumerate}
    \item \(X^t(0) = \set{1,2,...,t} \times X(0)\).
    \item \(X^t(d) = \)
    \[\sett{\set{(i_0,v_0),(i_1,v_1),(i_2,v_2),...,(i_d,v_d)}}{\set{v_1,...,v_d} \in X(d), \;\set{i_0,i_1,...,i_d} \subseteq \set{1,2,...,t} \text{ of size d+1} }.\]
    That is, the faces are all matchings between sets in \(X(d)\) and subsets of size \(d+1\) of the set \(\set{1,2,...,t}\).
\end{enumerate}
The measure on \(d\) faces is by sampling a \(d\)-face \(\set{v_0,v_1,...,v_d} \in X(d)\), sampling a set of size \(d+1\) from \(\set{i_0,i_1,...,i_d} \subseteq \set{1,2,...,d}\) uniformly, and then sampling a matching \(\set{(i_0,v_0),.,,,(i_d,v_d)}\).

\begin{claim}
\begin{enumerate}
    \item \(X^t\) has \(t \cdot |X(0)|\) vertices.
    \item If the maximal degree (i.e. \(d\)-faces covering a vertex) in \(X\) was \(Q\), then the maximal degree in \(X^t\) is \(d!\binom{t}{d}Q\).
    \item Let \(d-2\)-face \(\sigma' = \set{(i_0,v_0),...,(i_{d-2},v_{d-2})}\) where \(\set{v_0,...,v_{d-2}} = \sigma \in X(d-2)\). The \(1\)-skeleton of its link \(X^{t}_{\sigma'}\) is the tensor product of the link of \(X_\sigma\) and the complete graph over \(t-(d-1)\) vertices. In particular, if \(X_\sigma\) was a \(\lambda\)-spectral expander, then \(X^t_\sigma\) is a \(\lambda\)-spectral expander for a large enough \(t\).
    \item If the faces in \(X\) had uniform weights, then so do the faces in \(X^t\).
    \item The minimal degree in all links grow at least by a multiplicative factor of \(t-(d+1)\). Thus for any \(c > \) there is a large enough \(t\) so that the minimal degree of a vertex in a link greater than \(c \cdot \max_{v \in X^t(0)}deg_d(v)\).
\end{enumerate}
\end{claim}
We leave the verification of these statements to the reader.

\end{itemize}

\section{Combining two simplicial complexes} \label{sec:ramifications}
In this section we construct a random simplicial complex \(Y\) from two given simplicial complexes \(X\) and \(C\). The vertices of \(Y(0) = X(0)\) but the link of every face \(t \in Y(\ell)\) is homomorphic to a link of a face in \(C(\ell)\). Moreover, when \(X,C\) are high dimensional expanders, then \(Y\) is a high dimensional expander as well. The random construction in this section is similar to the one in \pref{sec:main}.

\begin{theorem} \label{thm:second-construction}
Let \(m \in \NN\) and let \(C\) be a \(d\)-dimensional simplicial complex with \(m\) vertices that is a \(\lambda\)-high dimensional expander. There exists universal constants \(r=r(m), c=c(m),\eta=\eta(m)\) so that the following holds. Let \(X\) be a \((c,r,\eta)\)-suitable  \(d\)-dimensional simplicial complex\footnote{as in \pref{def:suitable-complex}}.
Then there exists some \(Y \subseteq X\) so that \(Y\) is an \(d\)-dimensional \(\frac{2\lambda}{1-2\lambda}\)-high dimensional expander, and so that \(Y\) has a \emph{non-degenerate \(C\)-coloring}. \(Y\) contains \(\Omega_m(1)\) of the faces of \(X\). Moreover, there is a randomized algorithm that outputs \(Y\) in expected polynomial time.
\end{theorem}

\begin{remark}
As in the proof of \pref{thm:main-formal}. We show that \(Y\) exists using Lovazs Local Lemma. The randomized algorithm follows directly from the algorithmic version of this theorem by \cite{MoserT2010}. See \pref{sec:preliminaries} for more details.
\end{remark}

We use the techniques we developed in \pref{sec:main}. Let \(f:X(0) \to C(0)\). We say a face \(s=\set{v_0,v_1,...,v_\ell} \in X(\ell)\) is satisfied by \(f\) if \(f(s) = \set{f(v_0),f(v_1),...,f(v_\ell)} \in C(\ell)\). The \(f\)-pruning of \(X\) is the simplicial complex \(Y\) whose vertices are \(Y(0)=X(0)\) and whose \(d\)-faces are all \(s \in X(d)\) that are satisfied by \(f\). By construction \(f:Y(0)\to C(0)\) is a simplicial homomorphism from \(Y\) to \(C\).

Consider the following random process:
\begin{enumerate} 
    \item Sample a function \(f:X(0) \to Y(0)\) uniformly at random.
    \item Let \(Y=Y_f\) be the \(f\)-pruning of \(X\).
    \item If \(f:Y(0)\to C(0)\) is not a non-degenerate \(C\)-coloring we output \(FAIL\). Otherwise we output \(Y\).
\end{enumerate}
We denote the distribution defined by this process \(D(C,X)\).

The measure of \(Y\) is the coloring measure defined in \pref{sec:simplicial-homomorphism}. Namely, to sample a \(d\)-face \(t \in Y(d)\):
\begin{enumerate}
    \item We sample \(s \in C(d)\).
    \item We sample \(t \in Y(d)\) given that \(f(t)=s\).
\end{enumerate} 
In order to prove \pref{thm:second-construction}, we show that this process has a positive probability of sampling a high dimensional-expander \(Y\) with the properties described above via the Lovasz Local Lemma.

\begin{proof}[Proof of \pref{thm:second-construction}]
As in the proof of \pref{thm:main-formal} we define the satisfaction graph. For \(\ell \leq d-2\) and a satisfied face \(\sigma \in X(\ell)\), the \emph{satisfaction graph} of \(\sigma\), denoted by \(G_\sigma = (V_\sigma, E_\sigma)\) is the following:
\begin{itemize}
    \item The vertices are all \(v \in X_\sigma(0)\) so that \(\sigma \dunion v\) is satisfied.
    \item The edges are all \(uv \in X_\sigma(1)\) so that \(uw \dunion \sigma\) is satisfied.
\end{itemize}
The link of \(\sigma\) has a \(C_{f(\sigma)}\) coloring. When it is non-degenerate this also induces a measure on \(G_\sigma\).

We define the following events for faces \(\tau \in X\):
\begin{itemize}
\item[*] \(NE(\tau)\), the event where the \(1\)-skeleton of \(G_\tau\) is not a \(\lambda\)-spectral expander.
%
\item[*] \(AC(\tau)\), for a face \(\tau=\set{u_0,u_1,...,u_{\ell}} \in X(\ell)\), the event that \(\tau\) is satisfied by \(f\), colored by \(a=\set{f(u_0),...,f(u_\ell)}\) and there is some \(b \in C_{a}(0)\) such that no \(u \in X_\tau(0)\) is has color \(f(u)=b\) (MC stands for ``atypical color'').
\end{itemize}
%

These events allow us to define \(\sett{E_\tau}{\tau \in X \setminus X(d)}\):
\begin{align}
E_\tau = \begin{cases} 
    AC(\tau) & \tau \in X(d-1) \\
    AC(\tau) \cup NE(\tau) & \tau \in X(j), j \leq d-1 \\
\end{cases}.
\end{align}

Finally let \(\mathcal{G} = \bigwedge_{\tau \in X \setminus X(d)} \neg E_\tau\). We claim the following:
\begin{claim}\label{claim:G-implies-Y-is-good-ramifications}
When \(\mathcal{G}\) occurs, then 
\begin{enumerate}
    \item \(Y\) is an \(d\)-dimensional \(\frac{2\lambda}{1-\lambda}\)-high dimensional expander.
    \item \(Y\) contains \(\Omega_{m,d}(1)\) fraction of the faces of \(X\).
    \item \(f:Y(0) \to C\) is a non-degenerate \(C\)-coloring.
\end{enumerate} 
\end{claim}

\begin{claim} \label{claim:G-occurs-with-positive-probability-ramifications}
The event \(\mathcal{G}\) occurs with positive probability.
\end{claim}

The proof is done given these claims.
\end{proof}

\begin{proof}[Proof sketch of \pref{claim:G-implies-Y-is-good-ramifications}]
\begin{enumerate}
    \item \(Y\) is an \(d\)-dimensional \(\frac{2\lambda}{1-2\lambda}\)-high dimensional expander.

    Similar to the proof of \pref{thm:main-formal}, the fact that \(\neg AT(\tau)\) occurs for all \(\tau\) promises us that the \(1\)-skeletons of all proper links \(Y_\tau\) are the satisfaction graphs as sets, and that their measure is similar to the measure induced by the coloring. \(NE(\tau)\) promises us that \(G_\tau\) are \(\lambda\)-spectral expanders. Hence (similar to the argument in \pref{thm:main-formal}) it holds that all proper links are \(2\lambda\)-two sided spectral expanders.
    
    To show that the underlying graph itself in a \(\frac{2\lambda}{1-2\lambda}\) we will use \pref{thm:Oppenheim-trickling-down-lemma}. To use it we need to show that the \(1\)-skeleton of \(Y\) itself is connected.

    First we note that similar to \pref{thm:main-formal}, whenever \(\neg AC(\tau)\) holds for all \(\tau\), then a face is in \(Y\) if and only if it is satisfied.

    To show connectivity it is enough to show that for every edge \(vw \in X(1)\) there is a path from \(v\) to \(w\) in \(Y\). If \(vw \in Y(1)\) we are done. Otherwise, this implies (by \(\neg AC(vw)\)) that \(vw\) isn't satisfied by \(f\). This implies that \(d(f(v),f(w))= p > 1\). By \(\neg AT(vw)\), there is some \(u\) so that \(vu \in X(1)\), \(vu\) is satisfied by \(f\) (i.e. \(d(f(v),f(u))=1\) as \(f(u)f(v) \in \dir{C}(1)\)) and \(d(f(u),f(w))=p-1\). In addition, \(u,w\) are neighbours in \(X\). We take \(u\) to be the next vertex in our path, and proceed in this process until we reach a vertex \(u'\) so that \(d(f(u'),f(w))=1\). This vertex is connected to \(w\) as required.

    \item \(Y\) contains \(\Omega_{m,d}(1)\) fraction of the faces of \(X\): this follows directly from \(\neg AC(\tau)\) as in \pref{thm:main-formal}.

    \item \(f:Y(0) \to C(0)\) is a non-degenerate \(C\)-coloring:

    We first explain why \(f\) is surjective. Let \(x \in C(0)\) and let \(a=f(v)\) for an arbitrary vertex \(v\). By \(\neg AC(v)\) for every neighbour \(b \tilde a\) in \(C\), there is a neighbour of \(u \tilde v\) colored by \(b\). Thus by taking a path \((a=r_0,r_1,...,r_k=x)\) we can inductively find vertices colored by \(r_1,r_2,...,r_k=x\).

    Now for a \(d\)-face \(\set{a_0,...,a_d}\) we take a vertex \(v_0\) so that \(f(v_0)=a_0\). By \(\neg AC(v_0)\), \(v_0\) has a neighbour \(v_1 \in X_{v_0}(0)\) so that \(f(v_1)=a_1\). By \(\neg MC(v_0 v_1)\), there is a vertex \(v_2 \in X_{v_0 v_1}(0)\) so that \(f(v_2) = a_2\). Note that this implies that \(\set{v_0, v_1, v_2}\) is satisfied by \(f\). Continuing this process shows there is a face \(\set{v_0,v_1,...,v_d} \in X(d)\) that is colored by \(f(v_0)=a_0,...,f(v_d)=a_d\) as required.
\end{enumerate}
\end{proof}

\begin{proof}[Proof of \pref{claim:G-occurs-with-positive-probability-ramifications}]
To show that \(\mathcal{G}\) has positive probability we use Lovasz Local Lemma \pref{thm:l3}. 

Similar to the proof of \pref{thm:main-formal}, it is easy to verify that for every \(\tau,\tau' \in X\), if \(E_\tau, E_{\tau'}\) are not independent, then there is a path of length at-most \(2\) from a vertex in \(\tau\) to a vertex in \(\tau'\).

Thus, denoting \(Q\) the maximal \(d\)-degree of a vertex in \(X\), for every \(\tau \in X\), the event \(E_\tau\) depends only on \(K=poly(Q)\) events \(E_{\tau'}\) (the polynomial depends on \(d\)).

If we show that for all \(\tau \in X \setminus X(d)\) it holds that
\[ \prob{E_\tau}<\frac{1}{(K+1)e}\]
then by the Lovazs Local Lemma, we will be promised that \(\mathcal{G}\) occurs with some positive probability.

To use Lovazs Local Lemma, we need to upper bound the probabilities of the events \(NE(\tau),AC(\tau)\) by \(\frac{1}{2(K+1)e}\). The bounds on \(NE(\tau),AC(\tau)\) are the same as in the proof of  \pref{thm:main-formal}. We leave the details for the reader.

%
%
%

The Theorem follows.
\end{proof}

\section{Deramdomizing the constructions} \label{sec:derandomization}
In this section we show that when \(X\) has bounded degree then we have deterministic algorithmic versions for \pref{thm:main-formal} and \pref{thm:second-construction}.

To do so, we use the work of Chandrasekaran, Goyal and Haeupler \cite{ChandrasekaranGH2013} that gave a derandomization to the for Lovazs Local Lemma's algorithm. We cite a less general result than the result in their paper that suits our needs.
\begin{theorem}[\cite{ChandrasekaranGH2013}, Theorem 5] \label{thm:l3-deterministic}
Let \(X_1,...,X_p\) be independent random variables over a finite domain of size \(m\). Let \(E_1,E_2,...,E_n\) be events determined by these random variables. For every \(E_i\) we denote by \(A_i = \sett{j}{E_i,E_j \text{ are not independent}}\).
\begin{enumerate}
    \item There is some constant \(k\) so that for every \(i=1,2,...,n\), \(|A_i| \leq k\).
    \item Every \(E_n\) depends on at most \(r\) random variables.
    \item For every partial assignment \(\set{X_i=s_i}_{i \in I}\) and every \(E_j\), there is an algorithm that runs in time \(T\) and calculates \(\cProb{}{E_j}{\set{X_i=s_i}_{i \in I}}\).
    \item There exists some \(\varepsilon > 0\) so that
    \[ \prob{E_i} \leq  \left ( \frac{1}{e(k+1)}\right )^{1+\varepsilon}.\] 
\end{enumerate}
Then there exists a deterministic algorithm that finds assignments to \(X_1,X_2,...,X_d\) so that all the events \(E_1,E_2,...,E_n\) do not occur. The algorithm runs in time \(O \left (T m  (p+n+r+k)^{3+2/\varepsilon} \varepsilon^{-2} \right ) \).
\end{theorem}

We will use this theorem on simplicial complexes \(X\). We will think of \(p,n\leq |X|\), \(r = O(\log |X|)\), \(k = \poly \log|X|\), and \(T = \exp r=O(|X|^{O(\log m)})\). \(\varepsilon > 0\) will be some constant bounded away from \(0\) and \(m\) (the domain of the \(X_i\)'s) will be some large constant. In this regime of parameters the run time simplifies to \(\poly(|X|)\).

Using \pref{thm:l3-deterministic} we can show the following variant of \pref{thm:main-formal}.
\begin{theorem}[Derandomization of \pref{thm:main-formal}] \label{thm:main-formal-deterministic}
For every pair of integers \(d,m\) there exists some \(c,r > 1, \eta>0\) so that the following holds. Let \(\Gamma\) be a group with a generating set \(S\) of size \(m\). Assume that \(C(\Gamma,S)\) is a \(d\)-dimensional \(\frac{\lambda}{2}\)-high dimensional expander for some \(\lambda < 1\). Let \(X\) be a \((c,r,\eta)\)-suitable \(d\)-dimensional simplicial complex\footnote{as in \pref{def:suitable-complex}.}. Then there exists some \(Y \subseteq X\) so that:
\begin{enumerate}
    \item \(Y\) is an \(d\)-dimensional \(\lambda\)-high dimensional expander.
    \item \(Y\) has a connected \(\Gamma\)-cover.
    \item \(Y\) contains \(\Omega_{m,d}(1)\) fraction of the faces of \(X\).
\end{enumerate} 
There is a deterministic algorithm that runs in time polynomial in \(O(m^{O(L)}|X|^{O(1)})\) and outputs \(Y\), where \(L = \max_{v \in X(0)}\abs{X_v}\). The \(O\) notation hides some dependence on \(d\).
In particular, when \(L = O(\log(|X|))\) the algorithm runs in \(\poly(|X|)\) time.
\end{theorem}

\begin{proof}[Proof sketch of \pref{thm:main-formal-deterministic}]
In the proof of \pref{thm:main-formal} we defined every \(\sett{E_\tau}{\tau \in X}\). Our independent random variables were \(\sett{f(e)}{e \in X(1)}\). We calculate \(n,p,m,r,T,k,\varepsilon\) that appear in the statement of \pref{thm:l3-deterministic}:
\begin{enumerate}
    \item[(n)] Obviously \(n = \abs{\sett{E_\tau}{\tau \in X}} = |X|\).
    \item[(p)] \(p = \abs{\sett{f(e)}{e \in X(1)}} = |X(1)| \leq |X|\).
    \item[(m)] \(m=|S|\) (which was also denoted by \(m\) in the theorem statement).
    \item[(r)] One can verify by the definition given to \(E_\tau\), that \(E_\tau\) only depend on edges that are contained in some face, which shares a vertex with \(\tau\). Therefore the number of \(f(e)\) that \(E_\tau\) depends on is at most \(\abs{\tau}\binom{d+1}{2}L = O(L)\) (when we think of \(d\) as a constant).
    \item[(T)] We can calculate the probability \(\cProb{}{E_\tau}{\set{f(e_i)=s_i}_{i \in I}}\) by exhaustively going over all possible assignments the \(r\) edges that \(E_\tau\) depends on, and counting the fraction of assignments that fall into \(E_\tau\). Checking whether some assignment is in \(E_\tau\) amounts to calculating eigenvalues of an adjacency operator or counting the number of faces that got assigned some specific assignment. All possible in \(\poly(|X|)\) time. The number of possible assignments is at most \(m^{r}\) (it is an upper bound that since some values may have already been set by conditioning on some partial assignment). Hence \(T = m^{O(L)}\poly|X|\).
    \item[(k)] We saw in the proof of \pref{claim:G-occurs-with-positive-probability} that the number of \(E_\tau'\) that \(E_\tau\) depends on in \(k = \poly(L)\).
    \item[(\(\varepsilon\))] In the proof of \pref{claim:G-occurs-with-positive-probability} we show that \(\prob{E_\tau} \leq \frac{1}{(k+1)e} = \frac{1}{\poly(L)}\). We do so by showing that \(\prob{E_\tau} \leq \poly(L,m)\cdot \eta^{c \log Q}\) where \(Q\) is the number of \(d+1\)-faces that are contained in the link of \(\tau\), and \(\eta = \eta_{m,d} < 1\) some constant. Then we argue that taking \(c\) to be large enough, we can make this less or equal to \(\prob{E_\tau} \leq \frac{1}{(k+1)e}\). But taking a large enough \(c\) could also make this less or equal to \(\prob{E_\tau} \leq \left (\frac{1}{(k+1)e}\right )^{2}\). So for an appropriate \(c\) we can set \(\varepsilon = 1\).
\end{enumerate}
Thus by \pref{thm:l3-deterministic}, there is a deterministic algorithm that outputs \(Y\) as in the theorem statement in \(O(m^{O(L)}|X|^{O(1)})\).
\end{proof}

\begin{theorem} \label{thm:second-construction-deterministic}
Let \(m \in \NN\) and let \(C\) be a \(d\)-dimensional simplicial complex with \(m\) vertices that is a \(\lambda\)-high dimensional expander. There exists universal constants \(r=r(m), c=c(m),\eta=\eta(m)\) so that the following holds. Let \(X\) be a \((c,r,\eta)\)-suitable \(d\)-dimensional simplicial complex\footnote{as in \pref{def:suitable-complex}.}.
Then there exists some \(Y \subseteq X\) so that \(Y\) is an \(d\)-dimensional \(\frac{2\lambda}{1-2\lambda}\)-high dimensional expander, and so that \(Y\) has a \emph{non-degenerate \(C\)-coloring}. \(Y\) contains \(\Omega_m(1)\) of the faces of \(X\).  Moreover, there is a deterministic algorithm that finds \(Y\) and runs in time \(O(|C|^{O(L)}|X|^{O(1)})\), where \(L = \max_{v \in X(0)}\abs{X_v}\). In particular, when \(L = O(\log (|X|))\) and \(|C| = O(1)\) then the algorithm runs in time \(\poly(|X|)\).
\end{theorem}
The proof of \pref{thm:second-construction-deterministic} follows the same calculation of parameters and is therefore omitted.
\printbibliography
\appendix

\section{Sparsification of expander graphs} \label{sec:expander-sparsification}
In this section, we prove the necessary results we need in order to analyze the edge removal process in the links and prove \pref{lem:good-expander}. 
\subsection{Proof of \pref{claim:graph-composition}}
Recall that for measured graphs \(G = (V,E, \nu_G), H=(V',E', \nu_H)\) and a non-degenerate \(H\)-coloring \(f:V \to V'\) of \(G\), we defined the measure \(\nu_f(vu)\) on the edges of \(G\) by the following process:
\begin{enumerate}
    \item Sample an edge \(ab \in E'\).
    \item Sample an edge \(uv \in E\) given that \(f(u)=a, f(v)=b\).
\end{enumerate}

\restateclaim{claim:graph-composition}
\begin{proof}[Proof of \pref{claim:graph-composition}]
We prove the claim with respect to two sided expansion (i.e. we assume that \(H\) is a two-sided spectral expander, and show that \(G\) is a two sided spectral expander). We comment about adjusting the proof to one sided spectral expanders at the end of the proof. 

Let \(A\) be the normalized adjacency operator of \(G\) with respect to \(\nu_f\). Let \(g,f \bot \one\) be any two functions with norm \(1\). We need to show that 
\begin{equation}\label{eq:composition-inner-product}
\iprod{Af,g} \leq \max \set{\lambda,\eta}.
\end{equation}
Let 
\[W_C = \sett{h:V \to \RR}{\forall v \in V', u \in A_v, h(v)=h(u)}.\]
\[W_0 = \sett{h:V \to \RR}{\forall v \in V', \Ex[u \in A_v]{h(u)}=0}.\]
We note that \(W_C \oplus W_0\) is an orthogonal decomposition of \(L_2(V)\). Furthermore, \(A(W_C) \subseteq W_C, A(W_0) \subseteq W_0\).
Thus we can decompose \(f\) (resp. \(g\)) to \(f=f^{c}+f^{0}\) according to the spaces above, write 
\[ \iprod{Af,g} = \iprod{Af^{c},g^c} + \iprod{Af^0, g^0}.\]
We show \eqref{eq:composition-inner-product} separately for each term. Beginning with the left term, \(\iprod{Af^{c},g^c}\), we can define \(\tilde{f}:V' \to \RR, \tilde{g}:V' \to \RR\) by \(\tilde{f}(a) = f^c(a), \tilde{g}(a)=g^c(a)\) for some \(u \in S_a\). We denote \(A_H\) the normalized adjacency-operator of \(H\) and we have that
\[\iprod{Af^{c},g^c} = \iprod{A_H\tilde{f},\tilde{g}}.\]
The equality above is due to the definition of the weights \(\nu_f\) of \(G\), where in particular \(\nu_f(f^{-1}(a))=\nu_h(a)\).
We note that \(f,g \bot \one\) and that \(f^0,g^0 \bot \one\), hence \(f^c,g^c \bot \one\), which implies that \(\tilde{f},\tilde{g} \bot \one\). Thus by \(\lambda\)-spectral expansion of \(H\)
\[\iprod{A_H\tilde{f},\tilde{g}} \leq \lambda \iprod{\tilde{f},\tilde{g}} = \lambda \iprod{f^c,g^c}.\]

As for the right term, we note that 
\[ \iprod{Af^0, g^0} = \Ex[ab \in E']{\Ex[xy \in E_{ab}]{f^0(x)g^0(y)} } = \Ex[ab \in E']{\iprod{A_{ab}f^0|_{S_a},g^0|_{S_b}}_{ab}}.\]
Here \(A_{ab}\) is the normalized bipartite adjacency operator of the bipartite graph between \(S_a,S_b\). For all \(a,b\in V'\) we have that \(f^0|_{S_a} \bot \one|_{S_b}\) (and resp. \(g\)). Thus from \(\eta\)-bipartite expansion we have that 
\[\Ex[ab \in E']{\iprod{A_{ab}f^0|_{S_a},g^0|_{S_b}}_{ab}} \leq \eta \Ex[ab \in E']{\iprod{f^0|_{S_a},g^0|_{S_b}}_{ab}} = \eta \iprod{f^0,g^0}.\]
The claim follows.

\paragraph{The one-sided case} The same proof applies with the following changes. Instead of showing \eqref{eq:composition-inner-product}, we need to show
\begin{equation}\label{eq:composition-inner-product-1-sided}
    \iprod{A f, f} \leq \max \set{\lambda, \eta}.
\end{equation}
We can still orthogonally decompose
\[\iprod{A f, f} = \iprod{A f^c, f^c} + \iprod{A f^0, f^0}.\]
Showing that \(\iprod{A f^0, f^0} \leq \eta\) is the same as in the two-sided case, since the bound uses the assumptions about the bipartite graphs, not the expansion of \(H\). Bounding \(\iprod{A f^c, f^c} \leq \lambda\) follows by the same argument of the two-sided case, namely, that
\[\iprod{A f^c, f^c} = \iprod{ A_H \tilde{f}, \tilde{f}} \leq \lambda \iprod{\tilde{f},\tilde{f}} = \lambda \iprod{f^c,f^c}.\]
Here the inequality used that \(1\)-sided expansion of \(A_H\).

\end{proof}
\subsection{Proof of \pref{prop:part-of-dense-expander}}
\restateprop{prop:part-of-dense-expander}

Note that the sampling of the vertices in \(H\) is sampling two \emph{disjoint} sets \(A,B \subseteq V\) where each vertex is sampled into one of the sets independently with probability \(p\) (and isn't sampled into \(H'\) with probability \(1-2p\)). 

We also remark that this proposition is interesting when \(D = \poly(\log(n))\), otherwise the bound is trivial. This is somewhat justified though. Observe that the probability that for some vertex \(v \in A\) the probability that none of its neighbours are sampled into \(B\) is \((1-p)^D\). So,for example, if \(D\) was constant, this event would happen with high probability when the number of vertices tends to infinity.

Recall that we denote by \(\nu_G\) the probability measure on the graph of \(G\). Before we prove \pref{prop:part-of-dense-expander}, we assert some properties that \(H\) has with high probability. These events are loosely that:
\begin{enumerate}
    \item \(\nu_G(A) \approx p\).
    \item \(\nu_{H|A}(v) \approx p^{-1} \nu_G(v)\), where \(\nu_{H|A}(v)\) is the weight of \(v\) given that we sampled it in the \(A\) side of \(H\).
\end{enumerate}
\begin{claim} \label{claim:common-events}
Let \(\varepsilon > 0\). Then
\begin{enumerate}
    \item 
    \( \prob{|\nu_G(A) - p| \geq \varepsilon p } \leq 2e^{p^2\varepsilon^2 r^{-2}n}\)
    and similarly for \(B\).
    
    That is, the weight of \(A\) and the weight of \(B\) is approximately \(p\) with high probability.
    \item \( \prob{\forall v \in A \, |\sum_{u \in B}\nu_G(vu)- p \cdot \sum_{u \in V} \nu_G(vu)| < \varepsilon p \cdot \sum_{u \in V} \nu_G(vu)} \leq 2n e^{- \varepsilon^2 p^2 r^{-2} D}.\)

    That is, for every vertex \(v\), approximately a \(p\)-fraction of the its edges go to \(B\) with high probability.
\end{enumerate}
Here the \(\prob{X}\) notation is with respect to the sampling of \(H\).
\end{claim}
We denote by \(\mathcal{E} = \mathcal{E}(\varepsilon)\) the union of the two events that appear in \pref{claim:common-events}.

\begin{proof}[Proof of \pref{prop:part-of-dense-expander}]
Let \(\varepsilon = \min(\lambda^2, 0.0001)\).
The event \(\mathcal{E}\) occurs with probability at most \((2n+2)e^{-\varepsilon^2p^2c^{-2}D}\). We show that when \(\mathcal{E}\) doesn't occur, the resulting graph \(H\) is a \(\frac{100}{p}\lambda\)-bipartite expander. We hereby assume in all claims below that indeed \(\mathcal{E}\) doesn't occur.

For a function \(f:A\to \RR\) or \(g: B \to \RR\) we define \(\tilde{f},\tilde{g}:V \to \RR\) so that \(\tilde{f}|_A = f, \tilde{g}|_B = g\) and \(\tilde{f}|_{V \setminus A}, \tilde{g}|_{V \setminus B} = 0\). 

Fix a function \(f:A \to \RR\) so that \(\iprod{f, \one_A}_{H|A}= \Ex[v \sim \nu_{H|A}]{f(v)}= 0\). We need to show that
\begin{equation} \label{eq:part-of-dense-expander-1}
    \iprod{A_H f,A_H f}_{H|B} \leq \frac{100\lambda^2}{p^3}\norm{f}_{H|A}^2.
\end{equation}

We do the following steps:
\paragraph{Step 1} We relate \(\iprod{A_H f,A_H f}_{H|B} \) to \(\iprod{\widetilde{A_H f},\widetilde{A_H f}}_G\). Namely, \emph{we will prove that:}
\begin{equation} \label{eq:part-of-dense-expander-2}
    \iprod{A_H f,A_H f}_{H|B} \leq  p^{-1}(1+5\varepsilon) \iprod{ \widetilde{ A_H f},\widetilde{ A_H f}}_G.
\end{equation}

\paragraph{Step 2} We relate the inner product of \(\widetilde{A_H f}\) (first using the bipartite adjacency in \(H\), then extending the resulting function from \(B\) to all \(G\)), with the inner product of \(A_G \tilde{f}\) (first extending \(f\) from \(A\) to all \(G\) and then applying \(G\)'s adjacency operator). Namely, \emph{we need to show that:}
\begin{equation} \label{eq:part-of-dense-expander-3}
    p^{-1}(1+5\varepsilon) \iprod{ \widetilde{ A_H f},\widetilde{ A_H f}}_G \leq \frac{1+5\varepsilon}{p^{3}(1-\varepsilon)^2} \iprod{A_G \tilde{f}, A_G \tilde{f}}_G.
\end{equation}

\paragraph{Step 3} We would like to use the expansion of \(G\) to bound \(\iprod{A_G \tilde{f}, A_G \tilde{f}}_G \leq \lambda^2 \iprod{\tilde{f},\tilde{f}}_G\). However, even if \(f \bot \one_A\) it may be the case that \(\tilde{f} \not\bot \one_G\). However, \emph{we will show that}:
\begin{equation} \label{eq:relate-constant-part-of-f-to-its-extension}
    \iprod{\tilde{f}, \one_G}^2 \leq 17 \varepsilon p^2 \iprod{f,f}_{H|A}
\end{equation}
Plugging this in the right-hand side of \eqref{eq:part-of-dense-expander-3} we have that
\begin{equation} \label{eq:part-of-dense-expander-4}
\begin{aligned}
    \frac{1+5\varepsilon}{p^{3}(1-\varepsilon)^2} \iprod{A_G \tilde{f}, A_G \tilde{f}}_G \leq& 
    \frac{1+5\varepsilon}{p^{3}(1-\varepsilon)^2} \iprod{\tilde{f}, \one_G}_G^2 + \frac{1+5\varepsilon}{p^{3}(1-\varepsilon)^2} \lambda^2 \iprod{\tilde{f}, \tilde{f}}_G \leq \\
    & \frac{1+5\varepsilon}{p^3(1-\varepsilon)^2} 17\varepsilon p^2 \iprod{f,f}_{H|A} + \frac{1+5\varepsilon}{p^3(1-\varepsilon)^2} \lambda^2 \iprod{\tilde{f}, \tilde{f}}_G
\end{aligned}
\end{equation}
Finally \emph{we will show that}:
\begin{equation} \label{eq:part-of-dense-expander-5}
\iprod{\tilde{f},\tilde{f}}_G \leq \iprod{f,f}_{H|A}
\end{equation}

Which implies that 
\begin{equation} \label{eq:part-of-dense-expander-6}
    \begin{aligned}
    \frac{1+5\varepsilon}{p^3(1-\varepsilon)^2} 17\varepsilon p^2 \iprod{f,f}_{H|A} + \frac{1+5\varepsilon}{p^3(1-\varepsilon)^2} \lambda^2 \iprod{\tilde{f}, \tilde{f}}_G &\leq \\
    &\frac{(1+5\varepsilon)^2}{p^3(1-\varepsilon)^2}(17\lambda^2 + 16\varepsilon) \norm{f}_{H|A}^2 .
    \end{aligned}
\end{equation}
 By our choice of \(\varepsilon = \min (0.0001,\lambda^2)\) we obtain that \eqref{eq:part-of-dense-expander-1} holds, i.e. \(\iprod{A_H f,A_H f}_{H|B} \leq \frac{100\lambda^2}{p^3}\norm{f}_{H|A}^2.\)

It remains to prove inequalities \eqref{eq:part-of-dense-expander-2}, \eqref{eq:part-of-dense-expander-3}, \eqref{eq:relate-constant-part-of-f-to-its-extension} and \eqref{eq:part-of-dense-expander-5}.

\subsubsection*{Proving the steps}
\paragraph{Inequalities \eqref{eq:part-of-dense-expander-2} and \eqref{eq:part-of-dense-expander-5}} 
The following claim is proven after the proposition. Recall that the expression \(\Ex[v \nu_{H|A}]{g(v)} = \sum_{v \in A}\nu_{H|A}(v) g(v)\), where \(\nu_{H|A}(v) = \sum_{b \in B, v\sim b} \nu_{H}(vb)\), is the weight of \(v\) in the bipartite graph of \(H\) conditioned on sampling \(v \in A\).
\begin{claim} \label{claim:non-negative-function-estimate}
Let \(g:A\to \RR\) be a non-negative function. Then
\begin{equation}\label{eq:positive-functions-estimate-1}
(1-5\varepsilon) p^{-1} \Ex[v \sim \nu_G]{\tilde{f}(v)} \leq \Ex[v \sim \nu_{H|A}]{f(v)}
\end{equation}
and
\begin{equation}\label{eq:positive-functions-estimate-2}
\Ex[v \sim \nu_{H|A}]{f(v)} \leq (1+5\varepsilon)p^{-1}\Ex[v \sim \nu_G]{\tilde{f}(v)}.
\end{equation}

Similar inequalities hold for non-negative \(g:B\to \RR\).
\end{claim}

Recall that 
\[\iprod{A_H f, A_H f}_{H|B} := \Ex[v \sim \nu_{H|B}]{(A_H f(v))^2},\] 
\[\iprod{\widetilde{A_H f}, \widetilde{A_H f}} := \Ex[v \sim \nu_{G}]{(\widetilde{A_H f(v))^2}}.\]
Applying \eqref{eq:positive-functions-estimate-2} of \pref{claim:non-negative-function-estimate} to the function \(g:B\to \RR; \; g(v) = (A_H f(v))^2\) gives
\begin{equation*} 
    \iprod{A_H f,A_H f}_H \leq p^{-1}(1+5\varepsilon) \iprod{ \widetilde{ A_H f},\widetilde{ A_H f}}_G.
\end{equation*}
and \pref{eq:part-of-dense-expander-2} holds.

Similarly, applying \eqref{eq:positive-functions-estimate-1} of \pref{claim:non-negative-function-estimate} to \(g:A\to \RR; \; g(v)=f(v)^2\) gives \eqref{eq:part-of-dense-expander-5}.

\paragraph{Inequality \eqref{eq:part-of-dense-expander-3}}
We bound the pointwise values of \(\widetilde{A_H f}(v)^2\) by \(\frac{1}{4(1-\varepsilon)^2p^2} A_G \tilde{f}(v)^2\).
\begin{claim} \label{claim:comparison-of-adjacencies}
Let \(f:A \to \RR\). Then for every \(v \in B\) it holds that
\begin{gather}\label{eq:A-is-similar}
    \widetilde{A_H f}(v)^2 \leq \frac{1}{(1-\varepsilon)^2p^2} (A_G \tilde{f}(v))^2.
\end{gather}
\end{claim}
The claim is proven after the proposition.
Note that \eqref{eq:A-is-similar} also holds for \(v \notin B\), since the values on the left-hand side of \eqref{eq:A-is-similar} is \(0\) when \(v \notin B\), and the right-hand side is non-negative. Thus we can take expectation on both sides of \eqref{eq:A-is-similar} with respect to \(\nu_G\), and get \eqref{eq:part-of-dense-expander-3}, that is:
\begin{equation*}
    \begin{aligned}
    p^{-1}(1+5\varepsilon)\Ex[v \sim \nu_G]{\widetilde{A_H f(v)}} &=  p^{-1}(1+5\varepsilon)\iprod{ \widetilde{ A_H f},\widetilde{ A_H f}}_G \leq \\
    &\frac{1+5\varepsilon}{p^{3}(1-\varepsilon)^2} \iprod{A_G \tilde{f}, A_G \tilde{f}}_G = \\
    &\frac{1+5\varepsilon}{p^{3}(1-\varepsilon)^2} \Ex[v \in \nu_G]{A_g \tilde{f}(v)^2}.
    \end{aligned}
\end{equation*}

\paragraph{Inequality \eqref{eq:relate-constant-part-of-f-to-its-extension}}

To show \eqref{eq:relate-constant-part-of-f-to-its-extension} we need the following claim on the weights \(\nu_H(v)\).
\begin{claim} \label{claim:probabilities-in-induced-graph}
For every \(v \in A\)
\begin{equation} \label{eq:probabilities-in-induced-graph}
(1-3\varepsilon)\frac{\nu_G(v)}{\nu_G(A)} \leq \nu_{H|A}({v}) \leq (1+3\varepsilon)\frac{\nu_G(v)}{\nu_G(A)}.
\end{equation}
the same holds for every \(u \in B\).
\end{claim}

Let \(f: A\to \RR\) be so that \(\iprod{f, \one_A}_{H|A} = 0\).
\[ \abs{\iprod{\tilde{f},\one_G}_G} = \abs{\Ex[v \in V]{\tilde{f}(v)}} = \nu_G(A) \abs{\sum_{v \in A} \frac{\nu_G(v)}{\nu_G(A)} f(v)} \leq (1+\varepsilon)p \abs{\sum_{v \in A}\frac{\nu_G(v)}{\nu_G(A)} f(v)} \]
By subtracting \(0=\frac{1}{2}\Ex[v \sim \nu_{H | A}]{f(v)} = \sum_{v \in A}\nu_H(v) f(v)\), we get:
\[\leq (1+\varepsilon)p \abs{\sum_{v \in A}\left(\frac{\nu_G(v)}{\nu_G(A)} -\nu_{H|A}(v) \right) f(v)} \]
\[= (1+\varepsilon)p \abs{\sum_{v \in A} \nu_{H|A}(v) \left (\frac{\nu_G(v)}{\nu_H(v) \nu_G(A)} - 1 \right) f(v)}. \]
By \pref{claim:probabilities-in-induced-graph} it holds in particular that \(\left (\frac{\nu_G(v)}{\nu_{H|A}(v) \nu_G(A)} - 1 \right)^2 \leq 16 \varepsilon\) (for small enough \(\varepsilon\)). By Cauchy-Schwartz this is less or equal to
\[ (1+\varepsilon)p \sqrt{\sum_{v \in A}\nu_{H|A}(v) 16\varepsilon} \sqrt{\sum_{v \in A} \nu_{H|A}(v) f(v)^2} \leq 4(1+\varepsilon)\varepsilon^{0.5} p \sqrt{\iprod{f,f}_{H|A}}.\]
In conclusion we have that
\[\iprod{\tilde{f},\one_G}_G^2 \leq 17\varepsilon p^2 \iprod{f,f}_{H|A},\]
for small enough \(\varepsilon\).

\end{proof}

It remains to prove the claims stated in the above proof.
\begin{proof}[Proof of \pref{claim:common-events}]
For the first item, we note that \(\nu_G(A) = \sum_{v \in V}\one_{v \in A} \prob{v}\). This is a sum of \(n\)-independent random variables that are bounded between \(0\) and \(\frac{r}{n}\). Its mean is \(p\). By Hoeffding's inequality, the event that 
\(|\Prob[v\in G]{v \in A} - p| \geq \varepsilon p\) is bounded by 
\(2e^{-2\varepsilon^2p^2c^{-2}n}\).

For the second item, fix  \(v\in A\). The sum of edge weights \(\sum_{u \in B}\nu_G(vu) = \sum_{u \in V} \one_{u \in B} \nu_G(vu)\). This is a sum of \(\geq D\) random variables, that are bounded between \(0\) and \(\frac{r}{|E|}\). Its mean is \(p \cdot \sum_{u \in V} \nu_G(vu)\). By Hoeffding's inequality
\[ \prob{|\sum_{u \in B}\nu_G(vu)- p \cdot \sum_{u \in V} \nu_G(vu)| < \varepsilon p \cdot \sum_{u \in V} \nu_G(vu)} \leq 2 exp \left (-2\varepsilon ^2 p^2 \left (\sum_{u \in V} \nu_G(vu) \right)^2 / \sum_{u \in V} \nu_G(vu)^2 \right ).\]
By the assumption that \(\nu_G(uv) \leq \frac{r}{|E|}\) hence the denominator is bounded by \(\frac{r}{|E|} \sum_{u \in V} \nu_G(vu)\) and 
\[ (*) \leq exp \left (-2\varepsilon ^2 p^2r^{-1} |E| \left (\sum_{u \in V} \nu_G(vu) \right) \right ).\]
As \(\left (\sum_{u \in V} \nu_G(vu) \right) =2 \nu_G(v) \geq \frac{2}{rn}\) this is
\[\leq exp \left (-2\varepsilon ^2 p^2r^{-2} \frac{|E|}{n} \right ).\]
The degree of every vertex is at least \(D\) and in particular \(\frac{|E|}{n} \geq D\), hence
\[ \prob{|\sum_{u \in B}\nu_G(vu)- p \cdot \sum_{u \in V} \nu_G(vu)| \geq \varepsilon p \cdot \sum_{u \in V} \nu_G(vu) } \leq e^{-\varepsilon^2 p^2 r^{-2} D}.\]
Finally, the second item follows from a union bound on \(n\) vertices.
\end{proof}

The proof of \pref{claim:non-negative-function-estimate} relies on \pref{claim:probabilities-in-induced-graph} so we prove \pref{claim:probabilities-in-induced-graph} first.
\begin{proof}[Proof of \pref{claim:probabilities-in-induced-graph}]
For the right-hand inequality in \pref{eq:probabilities-in-induced-graph} observe that
\[\nu_{H|A}(v) = \frac{\sum_{u \in B}\nu_G(vu)}{\sum_{v' \in A}\sum_{u \in B}\nu_G(v'u)}.\]
When \(\mathcal{E}\) doesn't occur, then \(\sum_{u \in B}\nu_G(vu) \leq (1+\varepsilon) p \sum_{u \in V}\nu_G(vu) = 2(1+\varepsilon)p \nu_G(v)\). When \(\mathcal{E}\) doesn't occur, we can upper bound the denominator as well:
\[\sum_{v' \in A}\sum_{u \in B}\nu_G(v'u) \geq \sum_{v' \in A}p(1-\varepsilon)\sum_{u \in V}\nu_G(v'u) =\sum_{v' \in A}2p(1-\varepsilon)\nu_G(v') = 2p(1-\varepsilon)\nu_G(A).\]
Thus \(\nu_H({v}) \leq \frac{2p(1+\varepsilon)}{2p(1-\varepsilon)}\frac{\nu_G(v)}{\nu_G(A)} \leq (1+3\varepsilon)\frac{\nu_G(v)}{\nu_G(A)}\) when \(\varepsilon\) is small enough.
The left-hand inequality follows similarly.
\end{proof}

\begin{proof}[Proof of \pref{claim:non-negative-function-estimate}]
By \pref{claim:probabilities-in-induced-graph} we get that
\[\Ex[v \sim \nu_{H | A}]{f(v)} = \sum_{v \in A}\nu_{H|A}(v) f(v) \overset{\eqref{eq:probabilities-in-induced-graph}}{\leq} \]
\[ \frac{(1+3\varepsilon)}{\nu_G (A)} \sum_{v \in \nu_G}\nu_G(v) \tilde{f}(v) = \frac{(1+3\varepsilon)}{\nu_G (A)} \Ex[v \sim \nu_G]{\tilde{f}(v)} \leq \frac{(1+3\varepsilon)}{p(1-\varepsilon)}\Ex[v \sim G]{\tilde{f}(v)} \leq \frac{(1+5\varepsilon)}{p}\Ex[v \in V]{\tilde{f}(v)}\]
and \eqref{eq:positive-functions-estimate-1} holds. The second inequality is because when \(\mathcal{E}\) doesn't occur then \(\nu_G(A) \geq p(1-\varepsilon)\). The opposite inequality
 \eqref{eq:positive-functions-estimate-2} follows similarly.
\end{proof}

\begin{proof}[Proof of \pref{claim:comparison-of-adjacencies}]
Fix \(v \in B\). We explicitly write down
\[A_G \tilde{f}(v) = \sum_{u \in A}\frac{\nu_G(vu)}{2\nu_G(v)}f(u)\]
and
\[A_H f (v) = \sum_{u \in A}\frac{\nu_H(vu)}{\nu_{H|B}(v)}f(u) = \sum_{u \in A}\frac{\nu_G(vu)}{\nu_{H|B}(v) \cdot K }f(u)\]
where \(K = \sum_{e \in E(A,B)}\nu_G(e') \) is the \(\nu_G\)-weight of edges in \(H\). Thus we need to compare \(\nu_G(v)\) and \(\nu_{H|B}(v) K\).
As
\[ K \nu_{H|B}(v) = K\sum_{u \in A} \nu_H(vu) = K\sum_{u \in A} \frac{1}{K}\nu_G(vu) = \sum_{u \in A} \nu_G(vu).\]
When \(\mathcal{E}\) doesn't occur it holds that
\[K \nu_{H|B}(v) \geq (1-\varepsilon)p \sum_{u \in V}\nu_G(vu) = 2(1-\varepsilon)p \nu_G(v).\]
Hence
\[\left (\widetilde{A_H f}(v)\right )^2 =  (A_H f(v))^2 =  \left (\frac{1}{K \nu_{H|B}(v)} \sum_{u \in A} \nu_G(uv)f(u) \right )^2 \leq \]
\[\left (\frac{1}{(1-\varepsilon)p} \right )^2 \left (\frac{1}{2\nu_G(v)}\sum_{u \in A} \nu_G(uv)f(u) \right )^2  = \frac{1}{(1-\varepsilon)^2p^2} A_G \tilde{f}(v).\]

\end{proof}

\subsection{sparsifying a dense expander}
Let \(H = (V,E)\) be any graph and let \(p \in (0,1)\). We define a random sparsification \(H' = (V,E')\) by independently inserting every edge to \(E'\) with probability \(p\).

We show that if \(H\) is an expander where every vertex has degree \(\Omega(log n)\) and the measure is ``almost'' uniform, then with high probability the graph \(H'\) will also be an expander, albeit with worse parameters.

\restatelemma{lem:second-step}

\begin{proof}[Proof of \pref{lem:second-step}]
Our Lemma shall follow from this claim:
\begin{claim}\label{claim:eml-holds-whp}
With probability at least \(1- \badprob\) over the choice of \(H'\), for every \(S \subseteq A, T \subseteq B\) the following inequality holds:
\begin{equation}
    \label{eq:eml}
    \left |\prob{S}\prob{T} - \prob{E_{H'}(S,T)} \right| \leq \varepsilon \sqrt{\prob{S}\prob{T}},
\end{equation}
where all probabilities in \pref{eq:eml} are with respect to the edge weights of \(H'\).
\end{claim}

\pref{claim:eml-holds-whp} implies \pref{lem:second-step} via the inverse expander mixing lemma by \cite{BiluL2006}.
\restatetheorem{thm:converse-bipartite-eml}
\end{proof}

We continue to prove \pref{claim:eml-holds-whp}. To do so, we introduce notation that differs between the different probability spaces we refer to. For any graph \(G = (V(G),E(G))\) (be it either \(H\) or one of the random \(H'\)) and a set \(X \subseteq E(G)\) we denote its edge weights by \(\nu_G(X) = \sum_{e \in X}\nu_G(e)\). Likewise, for \(X \subseteq V(G)\) we denote by \(\nu(X)_G = \frac{1}{2}\sum_{v \in X}\sum_{u \in G}\nu_G(vu)\), the weight of \(X\) induced by the edge weights.

The notation \(\prob{A}\) denotes the probability that an event \(A\) occurs with respect to the random choice of \(H'\). With this notation, \eqref{eq:eml} translates to
\begin{equation} \label{eq:eml-with-nu}
        \left |\nu_{H'}(S)\nu_{H'}(T) - \nu_{H'}(E_{H'}(S,T)) \right| \leq \varepsilon \sqrt{\nu_{H'}(S)\nu_{H'}(T)}.
\end{equation}

\begin{proof}[Proof of \pref{claim:eml-holds-whp}]
For \(S \subseteq A, T \subseteq B\), denote by \(B_{S,T}\) the event where \eqref{eq:eml-with-nu} doesn't hold for the pair \(S,T\), and denote by \(B = \bigcup_{S \subseteq A, T \subseteq B} B_{S,T}\). We need to show that \(\prob{B} \leq \badprob\).

Denote by \(R\) the ``bad'' event where
\begin{enumerate}
    \item Either \(\nu_H(E)<(1-0.25\varepsilon)p\) or
    \item There exists some set \(S \subseteq A \cup B\) so that one of the the following inequalities hold:
\begin{gather} \label{eq:basic-prob-chernoff-2}
     \nu_{H'}(S) \leq (1+0.25\varepsilon)\nu_H(S) \text{ or } \nu_{H}(S) \leq (1+0.25\varepsilon)\nu_{H'}(S).
\end{gather}
\end{enumerate}
By a union bound it holds that
\[\prob{B} \leq \prob{R} + \sum_{S,T} \prob{B_{S,T} \ve \neg R}.\]

This claim bounds \(\prob{R}\).
\begin{claim}  \label{claim:vertex-probabilities-are-similar}
\(R\) occurs with probability at-most \(\badprobR\).
\end{claim}

These claims bounds the probability of \(\prob{B_{S,T} \ve \neg R}\).
\begin{claim} \label{claim:b-s-t-is-0-when-s-t-are-small}
Let \(S \subseteq A,T \subseteq B\) so that \(min(\nu_{H'}(S),\nu_{H'}(T)) \leq \varepsilon^3 p^3\). Then
\(\prob{B_{S,T} \ve \neg R} =0.\)
\end{claim}

\begin{claim} \label{claim:b-s-t-is-small}
Let \(S \subseteq A,T \subseteq B\) so that \(\nu_{H'}(S),\nu_{H'}(T) \geq \varepsilon^3 p^3\). Then
\(\prob{B_{S,T} \ve \neg R} \leq e^{-\mu D(|S|+|T|)}\)\footnote{recall that \(\mu = \badmu\).}.
\end{claim}

We combine the above for our bound.
\[\prob{B} \leq \prob{R} + \sum_{S,T} \prob{B_{S,T} \ve \neg R} \leq \]
\[\badprobR + \sum_{S,T : \;\nu_{H'}(S), \nu_{H'}(T) \geq \varepsilon^2p }e^{-\mu D(|S|+|T|)}  \leq \]
\[\badprobR + \left (\sum_{j=1}^{n} \binom{n}{j}e^{-\mu D j} \right )^2 \leq \]
By using the crude estimate \(\binom{n}{j} \leq n^j = e^{j \log n}\) we have
\[\badprobR + \left (\sum_{j=1}^{n} e^{j \log n-\mu D j} \right )^2 \leq \]
\[\badprobR + \frac{e^{2 (\log n - \mu D)}}{(1-e^{\log n - \mu D})^2} \leq \badprob \]
as required. Note that in the last inequality we used the fact that when \(\mu D > 2 \log n +5\) to bound the denominator by a constant close to \(1\) (say \(0.99\)).
\end{proof}

It remains to prove \pref{claim:vertex-probabilities-are-similar}, \pref{claim:b-s-t-is-0-when-s-t-are-small} and \pref{claim:b-s-t-is-small}.
To do so, we show that the probabilities of of sets in a typical \(H'\) are similar those of \(H\).
\begin{claim} \label{claim:second-step-set-probabilities}
Let \(\eta > 0\). For every set of edges \(F \subseteq E\), denote the number of remaining edges \(H'\) by \(F' = E' \cap F\).
\begin{equation}\label{eq:chernoff-basic}
    \Prob[H']{\abs{\nu_{H}(F')-p \nu_H(F)} \geq \eta p\nu_H(F)} < 2e^{-2p^2\eta^2r^{-2}|F|}.
\end{equation}
Furthermore, with probability at least \(1-4e^{-0.07 p\eta^2r^{-2}|F|}\) it holds that
\begin{equation} \label{eq:basic-prob-chernoff}
    (1-\eta)\nu_{H}(F) \leq \nu_{H'}(F) \leq (1+\eta)\nu_{H}(F).
\end{equation}
\end{claim}

\begin{proof}[Proof sketch of \pref{claim:second-step-set-probabilities}]
The proof of the first item is similar to the proof For the first item in \pref{claim:common-events} and is omitted.

As for the second item, note that \(\nu_{H'}(F') = \frac{\nu_H(F')}{\nu_H(E')}\). Thus when \eqref{eq:chernoff-basic} is true for the entire edge set of \(H'\) and for \(F\) with \(\eta' = \frac{\eta}{4}\), then \eqref{eq:basic-prob-chernoff} follows for \(\eta\). The exact calculations are omitted.
\end{proof}

\begin{proof}[Proof of \pref{claim:vertex-probabilities-are-similar}]
By \pref{claim:second-step-set-probabilities}
\[ \prob{|E'|<(1-\varepsilon)p|E|} \leq e^{-0.3p\varepsilon^2r^{-2}|E|}.\]
Thus the first item in \(R\)'s description doesn't occur except with probability \(e^{-0.3p\varepsilon^2r^{-2}|E|}.\)
We bound the probability that the second item occurs. The probability of a vertex \(v \in V\) is half the probability of \(F = \sett{e \in E}{v \in e}\). By assumption, \(|F| \geq D\). Thus for any \(v \in V\), by \pref{claim:second-step-set-probabilities}, \eqref{eq:basic-prob-chernoff-2} holds for \(S = \set{v}\) with probability \( \geq 1-2e^{-0.3p\varepsilon^2 D}\). By a union bound \eqref{eq:basic-prob-chernoff-2} holds for all \(S=\set{v} \subseteq V\) simultaneously with probability \(1-2ne^{-0.01\varepsilon p^{-3}r^{-2}D}\). When this is true, then \eqref{eq:basic-prob-chernoff-2} holds for all \(S \subseteq V\) since \(\nu_{H'}(S) = \sum_{v \in S}\nu(\set{v})_{H'}\).

A union bound over the first item and the second item (for all singletons) completes the proof.
\end{proof}

\begin{proof}[Proof of \pref{claim:b-s-t-is-0-when-s-t-are-small}]
Fix \(S\subseteq A, T \subseteq B\) so that \(\min (\nu_{H'}(S),\nu_{H'}(T)) \leq \varepsilon^3 p^3\). If
\[|\nu_{H'}(E_{H'}(S,T)) - \nu_{H'}(S) \nu_{H'}(T)|= \nu_{H'}(S) \nu_{H'}(T) - \nu_{H'}(E(S,T))\]
then \(\sqrt{ \nu_{H'}(S) \nu_{H'}(T)} \leq \varepsilon^{1.5} p^{1.5} \leq \varepsilon\), it holds that
    \[\nu_{H'}(S) \nu_{H'}(T) - \nu_{H'}(E(S,T)) \leq \nu_{H'}(S) \nu_{H'}(T) \leq \varepsilon \sqrt{\nu_{H'}(S) \nu_{H'}(T)}.\]
Thus assume that
\begin{gather} \label{eq:small-S-T-case-2}
    |\nu_{H'}(E_{H'}(S,T)) - \nu_{H'}(S) \nu_{H'}(T)| = \nu_{H'}(E_{H'}(S,T)) - \nu_{H'}(S) \nu_{H'}(T) \leq \nu_{H'}(E_{H'}(S,T)).
\end{gather}
The event \(\neg R\) implies that \(\nu_H(E') \geq (1-0.25\varepsilon)p\). Thus 
\[ \nu_{H'}(E_{H'}(S,T)) = \frac{\nu_H(E_{H'}(S,T))}{\nu_H(E')} \leq \frac{\nu_H(E_{H'}(S,T))}{(1-0.25\varepsilon)p} \leq \frac{(1+\varepsilon)\nu_H(E_H(S,T))}{p}.\]
This upper bound is tight when \emph{all} edges of \(E_{H}(S,T)\) are sampled into \(H'\).
The original graph \(H\) is a \(\lambda\)-spectral expander, so by the expander mixing lemma \(\nu_{H}(E_{H}(S,T))\) is also small:
\begin{gather} \label{eq:eml-for-small-S-T-in-H}
\nu_{H}(E_{H}(S,T)) \leq \nu_H(S) \nu_H(T) + \lambda \sqrt{\nu_H(S) \nu_H(T)}.
\end{gather}
When \(\neg R\) occurs, then \(\nu_H(S) \leq (1+\varepsilon)\nu_{H'}(S)\) and \(\nu_H(T) \leq (1+\varepsilon)\nu_{H'}(T)\). So we can upper bound \eqref{eq:eml-for-small-S-T-in-H} by 
\[(1+\varepsilon)^{2} \nu_{H'}(S) \nu_{H'}(T) + \lambda(1+\varepsilon) \sqrt{\nu_{H'}(S) \nu_{H'}(T)} \leq (\varepsilon^{1.5} p^{1.5}+\lambda)(1+\varepsilon)^{2} \sqrt{\nu_{H'}(S) \nu_{H'}(T)}.\] 
The inequality is by the fact that \(\nu_{H'}(S) \cdot\nu_{H'}(T) \leq \varepsilon^2 p\).
Plugging this in \eqref{eq:small-S-T-case-2} we get
\[|\nu_{H'}(E(S,T)) - \nu_{H'}(S) \nu_{H'}(T)| \leq \nu_{H'}((E_{H'}(S,T)) \leq \]
\[(1+\varepsilon)^3 (\lambda p^{-1}+\varepsilon^{1.5} p^{0.5})\sqrt{\nu_{H'}(S) \nu_{H'}(T)} \leq \varepsilon \sqrt{\nu_{H'}(S) \nu_{H'}(T)},\]
since by the assumptions on \(\lambda\), \((1+\varepsilon)^3 (\lambda p^{-1}+\varepsilon^{1.5}p^{0.5}) \leq \varepsilon\).
\end{proof}

\begin{proof}[Proof of \pref{claim:b-s-t-is-small}]
Fix \(S\subseteq A, T \subseteq B\) so that \(\nu_{H'}(S), \nu_{H'}(T) \geq \varepsilon^3 p^3\). We will prove below that
\begin{gather} \label{eq:number-of-edges-proportional-to-D-times-sizes}
|E_H(S,T)| \geq \mu D (|S|+|T|)
\end{gather}
for \(\mu\) in the Theorem statement \(\mu: = \badmu\). For now let us assume \eqref{eq:number-of-edges-proportional-to-D-times-sizes}. 
\pref{claim:second-step-set-probabilities} implies that the event that 
\begin{equation} \label{eq:good-inequality-for-s-t}
|\nu_{H'}(E_{H'}(S,T))-\nu_{H}(E_{H}(S,T)) | \leq 0.25\varepsilon \nu_{H}(E_{H}(S,T)),
\end{equation}
occurs with probability \(\geq 1 - 2e^{-0.3p\varepsilon^2r^{-2} |E_H(S,T)|}\). By \eqref{eq:number-of-edges-proportional-to-D-times-sizes} this is \(\geq 1 - 2e^{-0.3p\varepsilon^2r^{-2}\mu D (|S|+|T|)}\).

When \eqref{eq:good-inequality-for-s-t} occurs, then 
\begin{gather}
|\nu_{H'}(E_{H'}(S,T)) - \nu_{H'}(S)\nu_{H'}(T)| \leq \\
    |\nu_{H'}(E_{H'}(S,T))-\nu_{H}(E_{H}(S,T)) | + \label{eq:final-bound-1}\\
    |\nu_{H}(E_{H}(S,T)) - \nu_{H}(S)\nu_{H}(T)| + \label{eq:final-bound-2}\\
    |\nu_{H}(S)\nu_{H}(T) - \nu_{H'}(S)\nu_{H'}(T)| \label{eq:final-bound-3}
\end{gather}
Before bounding the terms \eqref{eq:final-bound-1}, \eqref{eq:final-bound-2}, \eqref{eq:final-bound-3}, we say informally that \eqref{eq:final-bound-2} is bounded by the expander mixing lemma in \(H\), and \eqref{eq:final-bound-1}, \eqref{eq:final-bound-3} are bounded by the fact that when \(\neg R\) occurs, the probabilities of verticies and edge sets between verticies are similar in \(H,H'\).
\begin{enumerate}
    \item By \eqref{eq:good-inequality-for-s-t}, \eqref{eq:final-bound-1} is bounded by \(0.25\varepsilon \nu_{H}(E_{H}(S,T))\). As \(\nu_{H}(E_{H}(S,T))\) is bounded both by \(\nu_{H}(S)\) and \(\nu_{H}(T)\), it is also bounded by their geometric mean. When \(\neg R\) occurs, this implies that \eqref{eq:final-bound-1} is bounded by
    \[0.25\varepsilon \sqrt{\nu_{H}(S)\nu_{H}(T)} \leq 0.25\varepsilon (1+0.25\varepsilon) \sqrt{\nu_{H'}(S)\nu_{H'}(T)}\leq 0.4\varepsilon \sqrt{\nu_{H'}(S)\nu_{H'}(T)}.\]
    \item By the Expander Mixing Lemma in \(H\), \eqref{eq:final-bound-2} is upper bounded by \(\lambda \sqrt{\nu_{H}(S)\nu_{H}(T)}\). When \(\neg R\) occurs, this is further bounded by \(\lambda(1+0.25\varepsilon) \sqrt{\nu_{H'}(S)\nu_{H'}(T)}\). By our assumption on \(\lambda\), this is \(\leq 0.09\varepsilon \sqrt{\nu_{H'}(S)\nu_{H'}(T)}\).
    \item When \(\neg R\) occurs, then \eqref{eq:final-bound-3} is bounded by \(((1+0.25\varepsilon)^2 - 1)\nu_{H'}(S)\nu_{H'}(T)\). This is further bounded by \(0.51\varepsilon \sqrt{\nu_{H'}(S)\nu_{H'}(T)}\).
\end{enumerate}
In conclusion, we showed that with probability \( 1 - 2e^{-0.3p\varepsilon^2r^{-2}\mu D (|S|+|T|)}\), 
\[|\nu_{H'}(E_{H'}(S,T)) - \nu_{H'}(S)\nu_{H'}(T)| \leq \varepsilon \sqrt{\nu_{H'}(S)\nu_{H'}(T)}\]
as required.
\medskip

It remains to show \eqref{eq:number-of-edges-proportional-to-D-times-sizes}.
By the expander mixing lemma for \(H\)
\[\nu_{H}(E_{H}(S,T)) \geq \nu_H(S) \nu_H(T) - \lambda \sqrt{\nu_H(S) \nu_H(T)} \geq\]
\[ \sqrt{\nu_H(S) \nu_H(T)} (\sqrt{\nu_H(S) \nu_H(T)} - \lambda) \geq \]
\begin{equation}\label{eq:b-s-t-is-small}
    \sqrt{\nu_H(S) \nu_H(T)}((1-0.25\varepsilon) \varepsilon^3 p^3-\lambda).
\end{equation}
The last inequality is due to the following: when \(\neg R\) occurs and \(\nu_{H'}(S), \nu_{H'}(T) \geq \varepsilon^3 p^3\) then \(\nu_{H}(S), \nu_{H}(T) \geq (1-0.25\varepsilon) \varepsilon^3 p^3\). Thus we replace their geometric mean by \((1-0.25\varepsilon)\varepsilon^3p^3\).

As \[\sqrt{\nu_{H}(S) \nu_{H}(T)}\geq \sqrt{(1-0.25\varepsilon) \varepsilon^3 p^3 \nu_{H}(T)}\geq (1-0.25\varepsilon) \varepsilon^{3/2} p^{3/2}\nu_{H}(T)\]
and 
\[\sqrt{\nu_{H}(S) \nu_{H}(T)}\geq \sqrt{(1-0.25\varepsilon) \varepsilon^3 p^3 \nu_{H}(S)}\geq (1-0.25\varepsilon) \varepsilon^{3/2} p^{3/2}\nu_{H}(S)\]
then by taking average, we have that 
\[\sqrt{\nu_{H}(S) \nu_{H}(T)} \geq 0.5(1-0.25\varepsilon) \varepsilon^{3/2} p^{3/2} (\nu_{H}(S) + \nu_{H}(T)).\]
Plugging this in \eqref{eq:b-s-t-is-small} we have that
\[\nu_{H}(E_{H}(S,T)) \geq 0.5\varepsilon^{3/2} p^{3/2}(1-0.25\varepsilon)((1-0.25\varepsilon)\varepsilon^3 p^3-\lambda)( \nu_{H}(S) + \nu_{H}(T) ).\]
Every edge in \(H\) has weight at most \(\frac{r}{|E|}\), we have that \(|E_{H}(S,T)| \cdot \frac{r}{|E|} \geq \nu_{H}(E_{H}(S,T))\). Furthermore, since every edge in \(H\) has weight at least \(\frac{1}{r|E|}\), then \(\nu_H(S) \geq \frac{D}{r|E|}|S|\) (and the same for \(T\)). Thus concludes that
\[|E(S,T)| \geq \mu D (|S|+|T|).\]
Where \(\mu = \badmu\).

\end{proof}
\section{Bilu-Linial's Converse Expander Mixing Lemma} \label{sec:bilu-linial}
In this appendix we give a variation of the converse expander mixing lemma given in \cite{BiluL2006}, that suits our purposes in \pref{sec:expander-sparsification}.
\begin{theorem}[\cite{BiluL2006}, Lemma 3.3] \torestate{\label{thm:bilu-linial-variation}
Let \(X\) be a finite set and \(\mu:X \to [0,1]\) be a probability distribution on \(X\). Let \(V = L_2(X,\mu)\) be its inner product space. Let \(A:V\to V\) be a self adjoint operator that satisfies:
\begin{enumerate}
    \item Non-laziness: for every \(x\in X\), \(\iprod{A\one_x,\one_x} = 0\).
    \item For every \(u \in \set{-1,0,1}^X\), and  \(v_1,v_2,...,v_m \subseteq \set{-1,0,1}^X\) whose supports are pairwise disjoint, it holds that
    \[ \sum_{j=1}^m \abs{\iprod{u,Av_j}} \leq d \iprod{u,u}.\]
    \item For every \(S,T \subseteq X\) with disjoint support,
    \[ \iprod{A\one_S,\one_T} \leq \alpha \norm{\one_S} \norm{\one_T}.\]
    Then \(\norm{A} \leq 130(\alpha + \alpha \log_2(d/\alpha))\).
\end{enumerate}
}
\end{theorem}
This version of Lemma 3.3 in \cite{BiluL2006}, deals with self-adjoint operators rather than real symmetric matrices which were in the original version.
\restatecorollary{thm:converse-bipartite-eml}.

\begin{proof}[Proof of \pref{thm:converse-bipartite-eml}]
Let \(J:L_2(H')\to L_2(H')\) be the operator of the complete (weighted) bipartite graph
\[J(f)(v) = \begin{cases}
\Ex[u \in X]{f(u)} & v \in Y \\
\Ex[u \in Y]{f(u)} & v \in X 
\end{cases}.
 \]
Let \(B = A-J\) and let \(v_0 = \one_{X\cup Y}\), \(v_n = \one_{X} - \one_{Y}\) the trivial eigenvectors of \(A\). One can verify that \(Bv_0 = Bv_n = 0\). Furthermore, for every vector \(v\) orthogonal to \(v_0,v_n\), \(Bv=Av\). Thus, \(\lambda(A) = \lambda(B) = ||B||\).

We note that \(B\) is self adjoint, and for every \(u \in A \cup B\), \(\iprod{B\one_u,\one_u} = 0\). Moreover, for every \(u \in \set{0,1}^{A \cup B}\) and pairwise disjoint \(v_j \in \set{0,1}^{A \cup B}\) we have that
\[ \sum_{j} \abs{\iprod{Bu,v_j}} \leq \sum_{j} \abs{\iprod{Au,v_j}} + \abs{\iprod{Ju,v_j}} \leq 2\sum_j \iprod{u,v_j} = 2\norm{u}^2.\]

In order to use \pref{thm:bilu-linial-variation}, we bound \(\abs{\iprod{B \one_S, \one_T}}\) for every \(S,T \subseteq A \cup B\) that are disjoint. Indeed, let \(S = S_1 \cup S_2, T = T_1 \cup T_2\) so that \(S_1,T_1 \subseteq A, S_2,T_2 \subseteq B\). Then
\[ \abs{\iprod{B \one_S, \one_T}} = \abs{\iprod{B \one_{S_1}, \one_{T_2}} + \iprod{B \one_{S_2}, \one_{T_1}}} \leq \]
\[ \abs{\iprod{B \one_{S_1}, \one_{T_2}}} + \abs{\iprod{B \one_{S_2}, \one_{T_1}}}.\]
Note that \(\iprod{B \one_{S_1}, \one_{T_2}}  = \prob{E(S_i,T_j)} - \prob{S_i}\prob{T_j}\) hence we get the bound
\[\abs{\iprod{B \one_S, \one_T}} \leq \alpha \sqrt{\prob{S_1}\prob{T_2}} + \sqrt{\prob{S_2}\prob{T_1}} \leq 2\alpha \sqrt{\prob{S}\prob{T}} = \]
\[ 2\alpha \norm{\one_S}\norm{\one_T}.\]
Hence, by \pref{thm:bilu-linial-variation}, we have that
\[ \lambda(A) = ||B|| \leq 260\alpha(1+\log_2(2/\alpha)).\]

\end{proof}

\subsection{Proof of \pref{thm:bilu-linial-variation}}
The proof of \pref{thm:bilu-linial-variation} is the same as the proof of the original version in \cite{BiluL2006}. The only change is that instead of taking matrix multiplication we take inner product. We give it here in full just to stay self contained.
\begin{proof}[Proof of \pref{thm:bilu-linial-variation}]
The operator \(A\) is self adjoint, our goal is to show that
\begin{equation} \label{eq:rayleigh-quotient}
    max_{0 \ne v \in L_2(X,\mu)} \frac{\iprod{\pm Av,v}}{\iprod{v,v}} \leq 130(\alpha + \alpha \log_2(d/\alpha)).
\end{equation}
As \(A\) and \(-A\) both uphold the assumptions of the Theorem, we focus on the inequality with \(A\). We begin by asserting that the third item's inequality also holds when the support is not disjoint, and when instead of vectors in \(\set{0,1}^X\) we take vectors in \(\set{-1,0,1}^X\).

\begin{claim} \label{claim:positive-negative-labels}
Let \(v,u \in \set{-1,0,1}^X\). Then
\[\iprod{Av, u} \leq 64\alpha \norm{v}\norm{u}.\]
\end{claim}

Next we show that we can approximate the maximum in \eqref{eq:rayleigh-quotient} by taking a maximum only on \(v\)'s whose values are either \(0\) or \(\pm 2^{-i}\)'s.
\begin{claim} \label{claim:powers-of-two-approximation}
Let \(v \in L_2(X,\mu)\). There exists some \(w \in L_2(X,\mu)\) so that 
\begin{enumerate}
    \item \(\frac{\iprod{Av,v}}{\iprod{v,v}} \leq 2\frac{\iprod{A w,w}}{\iprod{w,w}}\).
    \item For all \(x \in X\), \(w(x) = sign(v(x)) 2^{-i}\) for some negative integer \(i <0\).
\end{enumerate}

By \pref{claim:powers-of-two-approximation}, we need to show that for every \(w \in L_2(X,\mu)\) that takes values that are negative powers on \(2\),
\[\frac{\iprod{Aw,w}}{\iprod{w,w}} \leq 64 (\alpha +\alpha(\log_2(d/\alpha)).\]
Indeed, we partition the support of \(w = \sum_{i>0}2^{-i} w_i\) where \(w_i = \pm 1\) and the support of the \(w_i\)'s are pairwise disjoint.

By \pref{claim:positive-negative-labels} we have that
\begin{equation} \label{eq:first-bound}
    \iprod{w_i,Aw_j} \leq 64\alpha \norm{w_i}\norm{w_j}
\end{equation}
and by the second assumption in the theorem we have that
\begin{equation} \label{eq:second-bound}
    \sum_j \abs{\iprod{w_i,Aw_j}} \leq d \iprod{w_i,w_i}
\end{equation}

Furthermore, we have that
\begin{equation} \label{eq:rayleigh-quotient-2}
\frac{\iprod{Aw,w}}{\iprod{w,w}} \leq \frac{\sum_{i,j} 2^{-i-j} \abs{\iprod{w_i,w_j}}}{\sum_i 2^{-2i}\iprod{w_i,w_i}}.
\end{equation}

Denote by \(\gamma = \log(d/\alpha)\), \(q_i = 2^{-2i}\iprod{w_i,w_i}\) and \(Q = \sum_i q_i\). Add up inequalities \eqref{eq:first-bound} and \eqref{eq:second-bound} as follows. For \(i=j\) we multiply \eqref{eq:first-bound} by \(2^{-2i}\). when \(i<j\leq i+\gamma\) multiply it by \(2^{-(i+j)+1}\). Multiply \eqref{eq:second-bound} by \(2^{-(2i+\gamma)}\).
We get that:
\begin{equation} \label{eq:added-up-bounds}
    \sum_{i} 2^{-2i} \abs{\iprod{w_i, Aw_i}} + \sum_i \sum_{i<j\leq i+\gamma} 2^{-(i+j)+1} \abs{\iprod{w_i,Aw_j}} + \sum_i 2^{-2i+\gamma}\sum_j \abs{\iprod{w_i,A w_j}} \leq
\end{equation}
\[64\alpha \sum_i q_i + 64\alpha \sum_i \sum_{i<j\leq i+\gamma} 2\sqrt{q_i q_j} + \alpha \sum_i q_i \leq\]
\[ 67\alpha Q + 64\alpha \sum_i \sum_{i<j\leq i+\gamma} q_i +q_j.\]
Every member \(q_j\) appears in \(\sum_i \sum_{i<j\leq i+\gamma}q_i + q_j\) at most \(\lceil\gamma \rceil+1 \leq \gamma+2\) times, this is at most
\[ 130\alpha + \alpha \log_2(d/\alpha) Q.\]
As \(Q\) is the denominator in \eqref{eq:rayleigh-quotient-2}, it is enough to show that the numerator in \eqref{eq:rayleigh-quotient-2} is at most \eqref{eq:added-up-bounds}. Indeed, we compare the coefficients of \(\abs{\iprod{w_i,Aw_j}}\) in both expressions (as \(A\) is self adjoint, it is enough to consider \(i \leq j\). 

\begin{itemize}
    \item For \(i=j\) the coefficient in \eqref{eq:added-up-bounds} is \(2^{-2i}\) and \(2^{-2i} + 2^{-(2i+\gamma)}\) in \eqref{eq:rayleigh-quotient-2}.
    \item For \(i<j\leq i+\gamma\) the coefficient in \eqref{eq:added-up-bounds} is \(2^{-(i+j)+1}\) and \(2^{-(i+j)+1} + 2^{-(2i+\gamma)} + 2^{-(2j+\gamma)}\) in \eqref{eq:rayleigh-quotient-2}.
    \item For \(i+\gamma < j\) the coefficient in \eqref{eq:added-up-bounds} is \(2^{-(i+j)+1}\). In \eqref{eq:rayleigh-quotient-2} it is 
    \(2^{-(2i+\gamma)} + 2^{-(2j+\gamma)} \geq 2^{-(2i+\gamma)} \geq 2^{-(i+j)+1}\).
\end{itemize}
The claim follows.
\end{claim}

\end{proof}
\subsection{Proof of various claims}

\begin{proof}[Proof of \pref{claim:positive-negative-labels}]
We first show that for every \(u \in \set{0,1}^X\), \(\iprod{Au,u} \leq 2\alpha \iprod{u,u}\). Indeed, Let \(S = supp(u)\) be of size \(n\). We have that
\[ \abs{\sum_{T \subseteq S} \iprod{A \one_T, \one_{S \setminus T}}} = 2^{n-2} \abs{\sum_{x,y \in S, x \ne y}\iprod{A \one_x,\one_y}} = \]
\begin{equation} \label{eq:sum-all-sets}
2^{n-2} \abs{\iprod{Au,u}}.
\end{equation}
The last inequality uses the assumption that \(\iprod{A \one_x, \one_x} = 0\). On the other hand,

    \[\abs{\sum_{T \subseteq S} \iprod{A \one_T, \one_{S \setminus T}}} \leq \sum_{T \subseteq S} \abs{\iprod{A \one_T, \one_{S \setminus T}}}\]

and by the third assumption in the theorem, this is at most
\[ \alpha \sum_{T \subseteq S} \norm{\one_T}\norm{\one_{S \setminus T}}.\]
By AM-GM this is at most
\begin{equation}\label{eq:other-hand}
\alpha \sum_{T \subseteq S}\frac{1}{2}\norm{u}^2 = 2^{n-1}\alpha \norm{u}^2.
\end{equation}
By combining \eqref{eq:sum-all-sets} and \eqref{eq:other-hand} we get that \(\iprod{Au,u} \leq 2\alpha \iprod{u,u}\).

Next, for arbitrary \(u, v \in \set{0,1}^X\), we write \(u = w+u', v=w+v'\) where the support of \(u',v',w\) is disjoint. We have that
\[ \iprod{Au,v} = \iprod{Aw,w} + \iprod{Au',w} + \iprod{Aw,v'} + \iprod{Aw,v'} \leq\]
\[ 2\alpha (\norm{w}\norm{w} + \norm{u'}\norm{w} + \norm{w}\norm{v'} + \norm{w}\norm{v'}) = \]
\[2\alpha (\norm{w}+\norm{u'})(\norm{w}+\norm{v'}) \leq 8 \alpha \norm{v}\norm{u}.\]

Finally, let \(v,u \in \set{-1,0,1}^X\). We write \(u = u^+ - u^-\) and \(v= v^+ - v^-\) where \(u^+,u^-,v^+,v^- \in \set{0,1}^X\). We again have that
\[ \abs{\iprod{Au,v}} \leq \abs{\iprod{Au^+,v^+}} + \abs{\iprod{Au^+,v^-}} + \abs{\iprod{Au^-,v^+}} + \abs{\iprod{Au^-,v^-}} \leq \]
\[ \leq 8\alpha (\norm{u^+}+\norm{u^-})(\norm{v^+}+\norm{v^-}) \leq\]
\[ 32\alpha \norm{u}\norm{v}.\]
\end{proof}

\begin{proof}[Proof of \pref{claim:powers-of-two-approximation}]
\eqref{eq:rayleigh-quotient} is homogenuous, so we assume without loss of generality that \(max_{x \in X} \abs{v(x)} \leq 2^{-2}\). We denote \(v(x) = sign(v(x))(1+\delta_x)2^{i_x}\) for \(i_x\) a negative integer and \(\delta_x \in [0,1)\). We sample \(w\) as follows. For every \(x \in X\) independently we sample \(w(x) = sign(v(x))2^{i_x+1}\) with probability \(\delta_x\) and \(w(x) = sign(v(x)) 2^{i_x}\) with probability \(1-\delta_x\). Obviously, the second item in the claim holds, namely, \(w(x)\) is equal to \(sign(v(x))\) times some negative power of \(2\). Furthermore, it is clear that \(\iprod{w,w} \leq 2\iprod{x,x}\).
Finally, note that by linearity of expectation, \(\Ex[w]{\iprod{w,Aw}} = \sum_{x, y}\Ex[w]{w(x)w(y)}\iprod{A \one_x, \one_y}\). As \(A\) is non-lazy, and \(w(x),w(y)\) are pairwise independent, this is equal to \(\sum_{x \ne y}\Ex[w]{w(x)}\Ex[w]{w(y)}\iprod{A \one_x, \one_y} = \iprod{Av,v}\). Thus there exists some \(w\) so that its inner product is greater or equal the expectation, \(\iprod{Aw,w} \geq \iprod{Av,v}\). For that \(w\) the first item holds, namely, \(\frac{\iprod{Av,v}}{\iprod{v,v}} \leq 2\frac{\iprod{A w,w}}{\iprod{w,w}}\).
\end{proof}

\end{document}